\documentclass[a4paper,10pt]{article}
\usepackage[english]{babel}
\usepackage{a4wide}
\usepackage[utf8]{inputenc}
\usepackage{amssymb,amsmath,amscd,amsfonts,amsthm,bbm,mathrsfs}
\usepackage{booktabs}
\usepackage[colorlinks=true, linkcolor=blue, citecolor=blue]{hyperref}
\usepackage{mathabx}
\usepackage{graphicx}
\usepackage{color}
\usepackage{accents}
\usepackage{subcaption}
\usepackage{enumerate,enumitem}
\usepackage{tensor} 
\usepackage{authblk}

\DeclareUnicodeCharacter{00A0}{~}

\theoremstyle{definition}
\newtheorem{theorem}{Theorem}
\newtheorem*{theorem*}{Statement}
\newtheorem{lemma}[theorem]{Lemma}
\newtheorem{corollary}[theorem]{Corollary}
\newtheorem{proposition}[theorem]{Proposition}

\newtheorem{remark}{Remark}

\newtheorem*{condition*}{Condition}
\newtheorem{assumption}{Assumption}
\newtheorem{hypothesis}{Hypothesis}

\DeclareMathOperator{\var}{\mathbb Var}
\DeclareMathOperator{\cov}{\mathbb Cov}

\DeclareMathOperator{\LCP}{LCP}
\DeclareMathOperator{\PL}{PL}

\newcommand{\rownorm}{|\!|\!|}

\newcommand{\1}{\mathbbm 1}
\newcommand{\T}{\top}

\newcommand{\PP}{{{\mathbb P}}} 
\newcommand{\EE}{{{\mathbb E}}} 
\newcommand{\NN}{{{\mathbb N}}} 
\newcommand{\RR}{{{\mathbb R}}}

\newcommand{\mcF}{{\mathscr F}} 
\newcommand{\mcL}{{\mathscr L}}

\newcommand{\cA}{{\mathcal A}} 
\newcommand{\cB}{{\mathcal B}} 
 
\newcommand{\cD}{{\mathcal D}} 
\newcommand{\cE}{{\mathcal E}} 
 
\newcommand{\cG}{{\mathcal G}} 
 
\newcommand{\cI}{{\mathcal I}} 
 
\newcommand{\cK}{{\mathcal K}} 
\newcommand{\cL}{{\mathcal L}} 
 
\newcommand{\cN}{{\mathcal N}} 
 
\newcommand{\cP}{{\mathcal P}} 
\newcommand{\cR}{{\mathcal R}} 
\newcommand{\cQ}{{\mathcal Q}} 
\newcommand{\cS}{{\mathcal S}} 
\newcommand{\cT}{{\mathcal T}} 
\newcommand{\cV}{{\mathcal V}}

\newcommand{\sH}{{\mathsf H}} 

\newcommand{\tR}{{\widetilde R}} 
\newcommand{\tW}{{\widetilde W}} 
\newcommand{\tGamma}{{\widetilde \Gamma}} 
 
\newcommand{\bSigma}{{\bs\Sigma}} 
\newcommand{\tXi}{{\widetilde \Xi}} 

\newcommand{\chR}{{\widecheck R}} 
\newcommand{\chW}{{\widecheck W}}

\newcommand{\br}{\bs r} 
\newcommand{\bu}{\bs u}

\newcommand{\bx}{\bs x} 
\newcommand{\bcx}{\bs x} 
\newcommand{\cx}{\check x} 
\newcommand{\tx}{\tilde x} 

\newcommand{\bX}{\bs X} 
\newcommand{\tX}{{\widetilde X}} 
\newcommand{\chX}{{\widecheck X}} 
\newcommand{\tbX}{\bs{\widetilde X}} 
\newcommand{\cbX}{\bs X} 

\newcommand{\by}{\bs y} 
\newcommand{\bcy}{\bs{\check y}} 

\newcommand{\bZ}{\bs Z} 
\newcommand{\tZ}{{\widetilde Z}} 
\newcommand{\cZ}{{\widecheck Z}} 
\newcommand{\cbZ}{\bs{\widecheck Z}} 
\newcommand{\tbZ}{\bs{\widetilde Z}}

\newcommand{\barW}{{\overline W}} 
\newcommand{\ucU}{\overline{\mathcal U}}

\newcommand{\ps}[1]{\langle #1 \rangle}

\newcommand{\balpha}{\bs \alpha}

\newcommand{\ua}{\underline{a}} 
\newcommand{\ub}{\underline{b}} 
\newcommand{\ud}{\underline{d}} 
\newcommand{\uq}{\underline{q}} 
\newcommand{\ux}{x} 
\newcommand{\ueta}{\eta} 
\newcommand{\uxi}{\underline{\xi}}

\newcommand{\bG}{\bs G}

\newcommand{\bS}{\bs S} 
\newcommand{\bT}{\bs T} 
\newcommand{\bU}{\bs U} 
\newcommand{\bV}{\bs V} 
\newcommand{\bW}{\bs W}

\newcommand{\bm}{{\bs m}} 

\newcommand{\bz}{\bs z} 
\newcommand{\tbz}{\tilde{\bs z}} 

\newcommand{\tmax}{{t_{\max}}} 
\newcommand{\Cmom}{C_{\text{mom}}}  

\newcommand{\lessprobEw}{\stackrel{\mathcal{P}}{\leq}_{\cE_W}} 
\newcommand{\lessprob}{\stackrel{\mathcal{P}}{\leq}} 

\newcommand{\bs}{\boldsymbol}
\DeclareMathOperator*{\diag}{diag}
\newcommand{\eqlaw}{\stackrel{\cL}{=}}  
\newcommand{\toprobashortshort}{\stackrel{\mathcal{P}}{\to}}
\newcommand{\toprobashort}{\,\stackrel{\mathcal{P}}{\to}\,}
\newcommand{\toprobalong}{\xrightarrow[n\to\infty]{\mathcal P}}

\newcommand{\tolong}{\xrightarrow[n\to\infty]{}}

\begin{document}

\title{Approximate Message Passing for sparse matrices \\ with application 
to the equilibria of large ecological Lotka-Volterra systems} 

\author{Walid Hachem} 
\affil{CNRS, Laboratoire d'informatique Gaspard Monge (LIGM / UMR 8049),  
Universit\'e Gustave Eiffel, ESIEE Paris, France}  

\date{June 11, 2024 \ (revised version)} 

\maketitle

\begin{abstract} 
This paper is divided into two parts.  The first part is devoted to the study
of a class of Approximate Message Passing (AMP) algorithms which are widely
used in the fields of statistical physics, machine learning, or communication
theory. The AMP algorithms studied in this part are those where the measurement
matrix has independent elements, up to the symmetry constraint when this matrix
is symmetric, with a variance profile that can be sparse. The AMP problem is
solved by adapting the approach of Bayati, Lelarge, and Montanari (2015) to
this matrix model. \\ 
The Lotka-Volterra (LV) model is the standard model for studying the dynamical
behavior of large dimensional ecological food chains.  The second part of this
paper is focused on the study of the statistical distribution of the globally
stable equilibrium vector of a LV system in the situation where the random
symmetric interaction matrix among the living species is sparse, and in the
regime of large dimensions. This equilibirium vector is the solution of a
Linear Complementarity Problem, which distribution is shown to be characterized
through the AMP approach developed in the first part. In the large dimensional
regime, this distribution is close to a mixture of a large number of truncated
Gaussians. 
\end{abstract}


\section{Introduction}

An ecosystem can be seen as a multi-dimensional dynamical system that
represents the time evolution of the abundances of the interacting species. The
behavior of such systems is governed by the intrinsic population dynamics and
by the strengths of the interactions among these species. 
Given a system model, it is of interest to evaluate the distribution of the
species at the equilibrium when this equilibrium exists and is unique. The
present paper is motivated by this general problem. 

An archetypal model for an ecological dynamical system is provided by the so
called Lotka-Volterra (LV) multi-dimensional Ordinary Differential Equation
(ODE). The dynamics of a LV ODE with $n$ species take the form 
\[
\dot u(t) = u(t) \odot 
  \left( r + \left(\Sigma - I_n \right) u(t) \right), 
    \quad t \geq 0 ,   
\]
where the vector function $u : \RR_+ \to \RR_+^n$ represents the abundances of
the $n$ coexisting species after a proper normalization, $\odot$ is the
element-wise product, $r$ is the so-called vector of intrinsic growth rates of
the species, and $\Sigma$ is a $n\times n$ matrix which $(i,j)^{\text{th}}$
element reflects the interaction effect of Species $j$ on the growth of
Species~$i$ \cite{tak-livre96}. 

Denoting as $\| \cdot \|$ the spectral norm, it known that under a condition
such as $\| (\Sigma + \Sigma^\T)/2 \| < 1$, the ODE solution is well-defined,
and this ODE has an unique globally stable equilibrium $u_\star = \bigl[
u_{\star,i} \bigr]_{i=1}^n$ in the classical sense of the Lyapounov theory
\cite{tak-livre96,li-etal-09}. It is of interest to study the distribution of
the elements of this vector, which quantifies the relative abundances of the
species at the equilibrium.  It is useful to note here that $u_\star$
frequently lies at the boundary of the first quadrant of $\RR^n$, and
therefore, it is of particular interest to evaluate the proportion of surviving
species at these equilibria. 

Usually, the interaction matrix $\Sigma$ is difficult to measure or to
evaluate, and all the more so as the ecosystem's dimension $n$ gets large.  To
circumvent this difficulty, a whole line of research in theoretical ecology
considers that the matrix $\Sigma$ is a random matrix, and focuses on the
dynamics of the LV system in the regime where $n \to\infty$
\cite{all-tan-15,akj-etal-(arxiv)22}.  The idea is to predict some essential
aspects of the dynamical behavior of the ecosystem on the basis of a few
``phenomenological'' statistical features of the interaction matrix, rather
than on its fine structure.  The application of large random matrix tools to
the dynamical behavior study of ecological systems dates back to the work of
May~\cite{may-72}.  Among the most widely studied statistical models for the
interaction matrix from the standpoint of the large random matrix theory are
the Gaussian Orthogonal Ensemble (GOE) model, the Gaussian model with
i.i.d.~elements (sometimes called the Ginibre model), or the so-called
elliptical model, which can be seen as an ``interpolation'' these two with
possibly a non-zero mean \cite{all-tan-12}. 

Given a statistical model for $\Sigma$, and assuming the existence of $u_\star$
which is now a random vector, the problem amounts in our context to evaluating
the asymptotic behavior of the random probability measure 
\[
\mu^{u_\star} = \frac 1n \sum_{i\in[n]} \delta_{u_{\star,i}} 
\]
as $n\to\infty$.  In the literature, this has been mostly done with tools
issued from the physics. In~\cite{bun-17}, Bunin obtained the asymptotic
distribution of the equilibrium for the elliptical model by using 
the so-called dynamical cavity method. A similar result was obtained by
Galla in \cite{gal-18} with the help of generating functionals techniques. 
Older results in the same vein can be found in, \emph{e.g.},
\cite{opp-die-92,tok-04}. Heuristic evaluations of the asymptotic behavior of
$\mu^{u_\star}$ were proposed in \cite{cle-fer-naj-22,cle-naj-mas-(arxiv)22},
most generally in the elliptical non-centered case. 

In this paper, we consider a symmetric model for $\Sigma$, which can be
used to represent the competitive and the mutualistic interactions
\cite{bir-bun-cam-18,akj-etal-(arxiv)22}. In this framework, we assume that the
coefficients of $\Sigma$ are not necessarily Gaussian, are independent up to
the symmetry constraint, and are subjected to a variance profile that can be
sparse.  The two main features of our model are thus the inhomogeneity of the
interaction strengths between the species, and their sparsity. Regarding this
last assumption, it is indeed commonly observed that a species interacts with a
very small proportion of the other species coexisting within the ecosystem
\cite{bus-etal-17}. 

Our approach is mathematically rigorous,  contrary to the
references we just mentioned. To obtain our results, we generalize the
technique of our recent preprint \cite{akj-hac-mai-naj-(arxiv)23} devoted to
the GOE case.  The idea goes as follows: The vector $u_\star$ can be identified
as the solution of a Linear Complementarity Problem (LCP), a class of problems
studied in the field of linear programming (see
\cite{tak-livre96,cot-pan-sto-livre09}).  In \cite{akj-hac-mai-naj-(arxiv)23},
it is shown that in the GOE case, the asymptotic behavior of $\mu^{x_\star}$,
where $u_\star$ is now the LCP solution, can be evaluated with the help of an
Approximate Message Passing (AMP) technique. Such techniques have recently
aroused an intense and growing research effort in the fields of statistical
physics, communication theory, or statistical Machine Learning
\cite{fen-etal-(now)22}.  In a word, given a function 
$h: \RR \times \RR \times \NN \to \RR$ and a random symmetric $n\times n$ 
so-called measurement matrix $W$, a standard AMP algorithm is an iterative 
algorithm of the form 
\[
\hat x^{t+1} = W h(\hat x^t, \eta, t) \ + \ \text{a ``correction'' term} , 
\]
where $\eta = [\eta_i]_{i=1}^n \in \RR^n$ is a parameter vector, and where 
$h(\hat x^t, \eta, t) = \bigl[ h(\hat x^t_i, \eta_i, t) \bigr]_{i=1}^n$. 
By properly designing the correction term, one is able to control the joint
distribution of the $(t+1)$--uple $(\eta, \hat x^1, \ldots, \hat x^t)$ for
each $t \geq 0$ and for $n\to\infty$, as will developed more precisely below.
We shall be able to make use of such a result to evaluate the large--$n$
distribution of our equilibrium vector, and show that this distribution is
close to a mixture of Gaussians. 

Thus, from the standpoint of theoretical ecology, the present paper is a
generalization of \cite{akj-hac-mai-naj-(arxiv)23} to the case where $\Sigma$
is non-necessarily Gaussian and is subjected to a variance profile that can be
sparse.  To that end, a version of the AMP algorithm well-suited to these kind
of matrices is developed below, and might have its own interest due to its
potential applications in other fields than in ecology.  

Let us provide a quick review of the AMP literature in order to better position
our contribution in this respect.  Many of the original ideas lying behing the
AMP algorithms come from the fields of statistical physics and communication
theory.  The first rigorous AMP results in a framework close to this paper were
developed by Bolthausen~\cite{bol-14} and Bayati and
Montanari~\cite{bay-mon-11} both for GOE matrices and for rectangular matrices
with i.i.d.~Gaussian elements.  Since then, the AMP approach has been
generalized in many directions.  Let us cite the Generalized AMP of
\cite{ran-11}, the contributions \cite{jav-mon-13}, \cite{beh-ree-pmlr22} and
\cite{gui-ko-krz-zde-(arxiv)22,pak-ko-krz-(arxiv)23} where block variance
profiles are considered, \cite{fan-22} devoted to rotationally invariant
matrices, or the graph-based approach of \cite{ger-ber-arxiv22}.  Universality
results in terms of the distribution of the elements of $W$ were proposed in
the recent papers \cite{che-lam-21}, \cite{dud-lu-sen-(arxiv)22}, and
\cite{wan-zho-fan-(arxiv)22}. 
The closest to our paper among these is 
\cite{wan-zho-fan-(arxiv)22}, which considers among others the case where
$W$ is a centered symmetric
matrix with independent elements satisfying $\EE W_{ij}^2 \lesssim 1/n$ and
$\lim_n \max_i | (\sum_j \EE W_{ij}^2) - 1 | = 0$, and shows that 
$W$ can be replaced with a GOE matrix as regards the AMP problem. 
This constraint on the variance profile is alleviated in our context, leading
to a more involved expression of the asymptotic joint distribution of the
algorithm iterates. 

In \cite{bay-lel-mon-15}, Bayati, Lelarge, and Montanari were among the first
to establish a universal AMP result.  Starting with polynomial activation
functions and using a combinatorial approach, these authors propose a moment
computation of the elements of the vector iterates, where these elements are
expressed as sums of monomials in the matrix entries which are indexed by
labelled trees. As regards the AMP part, the present paper is essentially an
adaptation of the approach of \cite{bay-lel-mon-15} to the case of a sparse
variance profile.  

Section~\ref{def-amp} is devoted to our general AMP results for symmetric and
non-symmetric measurement matrices, independently of the ecological
application.  Our LV problem is then stated in Section~\ref{sec-lv}, along with
the result on the large--$n$ distribution of the equilibrium. The proofs for
Sections~\ref{def-amp} and~\ref{sec-lv} are provided in Sections~\ref{prf-amp}
and~\ref{prf-lvsym} respectively. 

\section{Sparse AMP with a variance profile: problem statement and results} 
\label{def-amp}

We start with our assumptions.  Let $(n)$ be a sequence of integers in the set
$\{ 2,3, \ldots \}$ that converges to infinity.  For each $n$, let
$\{X^{(n)}_{ij}\}_{1\leq i < j \leq n}$ be a set of real random variables such
that:  
\begin{assumption} The following facts hold true. 
\label{ass-X} 
\begin{itemize}
\item The $n(n-1)/2$ random variables $X^{(n)}_{ij}$ for $1\leq i < j \leq n$ 
 are independent. 
\item $\EE X^{(n)}_{i,j} = 0$ and $\EE (X^{(n)}_{ij})^2 = 1$. 
\item For each integer $k > 2$, there exists a constant $\Cmom(k) > 0$ such 
that 
\[
 \sup_n \max_{1\leq i < j \leq n} 
 \Bigl( \EE \Bigl| X^{(n)}_{ij} \Bigr|^k \Bigr)^{1/k} \leq \Cmom(k) . 
\]
\end{itemize}
\end{assumption} 
Let us write $X^{(n)}_{ji} = X^{(n)}_{ij}$ for $1\leq i < j \leq n$, and 
$X^{(n)}_{ii} = 0$ for $i\in [n]$, and let us consider the $n\times n$ random 
symmetric matrix $X^{(n)} = \bigl[ X^{(n)}_{ij} \bigr]_{i,j=1}^n$. 

For each $n$, let $\{s^{(n)}_{ij}\}_{1\leq i < j \leq n}$ be a set of 
deterministic non-negative numbers. Write $s^{(n)}_{ji} = s^{(n)}_{ij}$ for 
$1\leq i < j \leq n$, and $s^{(n)}_{ii} = 0$ for $i\in [n]$, and consider the 
$n\times n$ symmetric matrix with non-negative elements 
$S^{(n)} = \bigl[ s^{(n)}_{ij} \bigr]_{i,j=1}^n$. Define the random symmetric
matrix $W^{(n)}$ as 
\begin{equation}
\label{def-W} 
W^{(n)} = \begin{bmatrix} W^{(n)}_{ij} \end{bmatrix}_{i,j=1}^n 
 = \bigl( S^{(n)} \bigr)^{\odot {1/2}} \odot X^{(n)}, 
\end{equation} 
where $\odot$ is the Hadamard product, and where $A^{\odot{1/2}}$ is the 
element-wise square root of the matrix $A$. 

Letting $(K_n)$ be a sequence of positive integers indexed by $n$ such that 
$K_n \leq n$, the variance profile matrix $S^{(n)}$ of $W^{(n)}$ complies with 
the following assumption: 
\begin{assumption} The following facts hold true. 
\label{ass-A} 
\begin{itemize}
\item $K_n \to\infty$. 
\item There exists a constant $C_{\text{card}} > 0$ such that 
\[
\forall n, \ \forall i \in [n], \ 
 \left| \left\{ j \in [n] \, : \, s_{ij}^{(n)} > 0 \right\} \right|
    \leq C_{\text{card}} K_n , 
\]
where $|\cdot|$ is the cardinality of a set. 

\item There exists a constant $ C_S > 0$ such that 
 $s_{ij}^{(n)} \leq C_S K_n^{-1}$ for all $n$ and all $i,j \in [n]$. 
\end{itemize}
\end{assumption} 
Beyond these constraints, we do not put any structural assumption on the
variance profile matrix $S^{(n)}$.  In this paper, we shall be mostly
interested in the situations where $K_n / n \to_n 0$, making the 
matrices $W^{(n)}$ sparse. 

For each $n$, let $x^{(n),0} = \left[ x^{(n),0}_1, \ldots, 
 x^{(n),0}_n \right]^\T\in \RR^n$ be a deterministic vector that will 
represent the initial value of our AMP sequence, and let 
$\eta^{(n)} = \left[ \eta^{(n)}_1, \ldots, \eta^{(n)}_n \right]^\T\in \RR^n$ 
be a deterministic parameter vector. These vectors obey the following 
assumptions: 
\begin{assumption}
\label{x0} 
There exists a compact set $\cQ_x \subset \RR$ such that 
$x^{(n),0}_i \in \cQ_x$ for all $n$ and all $i\in [n]$. 
\end{assumption} 
\begin{assumption}
\label{eta} 
There exists a compact set $\cQ_\eta \subset \RR$ such that 
$\eta^{(n)}_i \in \cQ_\eta$ for all $n$ and all $i\in [n]$. 
\end{assumption} 

Consider a measurable function $h : \RR \times \cQ_\eta \times \NN \to \RR$,
called activation function, that satisfies one of the two following 
assumptions: 
\begin{assumption}[Lipschitz activation function in the AMP parameter] 
\label{ass:h-lip}
For each $t \in \NN$, there exists a constant $C > 0$ and a continuous 
non-decreasing function $\kappa : \RR_+ \to \RR_+$ with $\kappa(0) = 0$, such
that:
\[
\forall x, x' \in \RR, \quad 
 | h(x,\eta,t) - h(x',\eta,t) | \leq C | x - x' |, 
\]
and 
\[
 \forall \eta, \eta' \in \cQ_\eta, \quad 
 | h(x,\eta,t) - h(x,\eta',t) | \leq \kappa(|\eta-\eta'|) 
         ( 1 + |x|). 
\]
\end{assumption}  

\begin{assumption}[AC activation function with polynomial 
growth in the AMP parameter]  
\label{ass:h-gal}
The following properties hold true: 
\begin{itemize}
\item For each $\ueta\in\cQ_\eta$ and each $t\in\NN$, the function 
$h(\cdot, \ueta, t)$ is absolutely continuous. 

\item For each $t\in\NN$, there exists two constants $m,C > 0$ and a continuous
non-decreasing function $\kappa : \RR_+ \to \RR_+$ with $\kappa(0) = 0$, such
that:
\[
 | h(\ux,\ueta,t) | \leq C ( 1 + |\ux|^m), 
\]
and 
\[
 \forall \ueta, \ueta' \in \cQ_\eta, \quad 
 | h(\ux,\ueta,t) - h(\ux,\ueta',t) | \leq \kappa(|\ueta-\ueta'|) 
         ( 1 + |\ux|^m). 
\]

\end{itemize}
\end{assumption}

At the heart of the AMP algorithm that we shall study is a Gaussian 
$\RR^n$--valued sequence of random vectors that we denote for each $n$ as 
$(\bZ^{(n),t})_{t\in\NN_*}$, where $\NN_* = \NN\setminus\{ 0 \}$. The 
probability distribution of this sequence is defined recursively in $t$ as 
follows. 
Let $\bZ^{(n),t} = \left[ Z^{(n),t}_1, \ldots, Z^{(n),t}_n \right]^\T$. 
Writing $Z^{(n)}_i = (Z^{(n),1}_i, Z^{(n),2}_i, \ldots)$, the
sequences $\{Z^{(n)}_i\}_{i=1}^n$ are centered, Gaussian, and independent.
Denote as $R^{(n),t}_i$ the covariance matrix of the vector 
$\vec Z^{(n),t}_i = [ Z^{(n),1}_i, \ldots, Z^{(n),t}_i]^\T$. These matrices
are constructed recursively in the parameter $t$ as follows. 
Defining the variances 
$\{\Xi^{(n),0}_i \}_{i=1}^n$ as
$\Xi^{(n),0}_i = h(x^{(n),0}_i, \eta^{(n)}_i, 0)^2$, we set 
\[
R^{(n),1}_i = \sum_{l\in[n]} s^{(n)}_{il} \Xi^{(n),0}_l. 
\]
Given the matrices $\{ R^{(n),t}_i \}_{i=1}^n$, the covariance matrices 
\begin{equation}
\label{eq-Sigma}
\Xi^{(n),t}_i = \EE \begin{bmatrix} 
  h(x^{(n),0}_i, \eta_i^{(n)}, 0) \\
  h(Z^{(n),1}_i, \eta_i^{(n)}, 1) \\
  \vdots \\ 
  h(Z^{(n),t}_i, \eta_i^{(n)}, t) \end{bmatrix} 
 \begin{bmatrix} 
  h(x^{(n),0}_i, \eta_i^{(n)}, 0) & 
  h(Z^{(n),1}_i, \eta_i^{(n)}, 1) & 
  \cdots & h(Z^{(n),t}_i, \eta_i^{(n)}, t) \end{bmatrix} , 
\end{equation} 
are well-defined for all $i\in [n]$, whether Assumption~\ref{ass:h-lip} or
Assumption~\ref{ass:h-gal} is used. With this at hand, we set 
\begin{equation}
\label{eq-R} 
R^{(n),t+1}_i = \sum_{l\in[n]} s^{(n)}_{il} \Xi^{(n),t}_l . 
\end{equation} 
It is clear that $R^{(n),t}_i$ is the principal matrix of $R^{(n),t+1}_i$ 
consisting in its first $t$ rows and columns. 

The distribution of the sequence $(\bZ^{(n),t})_{t}$ is determined by the
initial value $x^{(n),0}$, the parameter vector $\eta^{(n)}$, the variance
profile matrix $S^{(n)}$, and the activation function $h$. We shall say from
now on that this distribution is determined by the
$(S^{(n)},h,\eta^{(n)},x^{(n),0})$-state evolution equations, which provide the
covariance matrices $R^{(n),t}_i$. The spectral norms $\| R^{(n),t}_i \|$ are 
bounded as specified by the following lemma: 
\begin{lemma} 
\label{lm:bndR} 
Let Assumptions~\ref{ass-A}--\ref{eta}, and either Assumption~\ref{ass:h-lip}
or Assumption~\ref{ass:h-gal} hold true. Then for each integer $t > 0$, it 
holds that 
\begin{equation}
\label{bnd-R} 
\sup_{n} \max_{i\in[n]} \| R^{(n),t}_i \| < \infty . 
\end{equation} 
\end{lemma}
\begin{proof}[Sketch of proof]
By recurrence on $t$. For $t=1$, the result is a consequence of 
Assumption~\ref{ass-A} and the easily shown continuity of 
$h(\cdot,\cdot,0)$ on the compact $\cQ_x \times \cQ_\eta$. 
Assume the result is true for $t$. Then, using Assumption~\ref{eta} and
either Assumption~\ref{ass:h-lip} or Assumption~\ref{ass:h-gal} again, 
standard Gaussian derivations show that
$\sup_n\max_i \EE h(Z^{(n),t}_i, \eta_i, t)^2  < \infty$. From this, 
we obtain by using the Cauchy-Schwarz inequality that 
$\sup_n\max_i \| \Xi^{(n),t}_i \|  < \infty$. By Assumption~\ref{ass-A}, we
then obtain that $\sup_n\max_i \| R^{(n),t+1}_i \|  < \infty$. 
\end{proof} 

Our next assumption will be used to ensure that the variances of the
random variables $Z^{(n),t}_i$ are bounded away from zero: 
\begin{assumption}
\label{nondeg} 
There exists a constant $c_S > 0$ such that 
\[
\inf_n \min_{i\in[n]} \sum_{j\in[n]} s^{(n)}_{ij} \geq c_S . 
\]
Recalling Assumption~\ref{ass-A}, let 
\[
\alpha_S = \frac{c_S}{2 C_S - c_S / C_{\text{card}}}. 
\]
There is a sequence of positive constants $(c_{\text{h}}(0), c_{\text{h}}(1),
\ldots)$ such that for each set $\cS^{(n)} \subset [n]$ with 
$| \cS^{(n)} | = \lfloor \alpha_S K_n \rfloor$, it holds that 
\[
\frac{1}{K_n} \sum_{l \in \cS^{(n)}} 
 h(x^{(n),0}_l,\eta_l^{(n)}, 0)^2 \geq c_{\text{h}}(0) , 
\]
and for each $t \geq 1$, 
\[
\frac{1}{K_n} \sum_{l\in\cS^{(n)}} 
 \EE h(\uxi,\eta_l^{(n)}, t)^2 \geq c_{\text{h}}(t) , 
\] 
where $\uxi \sim \cN(0,1)$.
\end{assumption} 
Non-degeneracy conditions similar to this assumption are usually not
constraining and are invoked in the vast majority of the contributions dealing
with AMP algorithms, see, \emph{e.g.}, \cite{jav-mon-13,fen-etal-(now)22}.  If
needed, it is certainly possible to lighten this assumption by properly
perturbing the matrix $S^{(n)}$ or the function $h$ in the course of the proof 
of Theorems~\ref{th-amp-lip} and~\ref{th-amp-gal} below. A perturbation 
argument of this sort was considered in, \emph{e.g.}, \cite{ber-mon-ngu-19}.  

In all this paper, we follow the following notational convention. Given a
function $f : \RR^k \to \RR$ and a $k$--uple of $\RR^n$--valued vectors
$(a^\ell = [ a^\ell_i ]_{i=1}^n)_{\ell\in [k]}$, we denote as 
$f(a^1,\ldots, a^k)$ the $\RR^n$--valued vector 
$[ f(a^1_i,\ldots, a^k_i) ]_{i=1}^n$.  Similarly, if $f$ is a 
$\RR^k \times \NN \to \RR$ function, $f(a^1,\ldots, a^k,t)$ is the
$\RR^n$--valued vector $[ f(a^1_i,\ldots, a^k_i,t) ]_{i=1}^n$.  
 
With either Assumption~\ref{ass:h-lip} or Assumption~\ref{ass:h-gal} at hand,
we denote as $\partial h(\ux,\ueta,t)$ a measurable function that coincides
almost everywhere on $\RR$ with the partial derivative $\partial h(\cdot,
\ueta,t)$ at $\ux$. Note that under Assumption~\ref{ass:h-lip}, the existence
of this function follows from Rademacher's theorem, and it is furthermore 
obvious that $\EE \partial h(\bZ^{(n),t}, \eta^{(n)}, t)$ exists for each 
$t \in \NN_*$. 
Under Assumption~\ref{ass:h-gal}, Stein's lemma says that 
$(\EE\uxi^2) \EE \partial h(\uxi,\ueta,t) = \EE\uxi h(\uxi,\ueta,t)$ for each 
Gaussian centered random variable $\uxi$, see, \emph{e.g.}, 
\cite[Th.~2.1]{fou-str-wel-(livre)18}. From
this identity, we can deduce that under Assumption~\ref{ass:h-gal}, 
$\EE \partial h(\bZ^{(n),t}, \eta^{(n)}, t)$ also exists for each 
$t \in \NN_*$. 

We are now ready to state the results of this section.  The AMP algorithm built
around the matrix $W^{(n)}$ and the activation function $h$ that we shall study
here takes the following form: Starting with the vector $x^{(n),0}$, the vector
$x^{(n),t+1} = [ x^{(n),t+1}_i ]_{i=1}^n \in \RR^n$ delivered by this algorithm
at Iteration~$t+1$ is given as 
\begin{equation}
\label{alg-amp} 
x^{(n),t+1} = 
 W^{(n)} h(x^{(n),t},\eta^{(n)},t)  - 
 \diag \Bigl(S^{(n)} \EE \partial h(\bZ^{(n),t}, \eta^{(n)}, t)\Bigr) 
   h(x^{(n),t-1},\eta^{(n)},t-1)  , 
\end{equation} 
with the term $\diag(\cdots) h(\cdots)$ being equal to zero for $t=0$. 
Variants for this algorithm could be considered by replacing the diagonal 
matrix 
\[
\diag (S^{(n)} \EE\partial h(\bZ^{(n),t}, \eta^{(n)}, t) )
\]
above with  $\diag (S^{(n)} \partial h(x^{(n),t}, \eta^{(n)}, t) )$, or
with $\diag ( (W^{(n)})^{\odot 2} \partial h(x^{(n),t}, \eta^{(n)}, t) )$
where $A^{\odot 2}$ is the square of the matrix $A$ in the Hadamard product, 
as is done in the literature in similar contexts. Algorithm~\eqref{alg-amp} 
(or its variants) generalizes the algorithms studied in, \emph{e.g.}, 
\cite{jav-mon-13} or \cite{bay-lel-mon-15}.  

Given an integer $k > 0$, a function $\varphi : \RR^k \to \RR$ is said
pseudo-Lipschitz of degree~$2$ (notation: $\varphi\in\PL_2(\RR^k)$) if there
exists a constant $L > 0$ such that 
\[
\forall u,v \in \RR^k, \quad 
|\varphi(u) - \varphi(v)| \leq L \| u - v \| ( 1 + \| u \| + \| v \| ), 
\]
where $\| u \|$ is the Euclidean norm of the vector $u$. 
Our first result is devoted to the case where $h(\cdot,\eta,t)$ is 
Lipschitz: 
\begin{theorem}
\label{th-amp-lip} 
Let Assumptions~\ref{ass-X}--\ref{ass:h-lip} and~\ref{nondeg} hold true.  Let
$C_W > 0$ be an arbitrary constant. Fix an arbitrary integer $\tmax > 0$.
For each positive integer $n$, let $(\beta_1^{(n)}, \ldots, \beta_n^{(n)})$ a 
$n$--uple of real numbers such that 
$\sup_n \max_{i\in[n]} | \beta^{(n)}_i | < \infty$.  
Let $\varphi : \RR^{\tmax+1} \to \RR$ be a function in $\PL_2(\RR^{\tmax+1})$.
Then, for each $\varepsilon > 0$, it holds that 
\begin{multline} 
\label{cvg-W2} 
\PP\left[\left[ \left| \frac 1n \sum_{i\in[n]} 
 \beta^{(n)}_i \varphi( \eta^{(n)}_i, x^{(n),1}_i, \ldots, x^{(n),\tmax}_i) - 
 \beta^{(n)}_i \EE\varphi(\eta^{(n)}_i, Z^{(n),1}_i, \ldots, Z^{(n),\tmax}_i) 
  \right| \geq \varepsilon \right] \right. \\
 \left. \phantom{\sum_{i\in[n]}} \cap \left[ \| W^{(n)} \| \leq C_W \right] 
 \right] \xrightarrow[n\to\infty]{} 0 . 
\end{multline} 
\end{theorem} 

Before interpreting this theorem, we need to bound the operator norm 
$\| W^{(n)} \|$. With the help of the results of Bandeira and van Handel in 
\cite{ban-vha-16}, this can be done in the framework of the following 
assumption:
\begin{assumption}
\label{ass-mom}
There exists a constant $C > 0$ such that the moment bounds from 
Assumption~\ref{ass-X} satisfy 
\[
\Cmom(k) \leq C k^{\rho / 2}, 
\]
for some $\rho \geq 0$. 
\end{assumption} 
As is well known, the cases $\rho = 1$ and $\rho = 2$ correspond to the 
sub-Gaussian and the sub-exponential distributions respectively. 
\begin{proposition} 
\label{Wfini} 
Let Assumptions~\ref{ass-X}, \ref{ass-A} and~\ref{ass-mom} hold true, and 
assume furthermore that $K_n \gtrsim (\log n)^{\rho \vee 1}$. Then, there 
exists a constant $C > 0$ such that $\lim\sup_n \| W^{(n)} \| \leq C$ with 
probability one. 
\end{proposition} 
Thus, the heavier the tail of the probability distributions of the
$X^{(n)}_{ij}$, the faster we need to make $K_n$ converge to infinity in order
to ensure that $\| W^{(n)} \|$ is bounded.  This proposition is proven in
Appendix~\ref{anx-Wfini}. 

We now comment Theorem~\ref{th-amp-lip}. 

It is useful to interpret the result of this theorem in terms of the
empirical distribution of the particles $\{ (\eta^{(n)}_i, x^{(n),1}_i, \ldots,
x^{(n),\tmax}_i) \}_{i=1}^n$ as an element of a Wasserstein space of probability
measures.  For an integer $k > 0$, denote as $\cP(\RR^k)$ the
space of probability measures on $\RR^k$, and by $\cP_2(\RR^k)$ the Wasserstein
space of probability measures on $\RR^k$ with finite second moment,
equipped with the Wasserstein distance $\bs d_2$ \cite{vil-livre09}. Given
a sequence of probability measures $(\nu_n)$ in $\cP_2(\RR^k)$ and a 
probability measure $\nu \in \cP_2(\RR^k)$, it is well-known that the three 
following assertions are equivalent \cite{vil-livre09,fen-etal-(now)22}: 
\begin{enumerate}[label=\roman*)]
\item  $\bs d_2(\nu_n, \nu) \to 0$. 
\item\label{testquad}
For each continuous function $\varphi : \RR^k \to \RR$ such that 
 $|\varphi(x_1, \ldots, x_k) | \leq C(1 + x_1^2 + \cdots + x_k^2)$ for
some $C > 0$, it holds that 
\begin{equation}
\label{cvg-test} 
\int \varphi d\nu_n \tolong \int \varphi d\nu . 
\end{equation} 
\item\label{testpl}
The convergence \eqref{cvg-test} holds true for each function 
 $\varphi \in \PL_2(\RR^k)$. 
\end{enumerate}
Furthermore, if the sequence $(\nu_n)$ is random, then, the
convergence $\bs d_2(\nu_n, \nu) \toprobashort 0$ is equivalent to the
convergence $\int \varphi d\nu_n \toprobashort \int \varphi d\nu$ for 
each continuous function $\varphi:\RR^k \to \RR$ either with quadratic
growth as in~\ref{testquad}, or belonging to the class $\PL_2(\RR^k)$. 

We are interested here in the asymptotic
behavior of the $\cP_2(\RR^{\tmax+1})$--valued random probability measure 
$\mu^{\eta^{(n)}, x^{(n),1}, \ldots, x^{(n),\tmax}}$, defined as
\[
\mu^{\eta^{(n)}, x^{(n),1}, \ldots, x^{(n),\tmax}} = \frac 1n \sum_{i=1}^n 
  \delta_{(\eta^{(n)}_i, x^{(n),1}_i, \ldots, x^{(n),\tmax}_i)} .
\]
In general, this random measure does not converge to any fixed probability
measure, because we did not put any assumption on the variance profile of the
elements of $W^{(n)}$ beyond Assumption~\ref{ass-A}. However, let us define the
deterministic measure $\bs\mu^{(n),\tmax} \in \cP(\RR^{\tmax+1})$ as follows: 
Letting $\theta^{(n)}$ be a discrete random variable uniformly distributed on 
the set $[n]$ and independent of the sequence $(\bZ^{(n),t})_t$, we put 
$\bs\mu^{(n),\tmax} =
\mcL((\eta^{(n)}_{\theta^{(n)}}, \vec Z^{(n),\tmax}_{\theta^{(n)}}))$. In 
particular, $\bs\mu^{(n),\tmax}(\RR \times \cdot)$ is a mixture of $n$ 
multivariate Gaussians. By Assumption~\ref{eta} and Lemma~\ref{lm:bndR}, one 
can readily show that the sequence 
$(\bs\mu^{(n),\tmax})_{n\geq 2}$ is pre-compact in the space 
$\cP_2(\RR^{\tmax+1})$. Moreover, taking $(\beta_1^{(n)}, \ldots, 
 \beta_n^{(n)}) = (1,\ldots, 1)$ in the statement of Theorem~\ref{th-amp-lip}
and assuming that $W^{(n)}$ is constructed in such a way that 
$\limsup_n \| W^{(n)} \| < \infty$ with probability one 
(see Proposition~\ref{Wfini}), Theorem~\ref{th-amp-lip} shows that 
\begin{equation}
\label{deteq} 
\forall \varphi \in \PL_2(\RR^{\tmax+1}), \quad 
\int \varphi \, d\mu^{\eta^{(n)}, x^{(n),1}, \ldots, x^{(n),\tmax}} 
  - \int \varphi \, d\bs\mu^{(n),\tmax} 
 \toprobalong 0,  
\end{equation} 
which is equivalent to 
$\bs d_2(\mu^{\eta^{(n)}, x^{(n),1}, \ldots, x^{(n),\tmax}}, 
 \bs\mu^{(n),\tmax}) \toprobashort 0$, and also equivalent to replacing the 
statement ``$\forall \varphi \in \PL_2(\RR^{\tmax+1})$'' above with 
``$\forall \varphi: \RR^{\tmax+1} \to \RR$ continuous with 
$|\varphi(u_1,\ldots,u_{\tmax+1})| \leq C (1 + u_1^2 + \cdots u_{\tmax+1}^2)$''
 \cite{fen-etal-(now)22}. 


Let us now describe a case where the convergence~\eqref{deteq} boils down to 
the convergence of the empirical measure 
$\mu^{\eta^{(n)}, x^{(n),1}, \ldots, x^{(n),\tmax}}$ to a
fixed probability measure.  Consider the situation where $K_n = n$ and where
the variances satisfy $s_{ij}^{(n)} = 1 / n$ for all $i,j \in [n]$ (without
more details, we neglect here the constraint $s_{ii}^{(n)} = 0$ that will
be used in the proof of Theorem~\ref{th-amp-lip}). In this case,
it is easy to check that the Gaussian vectors 
$\{ \vec Z^{(n),t}_i \}_{i\in[n]}$ are i.i.d., and furthermore, writing 
$\vec Z^{(n),t}_i \eqlaw \vec Z^{(n),t} = [ Z^{(n),1}, \ldots, Z^{(n),t}]^\T$ 
with the covariance matrix $R^{(n),t}$, it holds that 
\begin{align*} 
&R^{(n),t+1} = \\
&\frac 1n \sum_{i=1}^n \EE \begin{bmatrix} 
  h(x^{(n),0}_i, \eta_i^{(n)}, 0) \\
  h(Z^{(n),1}, \eta_i^{(n)}, 1) \\
  \vdots \\ 
  h(Z^{(n),t}, \eta_i^{(n)}, t) \end{bmatrix} 
 \begin{bmatrix} 
  h(x^{(n),0}, \eta_i^{(n)}, 0) & h(Z^{(n),1}, \eta_i^{(n)}, 1) & 
  \cdots & h(Z^{(n),t}, \eta_i^{(n)}, t) \end{bmatrix} .  
\end{align*} 
Assume now that the probability measure $\nu^{(n)} = n^{-1} \sum_i
\delta_{(x^{(n),0}_i, \eta_i^{(n)})}$ converges narrowly to a deterministic
probability measure $\nu^\infty$ as $n\to\infty$, as it is frequently done in
the literature, often under a randomness assumption on $(x^{(n),0},
\eta^{(n)})$ coupled with an independence assumption from $W^{(n)}$. 
Then, given a couple of real random variables
$(\overline x, \overline\eta)$ 
with the distribution $\nu^\infty$, we can readily replace $R^{(n),t+1}$ with 
the matrix $R^{\infty,t+1}$ that does not depend on $n$ and that is written as 
\[
R^{\infty,t+1} = 
\EE_{(\overline x, \overline\eta)} \EE_{Z^{\infty,t}} \begin{bmatrix} 
  h(\overline x, \overline\eta, 0) \\
  h(Z^{\infty,1}, \overline\eta, 1) \\
  \vdots \\ 
  h(Z^{\infty,t}, \overline\eta, t) \end{bmatrix} 
 \begin{bmatrix} 
  h(\overline x, \overline\eta, 0) & h(Z^{\infty,1}, \overline\eta, 1) & 
  \cdots & h(Z^{\infty,t}, \overline\eta, t) \end{bmatrix} , 
\] 
with $Z^{\infty,t} \sim \cN(0, R^{\infty,t})$ being independent of 
$(\overline x, \overline\eta)$. One consequence of this replacement is that
the measure $\bs\mu^{(n),\tmax}$ in~\eqref{deteq} can be replaced with the
distribution $\bs\mu^{\infty,\tmax} = \mcL((\overline\eta, Z^{\infty,\tmax}))
 \in \cP_2(\RR^{\tmax+1})$, and the convergence~\eqref{deteq} amounts to 
\[
\mu^{\eta^{(n)}, x^{(n),1}, \ldots, x^{(n),\tmax}}  
 \toprobalong \bs\mu^{\infty,\tmax}  \quad \text{in } \cP_2(\RR^{\tmax+1}). 
\]
We recover here a well-known result given in, \emph{e.g.}, 
\cite{bay-mon-11,fen-etal-(now)22}.  
 
In the general setting of this paper, only the convergence~\eqref{deteq} is
available. Note that in general, putting the assumptions on $x^{(n),0},
\eta^{(n)}$ mentioned above does not suffice to transform this convergence to a
convergence to a fixed deterministic measure. Once again, this is due to the
fact that we did not put any assumption on the asymptotic behavior of the
variance profile matrix $S^{(n)}$ beyond Assumption~\ref{ass-A}. 

Our aim now is to explore a generalization of Theorem~\ref{th-amp-lip} into two
directions: We replace Assumption~\ref{ass:h-lip} on the activation function
with the more general assumption~\ref{ass:h-gal}, and we replace the test
function $\varphi\in\PL_2(\RR^{\tmax+1})$ (or, equivalently, a test function
$\varphi$ satisfying the statement \ref{testquad} above) with a continuous 
function $\varphi: \cQ_\eta\times \RR^{\tmax} \to \RR$ such that 
 $|\varphi(\alpha, u_1, \ldots, u_\tmax) | 
 \leq C(1 + |u_1|^m + \cdots + |u_\tmax|^m)$ for a given arbitrarily integer 
$m > 0$. In this case, we obtain the following partial 
result that generalizes \cite{bay-lel-mon-15}.  
For a vector $u$, we write hereinafter $\| u \|_n = \| u \| / \sqrt{n}$. 
\begin{theorem}
\label{th-amp-gal} 
Let Assumptions~\ref{ass-X}--\ref{eta} and \ref{ass:h-gal}, \ref{nondeg} hold
true. Let $C_W > 0$ be an arbitrary constant.  Fix an arbitrary integer 
$\tmax > 0$. For each positive integer $n$, let $(\beta_1^{(n)}, \ldots,
\beta_n^{(n)})$ be a $n$--uple of real numbers such that 
$\sup_n \max_{i\in[n]} | \beta^{(n)}_i | < \infty$. 
Let $\varphi : \cQ_\eta \times \RR^{\tmax} \to \RR$ be a continuous function 
satisfying  
$| \varphi(\alpha,u_1,\ldots,u_{\tmax})| \leq C ( 1 + |u_1|^m + \cdots +
|u_{\tmax}|^m)$ for some $C, m > 0$. 
Then, there exists a sequence of matrices $(\tbX^{(n)})_{n\geq 2}$,  
with $\tbX^{(n)} = \left[ \tx^{(n),1}, \ldots, \tx^{(n),\tmax} \right] \in
\RR^{n\times \tmax}$, such that for every $\varepsilon > 0$,  
\[
\PP\left[ \left[ 
 \left\| \tx^{(n),t+1} - \bar x^{(n),t+1} \right\|_n > \varepsilon) \right]  
  \cap \left[ \| W^{(n)} \| \leq C_W \right] \right] \tolong 0 
\]
for each $t = 0, 1, \ldots, \tmax-1$, with 
\[
\bar x^{(n),t+1} = 
  W^{(n)} h(\tx^{(n),t},\eta^{(n)},t)  - 
 \diag \left( S^{(n)} \EE \partial h(\bZ^{(n),t},\eta^{(n)},t)\right) 
  h(\tx^{(n),t-1},\eta^{(n)},t-1) 
\]
and with $\tx^{(n), 0} = x^{(n), 0}$ and $\diag(\cdots) h(\cdots) = 0$ for 
$t = 0$, and furthermore, 
\[
\frac 1n \sum_{i\in[n]} \beta^{(n)}_i
  \varphi( \eta^{(n)}_i, \tx^{(n),1}_i, \ldots, \tx^{(n),\tmax}_i) - 
  \beta^{(n)}_i
  \EE\varphi(\eta^{(n)}_i, Z^{(n),1}_i, \ldots, Z^{(n),\tmax}_i) 
 \toprobalong 0 . 
\] 
\end{theorem} 
Observe that the sequence $(\tbX^{(n)})$ does not exactly describe an AMP
sequence, furthermore, its construction depends on the function $\varphi$ and
the $n$--uples $(\beta^{(n)}_i)_{i\in[n]}$. Despite these limitations,
Theorem~\ref{th-amp-gal} is useful in most of the practical situations where
the activation functions are not necessarily Lipschitz in the AMP parameter or 
where the test functions grow faster than quadratically at infinity. \\ 

The AMP algorithms studied in the literature for symmetric measurement 
matrices have non-symmetric analogues. For the sake of generality, we now 
succinctly deal with this case. The results of the next paragraph will not
be used elsewhere in this paper. 

\subsection*{The non-symmetric variant (sketch)} 

Let $(p(n))$ be a sequence of positive integers such that $0 < \inf p(n) / n
\leq \sup p(n)/n < \infty$.  For a given $p = p(n)$, let 
$G^{(n)} = \bigl[ G_{ij}^{(n)} \bigr]_{i,j=1}^{p,n}$ be a random rectangular 
$p\times n$ matrix, where the random variables $G_{ij}^{(n)}$ for 
$(i,j)\in[p]\times [n]$ are independent, centered, and unit-variance random 
variables that share with the elements of $X^{(n)}$ the moment bounds given by 
Assumption~\ref{ass-X}. Let 
$B^{(n)}  = \bigl[ b^{(n)}_{ij} \bigr]_{i,j=1}^{p,n} \in \RR^{p\times n}$ be a 
deterministic matrix with non-negative elements, and assume that the 
$(n+p)\times (n+p)$ symmetric matrix 
$\begin{bmatrix} 0 & B^{(n)} \\ (B^{(n)})^\T & 0 \end{bmatrix}$ 
satisfies Assumption~\ref{ass-A} in the role of $S^{(n)}$ there. 
The rectangular $p\times n$ matrix of interest here is the random matrix
$Y^{(n)}$ with independent elements and a variance profile defined as 
\[
Y^{(n)} = (B^{(n)})^{\odot {1/2}} \odot G^{(n)} .
\]
Let $\tilde x^{(n),0}  = [ \tilde x^{(n),0}_j ]_{i\in[n]} \in \RR^n$ be a 
deterministic vector which elements belong to a compact set as in 
Assumption~\ref{x0}. Let $\eta^{(n)} = [ \eta^{(n)}_i ]_{i\in[p]} \in \RR^p$ 
and $\tilde\eta^{(n)} = [ \tilde\eta^{(n)}_j ]_{j\in[n]} \in \RR^n$ be 
deterministic parameter vectors which 
elements belong to a compact set as in Assumption~\ref{eta}.  Define two 
functions $h(\ux, \ueta, t)$ and $\tilde h(\tilde\ux,\tilde\ueta,t)$ that 
both satisfy Assumption~\ref{ass:h-lip}. Assume furthermore that these 
functions satisfy non-degeneracy assumptions of the type of 
Assumption~\ref{nondeg}, and that all the row and column sums of $B^{(n)}$ are 
lower bounded by a positive bound independent of $n$.  

With these objects, we define a centered Gaussian $\RR^p$--valued process
$(\bZ^{(n),t})_{t\in\NN}$, and a centered Gaussian $\RR^n$--valued process 
$(\tbZ^{(n),t})_{t\in\NN_*}$ as follows. 
Write $\bZ^{(n),t} = \left[ Z^{(n),t}_1, \ldots, Z^{(n),t}_p \right]^\T$, and
$\tbZ^{(n),t} = \left[ \tZ^{(n),t}_1, \ldots, \tZ^{(n),t}_n \right]^\T$, 
and define the sequences 
$Z^{(n)}_i = (Z^{(n),0}_i, Z^{(n),1}_i, \ldots)$ for $i\in[p]$, and 
$\tZ^{(n)}_j = (\tZ^{(n),1}_j, \tZ^{(n),2}_j, \ldots)$ for $j\in[n]$. 
Then, the sequences $\{Z^{(n)}_i\}_{i=1}^p$ are independent, and so is the
case of the sequences $\{\tZ^{(n)}_j\}_{j=1}^n$. 
The covariance matrices 
$R^{(n),t}_i = \cov( Z^{(n),0}_i, \ldots, Z^{(n),t}_i)$ and 
$\tR^{(n),t}_j = \cov( \tZ^{(n),1}_j, \ldots, \tZ^{(n),t}_j)$ are recursively
defined in $t$ as follows. 

Writing 
$\tXi^{(n),0}_j = \tilde h(\tilde x^0_j, \tilde\eta^{(n)}_j, 0)^2$ for 
$j\in[n]$, we set 
\[
R^{(n),0}_i = \sum_{j\in[n]} b^{(n)}_{ij} \tXi^{(n),0}_j , \quad i\in[p]. 
\]
With these variances at hand, we can construct the variances  
$\Xi^{(n),0}_i = \EE h(Z^{(n),0}_i, \eta^{(n)}_i, 0)^2$ for $i\in[p]$, and 
then, we set 
\[
\tR^{(n),1}_j = \sum_{i\in[p]} b^{(n)}_{ij} \Xi^{(n),0}_i , \quad j\in[n]. 
\]
Let $t > 0$, and assume that the matrices $\tR^{(n),t}_j$ are available. Then
we are able to construct the $n$ covariance matrices 
\[
\tXi^{(n),t}_j = \EE 
 \begin{bmatrix} \tilde h(\tilde x^{(n),0}_j, \tilde\eta^{(n)}_j, 0) \\
  \tilde h(\tZ^{(n),1}_j, \tilde\eta^{(n)}_j, 1) \\
   \vdots \\
  \tilde h(\tZ^{(n),t}_j, \tilde\eta^{(n)}_j, t) \end{bmatrix} 
  \begin{bmatrix} 
   \tilde h(\tilde x^{(n),0}_j, \tilde\eta^{(n)}_j, 0) & \cdots & 
  \tilde h(\tZ^{(n),t}_j, \tilde\eta^{(n)}_j, t) 
  \end{bmatrix} ,
\]
and we can then construct the $p$ covariance matrices 
\[
R^{(n),t}_i = \sum_{j\in[n]} b^{(n)}_{ij} \tXi^{(n),t}_j . 
\]
With these matrices at hand, we are able to write 
\[
\Xi^{(n),t}_i = \EE 
 \begin{bmatrix} 
   h(Z^{(n),0}_i, \eta^{(n)}_i, 0) \\
   \vdots \\
  h(Z^{(n),t}_i, \eta^{(n)}_i, t) \end{bmatrix} 
  \begin{bmatrix} 
   h(Z^{(n),0}_i, \eta^{(n)}_i, 0) & \cdots & 
  h(Z^{(n),t}_i, \eta^{(n)}_i, t) 
  \end{bmatrix} ,
\]
and we set 
\[
\tR^{(n),t+1}_j = \sum_{i\in[p]} b^{(n)}_{ij} \Xi^{(n),t}_i , 
 \quad j\in [n]. 
\]
With these equations, the AMP non-symmetric equations take the form
\begin{align*}
x^{(n),t} &= Y^{(n)} \tilde h(\tilde x^{(n),t}, \tilde\eta^{(n)}, t) 
 - \diag\bigl( B^{(n)} 
   \EE \partial\tilde h(\tZ^{(n),t}, \tilde\eta^{(n)}, t)\bigr) 
  h(x^{(n),t-1}, \eta^{(n)}, t-1) \\
\tilde x^{(n),t+1} &= (Y^{(n)})^\T h(x^{(n),t}, \eta^{(n)}, t) 
 - \diag\bigl( (B^{(n)})^\T  
   \EE \partial h(Z^{(n),t}, \eta^{(n)}, t)\bigr) 
  \tilde h(\tilde x^{(n),t}, \tilde\eta^{(n)}, t) , 
\end{align*}  
where the vectors $\tilde h(\tilde x^t, \tilde\eta, t) \in \RR^n$ and 
$h(x^t, \eta, t) \in \RR^p$ are defined in the obvious way, and where 
$h(x^{-1},\eta,-1)$ is zero. 

Furthermore, considering an arbitrary integer $\tmax \geq 0$, two 
test functions $\varphi, \tilde\varphi \in \PL_2(\RR^{\tmax +2})$, and for
each $n$, a $p$--uple $(\beta^{(n)}_i)_{i\in[p]}$ and a $n$--uple 
$(\tilde\beta^{(n)}_j)_{j\in[n]}$ which elements are uniformly bounded as in 
the statement of Theorem~\ref{th-amp-lip}, it holds that 
\begin{multline*} 
 \PP\left[\left[ \left| \frac 1n \sum_{i\in[p]} 
 \beta^{(n)}_i \varphi( \eta^{(n)}_i, x^{(n),0}_i, \ldots, x^{(n),\tmax}_i) - 
 \beta^{(n)}_i \EE\varphi(\eta^{(n)}_i, Z^{(n),0}_i, \ldots, Z^{(n),\tmax}_i) 
 \right| \geq \varepsilon \right] \right. \\ 
 \left. \phantom{\sum_{i\in[p]}} \cap \left[ \| Y^{(n)} \| \leq C_Y \right] 
 \right] \tolong 0 , 
\end{multline*}
and 
\begin{multline*} 
 \PP\left[\left[ \left| \frac 1n \sum_{j\in[n]} 
 \tilde\beta^{(n)}_j 
 \tilde\varphi( \tilde\eta^{(n)}_j, \tilde x^{(n),1}_j, \ldots, 
               \tilde x^{(n),\tmax+1}_j) - 
 \tilde\beta^{(n)}_j \EE\tilde\varphi(\tilde\eta^{(n)}_j, \tZ^{(n),1}_j, 
    \ldots, \tZ^{(n),\tmax+1}_j) 
 \right| \geq \varepsilon \right] \right. \\ 
 \left. \phantom{\sum_{i\in[p]}} \cap \left[ \| Y^{(n)} \| \leq C_Y \right] 
 \right] \tolong 0 
\end{multline*}
for each $\varepsilon > 0$ and each $C_Y > 0$. 

These expressions can be deduced from the symmetric case described by
Theorem~\ref{th-amp-lip} by replacing the matrix $W$ there with the $(n+p)
\times (n+p)$ symmetric matrix $\begin{bmatrix} 0 & Y^{(n)} \\ (Y^{(n)})^\T & 0
\end{bmatrix}$, and by choosing the activation function adequately. More
details on this passage from the symmetric to the non-symmetric case are
provided in \cite{jav-mon-13}. 

An analogue of Theorem~\ref{th-amp-gal} can also be devised for the 
non-symmetric case. We omit the related details.

\section{Application to the equilibria of LV systems} 
\label{sec-lv} 

We now apply the previous results to the study of the distribution of the
globally stable equilibria of large Lotka-Volterra ecological systems. 

Let $\RR_{*+} = (0,\infty)$. 
Keeping the sequence $(n)$ of integers introduced in Section~\ref{def-amp}, let
$(r^{(n)})$ be a sequence of vectors such that $r^{(n)} \in \RR_{*+}^n$, and
let $(\Sigma^{(n)})$ be a sequence of random symmetric matrices such that 
$\Sigma^{(n)} \in \RR^{n\times n}$.  
For a given $n$, starting with a vector $u^{(n)}(0) \in \RR_{*+}^n$, our 
LV ODE is given as 
\begin{equation}
\label{lv} 
\dot u^{(n)}(t) = u^{(n)}(t) \odot 
  \left( r^{(n)} + \left(\Sigma^{(n)} - I_n \right) u^{(n)}(t) \right), 
    \quad t \geq 0 .  
\end{equation}
We recall that $r^{(n)}$ is the so-called vector of intrinsic growth
rates of the species, and the matrix $\Sigma^{(n)}$ is the species interaction
matrix. 
When this ODE has a unique solution $u^{(n)}: \RR_+ \to \RR_+^n$ for each
$u^{(n)}(0) \in \RR_{*+}^n$, and when the image of this function is
pre-compact, we say that the ODE is well-defined.  A sufficient condition for
well-definiteness is $\| \Sigma^{(n)} \| < 1$ \cite{li-etal-09}.  It is also
well-known that under this condition, the ODE~\eqref{lv} has a globally stable
equilibrium $u^{(n)}_\star\in\RR_+^n$ in the classical sense of the Lyapounov
theory~\cite[Chap.~3]{tak-livre96}. We shall focus herein on the distribution
$\mu^{u^{(n)}_\star} \in \cP(\RR_+)$ of the elements of $u^{(n)}_\star$ for the
large values of $n$, which reflects the relative abundances of the species at
the equilibrium. 

We now state our assumptions regarding the vectors $r^{(n)}$ and the
matrices $\Sigma^{(n)}$.  Let us call ``hypotheses'' the assumptions relative
to our LV system. Our first hypothesis reads as follows: 

\begin{hypothesis}
\label{lv-sigma}
Let $X^{(n)}$ be the random symmetric $n\times n$ matrix constructed after
Assumption~\ref{ass-X} above.
Let $V^{(n)}  = \left[ v^{(n)}_{ij} \right]_{i,j=1}^n$ be a deterministic
symmetric matrix with the same structure as $S^{(n)}$ in Section~\ref{def-amp},
and that complies with Assumption~\ref{ass-A} above, with the $s^{(n)}_{ij}$ in
this assumption being replaced with $v^{(n)}_{ij}$. Furthermore, assume that
there is a constant $c_V > 0$ such that
\[
\min_{i\in[n]} \sum_{j\in[n]} v^{(n)}_{ij} \geq c_V
\]
(see Assumption~\ref{nondeg}). Then, the matrices $\Sigma^{(n)}$ are written as 
\[
\Sigma^{(n)} = (V^{(n)})^{\odot {1/2}} \odot X^{(n)} . 
\]
\end{hypothesis} 
According to our model, Species $i$ interacts only with the species in the 
set $\cV_i^{(n)} = \{ j \in [n], \ v^{(n)}_{ij} > 0 \}$. 
The relative cardinalities of these sets satisfy 
$\min_i (| \cV^{(n)}_i | / n) \sim K_n / n$ and 
$\max_i (| \cV^{(n)}_i | / n) \sim K_n / n$.  This ratio is the 
 ``degree of sparsity'' of our model,

Since $\Sigma^{(n)}$ is random, we need to guarantee that 
$\limsup_n \| \Sigma^{(n)} \| < 1$ with probability one to ensure that the LV
ODE is well-defined and has a global equilibrium. We consider this condition
as a hypothesis: 
\begin{hypothesis}
\label{|Sig|} 
$\limsup_n \| \Sigma^{(n)} \| < 1$ with probability one. 
\end{hypothesis} 
Let us provide some comments on this hypothesis. 
The spectral norms of random matrices with structures close to $W$
were studied in~\cite{ban-vha-16} in the framework of Assumption~\ref{ass-mom}. 
In this setting, explicit bounds on the spectral norm of $\Sigma^{(n)}$ can be
obtained in some cases, among which the Gaussian case plays a prominent role.
We note that in general, these bounds are not tight. 
Let $\rownorm A \rownorm$ and $\| A \|_\infty$ be respectively the maximum row
sum norm and the max norm of the matrix $A$ respectively.  When $X^{(n)}$ is a 
Gaussian matrix, we know from \cite[Th.~1.1]{ban-vha-16} that 
\[
\EE\| \Sigma^{(n)} \| \leq T_{\text{Gauss}}^{(n)}, 
\]
where 
\[
T_{\text{Gauss}}^{(n)} = (1+\varepsilon) 
  \left( 2 \rownorm V^{(n)} \rownorm^{1/2} 
  + \frac{6}{\sqrt{\log(1+\varepsilon)}} (\| V^{(n)} \|_\infty \log n)^{1/2}
 \right) 
\]
for an arbitrary $\varepsilon > 0$. Furthermore, by Gaussian concentration, 
\begin{equation}
\label{gauss-con} 
\PP\left[ \| \Sigma^{(n)} \| \geq T_{\text{Gauss}}^{(n)} + t \right] 
  \leq \exp(- t^2 / (2 \| V^{(n)}\|_\infty)^2 ) , 
\end{equation} 
for $\varepsilon \in (0, 1/2]$, as given by \cite[Cor.~3.9]{ban-vha-16}.  By
consequence, it holds that in the Gaussian case,  
$\limsup_n \| \Sigma^{(n)} \| < 1$ with probability one if $\limsup_n
T_{\text{Gauss}}^{(n)} < 1$ for some choice of $\varepsilon \in (0, 1/2]$, 
and if the sequence $(K_n)$ satisfies $K_n \gtrsim \log n$, owing to the 
constraint on $\| V^{(n)}\|_\infty$ provided in Assumption~\ref{ass-A}. 
In the Gaussian case, we can therefore admit a degree of sparsity of 
$\log n / n$. Under our hypothesis~\ref{lv-sigma}, this rate turns out to be 
the optimal rate for Hypothesis~\ref{|Sig|} to be verified, as a consequence 
of \cite[Cor.~3.15]{ban-vha-16}. 

Similar explicit bounds on $\limsup_n \| \Sigma^{(n)} \|$ are available in 
the case where the random variables $X^{(n)}_{ij}$ are bounded,  
see \cite[Cor.~3.12]{ban-vha-16} for more details. Here also, 
Hypothesis~\ref{|Sig|} is ensured when $\rownorm V^{(n)} \rownorm$ and
 $ \| V^{(n)} \|_\infty \log n$ are bounded properly. 

Under Hypothesis~\ref{lv-sigma} without more specific information on the 
distributions of the $X_{ij}$, Hypothesis~\ref{|Sig|} is verified if both 
$\limsup_n \rownorm V^{(n)} \rownorm$ and 
 $\limsup_n \| V^{(n)} \|_\infty (\log n)^{\rho\vee 1}$ are small enough, 
as shown by Proposition~\ref{Wfini}. We thus need that 
$K_n \gtrsim (\log n)^{\rho\vee 1}$ with a large enough factor. 

In any case, we shall also need 
\begin{hypothesis} 
\label{hyp-V} 
$\limsup_n |\!|\!| V^{(n)} |\!|\!| < 1/4$.
\end{hypothesis}
In the classical Wigner case where $V^{(n)} = \alpha n^{-1} 1_n 1_n^\T$ for
some $\alpha > 0$, this condition is necessary and sufficient to ensure that
$\limsup_n \| \Sigma^{(n)} \| < 1$ with probability one, by the well-known
result on the almost sure convergence of $\|\Sigma^{(n)} \|$ to the edge of the
semi-circle law.  Sparse cases that behave similarly to the Wigner case in this
respect, and thus, that require Hypothesis~\ref{hyp-V}, were treated in the
literature. These include the models dealt with in \cite[Sec.~4]{ban-vha-16}.
Similar results can be found in \cite{kho-08,sod-10,ben-pec-14}.  

Regarding the intrinsic growth rate vector, we consider the following 
hypothesis: 
\begin{hypothesis} 
\label{lv-r} 
The vector $r^{(n)}$ is a deterministic vector 
which elements belong to a compact $\cQ_r \subset \RR_{*+}$. 
\end{hypothesis} 

In the practical ecological settings where the intrinsic growth rates within a
given ecosystem are positive, it is not unrealistic to assume that these rates
are lower and upper bounded by positive constants. 

Let $x,y \in \RR^n$.  In the statement and in the proof of the following
theorem, we denote as $xy, x/y, \sqrt{x}$, etc., the $\RR^n$--valued vector
obtained by performing these operations elementwise.  Given a $\RR^n$--valued
random vector $U = [U_i]_{i=1}^n$, the notations $\EE U$ and $\PP[U \geq 0]$
will respectively refer in the remainder to the vectors $[\EE U_i ]_{i=1}^n$
and $[\PP[U_i \geq 0]]_{i=1}^n$. 
\begin{theorem}
\label{th-lvsym} 
Let Hypotheses~\ref{lv-sigma} to~\ref{lv-r} hold true. Then, for each $n$
for which $\| \Sigma^{(n)} \| < 1$, the ODE~\eqref{lv} is well-defined, 
and it has a globally stable equilibrium $u^{(n)}_\star$. For the other 
values of $n$, put $u^{(n)}_\star = 0$. The distribution 
$\mu^{u^{(n)}_\star}$ is a $\cP_2(\RR)$--valued random variable on the
probability space where $\Sigma^{(n)}$ is defined. 

Let $\xi^{(n)} \sim \cN(0, I_n)$. Then, for each $n$ large enough, the system 
of equations 
\begin{subequations}
\label{sys-lv} 
\begin{align}
p &= V^{(n)} \diag(1+\zeta)^{2} \   
  \EE \left(\sqrt{p} \xi^{(n)} + r^{(n)} \right)_+^2 \label{var-lcp} \\
\zeta &= \diag(1+\zeta) V^{(n)} \diag (1+\zeta) \  
  \PP\left[\sqrt{p} \xi^{(n)} + r^{(n)} \geq 0\right]. \label{ons-lcp} 
\end{align}
\end{subequations} 
admits an unique solution $(p,\zeta) = (p^{(n)},\zeta^{(n)}) 
 \in \RR_+^n \times [0,1]^n$. 
This solution satisfies 
\[
\sup_n \| p^{(n)} \|_\infty < \infty. 
\]
Let $Y^{(n)}$ be the Gaussian vector 
\[
Y^{(n)} = \Bigl[ Y^{(n)}_i \Bigr]_{i=1}^n = 
\left( 1 + \zeta^{(n)} \right) \Bigl( \sqrt{p^{(n)}} \xi^{(n)} 
  + r^{(n)} \Bigr) , 
\]
and define the deterministic measure $\bs\mu^{(n)} \in \cP_2(\RR)$ as 
$\bs\mu^{(n)} = \mcL( ( Y^{(n)}_{\theta^{(n)}} )_+ )$, where 
$\theta^{(n)}$ is a uniformly distributed random variable on the set $[n]$,
which is independent of $\xi^{(n)}$. Then 
\[
\bs d_2\left( \mu^{u^{(n)}_\star}, \bs\mu^{(n)} \right) 
 \toprobalong 0 .
\] 
\end{theorem}  

\begin{remark}
\label{lv-tight} 
Since $\sup_n \| p^{(n)} \|_\infty < \infty$, the sequence of measures
$(\bs\mu^{(n)})$, which are mixtures of truncated Gaussians with bounded means
and variances, is a pre-compact sequence in the space $\cP_2(\RR)$.  Therefore,
for each sub-sequence of $(n)$, there is a further sub-sequence $(n')$ and a
measure $\bs\nu \in \cP_2(\RR)$ such that $\mu^{u^{(n')}_\star} \to \bs\nu$ in
probability in $\cP_2(\RR)$.  
\end{remark} 

\begin{remark}
Write $u^{(n)}_\star = [ u^{(n)}_{\star,i} ]_{i=1}^n$. By 
Theorem~\ref{th-lvsym}, 
\[
\frac 1n \sum_{i\in[n]} \varphi(u^{(n)}_{\star,i}) - 
 \EE \varphi((Y^{(n)}_{\theta^{(n)}})_+) \toprobalong 0 
\]
for each continuous function $\varphi : \RR \to [0,1]$ such that 
$\varphi(0) = 0$. With this at hand, we can consider the positive
number $\gamma^{(n)} = \PP [ Y^{(n)}_{\theta^{(n)}} > 0 ]$ as an 
approximation of the proportion of surviving species at the equilibrium, in
the sense that $\gamma^{(n)} \simeq 
n^{-1} \sum_{i\in[n]} \varphi(u^{(n)}_{\star,i})$ when $\varphi$ is chosen as 
a ``continuous approximation'' of the real function $\psi(x) = \bs 1_{x > 0}$. 
Of course, this does not imply that 
$\gamma^{(n)} - n^{-1} \sum_i \bs 1_{u^{(n)}_{\star,i} > 0}$ converges to 
zero in probability. 

Note that $\liminf_n \gamma^{(n)} > 1/2$ by Hypothesis~\ref{lv-r} and by the
bound $\sup \| p^{(n)} \|_\infty < \infty$, which implies that the proportion
of species that survive at the equilibrium for large $n$ is lower bounded by 
$1/2$ with a positive gap. 
\end{remark} 

The following corollary to Theorem~\ref{th-lvsym} shows that when the variance
profile matrix is $V^{(n)} = \alpha n^{-1} 1_n 1_n^\T$ with $\alpha \in (0,
1/4)$, then, the large--$n$ behavior of the empirical measure
$\mu^{u^{(n)}_\star}$ is the one predicted by
\cite{bun-17,gal-18,akj-etal-(arxiv)22}. Remember that we are not restricted 
in the present paper to the Gaussian case. 
\begin{corollary}
In the setting of the previous theorem, assume that $V^{(n)} = \alpha n^{-1}
1_n 1_n^\T$ with $\alpha\in(0,1/4)$.  Let $\uxi \sim\cN(0,1)$, and let
$\theta^{(n)}$ be independent of $\uxi$ and uniformly distributed on the set
$[n]$.  Then, there is an unique couple $(\underline p^{(n)},
\underline\zeta^{(n)}) \in \RR_+ \times [0,1]$ such that, dropping the
superscript $^{(n)}$,  
\begin{align*}
\underline p &= \alpha (1 + \underline\zeta)^2 
   \EE_{\theta} \EE_{\uxi} \bigl[ \bigl(\sqrt{\underline p} \, 
   \uxi + r_{\theta} \bigr)_+^2 \bigr],  \\
\underline \zeta &= \alpha (1 + \underline\zeta)^2 
   \EE_{\theta} \PP_{\uxi} \bigl[ \sqrt{\underline p} \, \uxi + 
   r_\theta \geq 0 \bigr] . 
\end{align*}
Moreover, 
\[
\bs d_2\left( \mu^{u_\star}, 
  \mcL\left( (1 + \underline\zeta) 
   \bigl( \sqrt{\underline p} \, \uxi + r_\theta \bigr)_+ \right) \right) 
 \toprobalong 0 . 
\]
\end{corollary} 
To prove this corollary, we notice that due to the structure of $V$, the 
solution of the system~\eqref{sys-lv} is of the form 
$(p^{(n)},\zeta^{(n)})  = (\underline p^{(n)} 1_n, \underline\zeta^{(n)} 1_n)$
where $(p^{(n)},\zeta^{(n)})$ is the solution of the system shown in the
statement. The result follows. \\ 

We now turn to our proofs.  In the remainder of this paper, the notations $C,
c >0$ refer to constants that can change from line to line.  Some of the
statements hold true for $n$ large enough instead of holding for each $n$. This
will not be specified.

\section{Proofs of Theorems~\ref{th-amp-lip} and~\ref{th-amp-gal}} 
\label{prf-amp}

The first and most important part of this section is common to the proofs of
Theorems~\ref{th-amp-lip} and~\ref{th-amp-gal}. The idea is to consider an AMP
sequence where the function $h(\cdot, \ueta,t)$ is replaced with a polynomial,
and to apply to this sequence the combinatorial technique of Bayati, Lelarge
and Montanari in~\cite{bay-lel-mon-15}.  

In the remainder, $c > 0$ and $C > 0$ denote absolute constants that 
can change from a display to another.  

\subsection{AMP with polynomial activation functions} 

\subsubsection{The setting} 

Let $d > 0$ be a fixed integer. Consider a function $p^{(n)} : \RR \times [n]
\times \NN \to \RR$ with the property that $p^{(n)}(\cdot, i, t)$ is a 
polynomial with a degree bounded by $d$, and is thus written as  
\[
p^{(n)}(u,i, t) = \sum_{\ell = 0}^d \alpha_{\ell}^{(n)}(i,t) u^\ell .
\]
Starting with the deterministic vector $\cx^{(n),0} = x^{(n),0}$, 
our purpose here is to study the AMP recursion in the vectors 
$\cx^{(n),1}, \cx^{(n),2},\ldots$ given as 
\begin{equation} 
\label{amp-posc} 
\cx^{(n),t+1} = 
 W^{(n)} p^{(n)}(\cx^{(n),t},\cdot,t) - 
 \diag \Bigl((W^{(n)})^{\odot 2} \partial p^{(n)}(\cx^{(n),t},\cdot,t) \Bigr)  
   p(\cx^{(n),t-1},\cdot,t-1) ,  
\end{equation} 
where, writing $x = [x_i]_{i=1}^n$, we set 
$p^{(n)}(x,\cdot,t) = [ p^{(n)}(x_i,i,t) ]_{i=1}^n$ and 
$\partial p^{(n)}(x,\cdot,t) = [\partial p^{(n)}(\cdot, i,t)|_{x_i} ]_{i=1}^n$,
and we consider that $p^{(n)}(\cdot,\cdot,-1) \equiv 0$. 

The general idea goes as follows.  With the polynomial activation function
$p^{(n)}(\cdot, i, t)$,  the elements of the vectors $\cx^{(n),t}$ are random
variables which are sums of monomials in the elements of $W^{(n)}$ with indices
indexed by labelled trees.  Our purpose will be to compute the moments of these
random variables.  Due to the presence of the ``correction'' so-called Onsager
term $\diag(\cdot) p(\cdot)$, the effect of the paths with the so-called
backtracking edges will be cancelled, rendering these moments close to the
Gaussian moments provided by the state evolution equations below.  A nice
intuitive explanation of the main proof idea is provided in
\cite[\S~2]{bay-lel-mon-15}. 

Let us provide the expressions of the state evolution equations for 
Algorithm~\eqref{amp-posc}. 
Consider the Gaussian sequence $(\cbZ^{(n),t})_{t\in\NN_*}$ described as 
follows. Writing $\cbZ^{(n),t} = \left[ \cZ^{(n),t}_1, \ldots, \cZ^{(n),t}_n 
\right]^\T$ and $\cZ^{(n)}_i = (\cZ^{(n),1}_i, \cZ^{(n),2}_i, \ldots)$, the
sequences $\{\cZ^{(n)}_i\}_{i=1}^n$ are centered, Gaussian, and independent, 
and the covariance matrix $\chR^{(n),t}_i$ of the vector 
$\vec \cZ^{(n),t}_i = [ \cZ^{(n),1}_i, \ldots, \cZ^{(n),t}_i]^\T$ is 
constructed recursively in $t$ as follows. 
\[
\chR^{(n),1}_i = \sum_{l\in[n]} s^{(n)}_{il} 
 p^{(n)}(x^{(n),0}_l, l, 0)^2 , 
\]
and 
\[
\chR^{(n),t+1}_i = \sum_{l\in[n]} s^{(n)}_{il} 
 \EE \begin{bmatrix} 
  p^{(n)}(x^{(n),0}_l, l, 0) \\
  p^{(n)}(\cZ^{(n),1}_l, l, 1) \\
  \vdots \\ 
  p^{(n)}(\cZ^{(n),t}_l, l, t) \end{bmatrix} 
 \begin{bmatrix} 
  p^{(n)}(x^{(n),0}_l, l, 0) & 
  p^{(n)}(\cZ^{(n),1}_l, l, 1) & 
  \cdots & p^{(n)}(\cZ^{(n),t}_l, l, t) \end{bmatrix} . 
\]
With this at hand, the main result of this subsection reads as follows. 
\begin{proposition}
\label{amp-pscal} 
Let Assumptions~\ref{ass-X}--\ref{x0} hold true, and consider the AMP 
algorithm~\eqref{amp-posc}. Assume that for each integer $t\geq 0$, 
\begin{equation}
\label{bnd-ag} 
\sup_n \max_{\ell\leq d} \max_{i\in[n]} | \alpha_{\ell}^{(n)}(i,t) | < \infty.
\end{equation} 
Then, 
\[
\forall t > 0, \quad 
\sup_n \max_{i\in[n]} \| \chR^{(n),t}_i \| < \infty .
\]
Moreover, 
\begin{equation}
\label{mom-cx} 
\forall t > 0, \ \forall m \in \NN_*, \quad 
\sup_n \max_{i\in[n]} \EE |\cx^{(n),t}_i|^m < \infty. 
\end{equation} 
Given $t\in\NN_*$, let $\psi^{(n)} : \RR^t \times [n] \to \RR$ be such that 
$\psi^{(n)}(\cdot, l)$ is a multivariate polynomial with a bounded degree and 
bounded coefficients as functions of $(l,n)$. Let $\cS^{(n)} \subset [n]$ be 
such that $| \cS^{(n)} | \leq C K^{(n)}$. Then, 
\begin{subequations}
\label{cvg-pscal} 
\begin{align} 
&\frac{1}{|\cS^{(n)}|} \sum_{i\in \cS^{(n)}} 
 \psi^{(n)}(\cx^{(n),1}_i,\ldots,\cx^{(n),t}_i,i) - 
   \EE\psi^{(n)}(\cZ^{(n),1}_i, \ldots, \cZ^{(n),t}_i,i) \toprobalong 0 , 
        \quad \text{and} \label{cvg-small}  \\ 
&\frac{1}{n} \sum_{i\in[n]} 
 \psi^{(n)}(\cx^{(n),1}_i,\ldots,\cx^{(n),t}_i,i) - 
   \EE\psi^{(n)}(\cZ^{(n),1}_i, \ldots, \cZ^{(n),t}_i,i) \toprobalong 0 . 
\end{align}
\end{subequations} 
\end{proposition} 

The proof of this proposition will revolve around a version of our AMP sequence
where the $\RR^n$--valued vectors $\cx^{(n),t}$ will be replaced with
$\RR^{n\times q}$--valued matrices, with $q > 0$ being a given integer.  
At the same time, the test function $\psi^{(n)}(\cdot,i)$ acting on the 
iterates $\cx^{(n),1}_i,\ldots,\cx^{(n),t}_i$ above will be replaced with
a test function acting on the $t^{\text{th}}$ iterate only. In all the 
remainder, a vector $\bx \in \RR^q$ will be written as 
$\bx = [ x(1),\ldots, x(q) ]^\T$. Consider the function 
\[
\begin{array}{lcccc} 
f^{(n)} &:& \RR^q \times [n] \times \NN &\longrightarrow & \RR^q \\
 & & (\bu,l,t) & \longmapsto &  
  \begin{bmatrix} f_1^{(n)}(\bu,l,t) \\ \vdots \\ f_q^{(n)}(\bu,l,t)
 \end{bmatrix} , 
\end{array}
\]
which is a polynomial in $\bu$, with degree bounded by $d$, written as 
\[
f_r^{(n)}(\bu,l,t) = \sum_{i_1+\cdots+i_q = 0}^d 
 \alpha^{(n)}_{i_1,\ldots,i_q}(r,l,t) \prod_{s=1}^q u(s)^{i_s}. 
\]
Starting with the deterministic $n$-uple 
$(\bcx^{(n),0}_1,\ldots, \bcx^{(n),0}_n)$ with $\bcx^{(n),0}_i \in\RR^q$, 
the AMP recursion in $t$ will provide us at Iteration $t+1$ with the 
$n$-uple $(\bcx^{(n),t+1}_1,\ldots, \bcx^{(n),t+1}_n)$, where 
$\bcx^{(n),t+1}_i = [ x^{(n),t+1}_{i}(1), \ldots, x^{(n),t+1}_{i}(q)]^\T$
is given as 
\begin{multline} 
\label{ampol} 
x^{(n),t+1}_{i}(r) = 
 \sum_{l\in[n]} W_{i,l}^{(n)} f_r^{(n)}(\bcx^{(n),t}_l,l,t) \\ 
- \sum_{s\in[q]} f_s^{(n)}(\bcx^{(n),t-1}_i,i,t-1)
  \sum_{l\in[n]} (W_{i,l}^{(n)})^2 
  \frac{\partial f_r^{(n)}}{\partial x(s)} 
   (\bcx^{(n),t}_l,l,t) , \quad r \in [q], 
\end{multline} 
with $f(\cdot, \cdot, -1) \equiv 0$. 

We also introduce the following sequence of Gaussian $\RR^{nq}$--valued
random vectors $(U^{(n),t})_{t\in\NN_*}$. Writing 
\[
U^{(n),t} = \begin{bmatrix} U^{(n),t}_1 \\ \vdots \\ U^{(n),t}_n
\end{bmatrix}, 
\]
the $\{ U^{(n),t}_i \}_{i\in[n]}$ are $\RR^q$--valued independent 
Gaussian random vectors, $U^{(n),t}_i \sim \cN\left( 0, Q^{(n),t}_i \right)$, 
and the covariance matrices $Q^{(n),t}_i$ are recursively defined in $t$ as 
follows. Starting with $t=1$, we set 
\[
Q^{(n),1}_i = \sum_{l\in[n]} s_{i l}^{(n)} 
 f^{(n)}(\bcx^{(n),0}_l, l, 0)
 f^{(n)}(\bcx^{(n),0}_l, l, 0)^\T \quad \text{for} \ i \in [n]. 
\]
Given $\{ Q^{(n),t}_i \}_{i\in[n]}$, we set 
\[
Q^{(n),t+1}_i = \sum_{l\in[n]} s_{il}^{(n)} 
 \EE f^{(n)}(U^{(n),t}_l, l, t) f^{(n)}( U^{(n),t}_l, l, t)^\T 
  \quad \text{for} \ i \in [n]. 
\] 
The correlations between the elements of $U^{(n),t}$ and $U^{(n),s}$ for 
$t\neq s$ are irrelevant to our purpose. 

Given $\bx \in \RR^q$ and a multi-index $\bm = \begin{bmatrix}
m(1),\ldots, m(q) \end{bmatrix} \in \NN^q$, we write in the sequel 
\[
{\bs x}^{\bs m} = \prod_{r\in[q]} x(r)^{m(r)} . 
\]
\begin{proposition}
\label{amp-vec} 
Let Assumptions~\ref{ass-X} and \ref{ass-A} hold true, and consider the 
iterative algorithm\eqref{ampol}. Assume that for each $t > 0$, there is a 
constant $C > 0$ such that 
\begin{equation}
\label{alpha-vec} 
\left| \alpha^{(n)}_{i_1,\ldots,i_q}(r,l,t) \right| \leq C, 
\end{equation} 
and furthermore, 
\begin{equation}
\label{bnd-x0} 
\sup_n \max_{i\in[n]} \left\| \bcx^{(n),0}_{i} \right\|  < \infty. 
\end{equation} 
Then, 
\begin{equation} 
\label{Qfini} 
\forall t > 0, \quad 
\sup_n \max_{i\in[n]} \| Q^{(n),t}_i \| < \infty . 
\end{equation} 
Moreover, 
\begin{equation}  
\label{Exbnd} 
\forall t > 0, \ \forall \bm \in \NN^q, \quad 
\sup_n \max_{i\in[n]} \EE | (\bcx^t_i)^\bm | < \infty. 
\end{equation} 
Let $\psi^{(n)} : \RR^q \times [n] \to \RR$ be such that 
$\psi^{(n)}(\cdot, l)$ is a multivariate polynomial with a bounded 
degree and bounded coefficients as functions of $(l,n)$. 
Let $\cS^{(n)} \subset [n]$ be a non empty set such that 
$| \cS^{(n)} | \leq C K^{(n)}$. Then, 
\begin{subequations}
\label{cvg-vec} 
\begin{align} 
&\frac{1}{K^{(n)}} \sum_{i\in \cS^{(n)}} 
 \psi(\bcx^{(n),t}_i,i) - \EE\psi(U^{(n),t}_i,i) \toprobalong 0 , 
        \quad \text{and} \\
&\frac{1}{n} \sum_{i\in[n]} 
 \psi(\bcx^{(n),t}_i,i) - \EE\psi(U^{(n),t}_i,i) \toprobalong 0 . 
\end{align}
\end{subequations} 
\end{proposition} 

The proof of Proposition~\ref{amp-vec} is an adaptation of the approach
of~\cite{bay-lel-mon-15} to the structure of the variance profile at interest
in this paper.  For self-containedness, we reproduce large parts of the proof
of~\cite{bay-lel-mon-15}, putting the focus on the parts of this proof where
these adaptations are necessary. 

In the remainder, the superscript $^{(n)}$ will be often omitted for notational
simplicity.

\subsubsection{Proof of Proposition~\ref{amp-vec}} 

Let us quickly prove that $\sup_n \max_{i\in[n]} \left\| Q^t_i \right\| < 
\infty$ for each integer $t > 0$. For $t=1$, this
is a consequence of Assumption~\ref{ass-A} and the bounds~\eqref{alpha-vec}
and~\eqref{bnd-x0}. Assume the result is true for $t$. Then, using the
bound~\eqref{alpha-vec} again, standard Gaussian derivations show that
$\sup_n\max_i \left\| \EE f(U^{t}_i, i, t) f(U^{t}_i, i, t)^\T \right\| < 
\infty$. Then, using Assumption~\ref{ass-A}, we
obtain from the expression of $Q^{t+1}_i$ above that 
$\sup_n\max_i\left\| Q^{t+1}_i \right\| < \infty$. 

Of importance in the proof of Proposition~\ref{amp-vec} are the sequences 
issued from the so-called ``non-backtracking'' iterations. Given any $i,j\in
[n]$ with $i\neq j$, define the set of $\RR^q$--valued vectors $\{ \bz^0_{i\to
j}, \ i,j \in [n], \ i\neq j \}$ as $\bz^0_{i\to j} = \bcx^0_i$. Assuming that
the $\RR^q$--valued vectors $\{ \bz^t_{i\to j}, \ i,j \in [n], \ i\neq j \}$
are defined, the vectors $\bz^{t+1}_{i\to j}$ for $i\neq j$ are given as 
\begin{equation}
\label{zij(r)} 
z^{t+1}_{i\to j}(r) = \sum_{l\in[n]\setminus\{j\}} 
  W_{il} f_r(\bz^t_{l\to i}, l, t) 
\end{equation} 
(here, we implicitly consider that in the summation over $l$, the case 
$l= i$ is excluded because $W_{ii} = 0$). 
Having the vectors $\{ \bz^t_{i\to j}, \ i,j \in [n], \ i\neq j \}$ at hand, we
define the $\RR^q$--valued vectors $\{ \bz_i^{t+1}, \ i \in [n] \}$ as 
\begin{equation}
\label{zi(r)}
z^{t+1}_{i}(r) = \sum_{l\in[n]} W_{il} f_r(\bz^t_{l\to i}, l, t) . 
\end{equation} 

For each $n$, let us now consider an i.i.d.~sequence
$(W^{(n),t})_{t=0,1,\ldots}$ of symmetric $n\times n$ matrices such that
$W^{(n),t} \eqlaw W^{(n)}$.  We define the vectors $\by^t_{i\to j}$ and
$\by^t_i$ recursively in $t$ similarly to what we did for the vectors
$\bz^t_{i\to j}$ and $\bz^t_i$, with the difference that we now replace the
matrix $W$ with the matrix $W^t$ at step $t$. More precisely, we set
$\by^0_{i\to j} = \bcx^0_i$ for each $i,j \in [n]$ with $i\neq j$; Given $\{
\by^t_{i\to j}, \ i,j \in [n], \ i\neq j \}$, we set  
\[
y^{t+1}_{i\to j}(r) = \sum_{l\in[n]\setminus\{j\}} 
  W_{il}^t f_r(\by^t_{l\to i}, l, t) ,  \quad i \neq j. 
\]
Also, 
\[
y^{t+1}_{i}(r) = \sum_{l\in[n]} W_{il}^t f_r(\by^t_{l\to i}, l, t) . 
\]

\paragraph{The tree structure.}

The pivotal object in the proof is a rooted and labelled tree $T$ with the
following structure (we paraphrase here \cite[\S 4.2]{bay-lel-mon-15} with the
same notations).  Write $T = (V(T), E(T))$ where $V(T)$ and $E(T)$ are the sets
of vertices and edges of $T$ respectively.  We also denote as $L(T)$ the set of
the leaves of $T$.  Denote as $\circ \in V(T)$ the root of the tree, and let
$|u|$ be the distance of a vertex $u$ to $\circ$.  We consider that the edges
are oriented towards the root.  Thus, considering the edge $(u\to v) \in E(T)$,
we have that $v = \pi(u)$, where $\pi(u)$ the parent of the node $u$.  We
assume that the root has only one child (thus, $T$ is a so called planted tree).
Each vertex other than the root and the leaves can have up to $d$ children,
thus, its degree is bounded by $d+1$. By definition, the degree of a leaf is
one, its only neighbor being its parent.  

The tree $T$ is a labelled tree. We now describe this labelling.  The label of
the root is an integer $\ell(\circ) \in [n]$. The label of a vertex $v$ which
is neither the root nor a leaf is a couple $(\ell(v), r(v)) \in [n] \times
[q]$. The label of a leaf $v$ is $(\ell(v), r(v), v[1], \ldots, v[q]) \in [n]
\times [q] \times \{ (a_1, \ldots, a_q) \in \NN^q, \ a_1+\cdots+a_q \leq d \}$.
If $v\in L(T)$ and $|v| \leq t-1$, then we set $v[1] = \cdots = v[q]) = 0$. 

The integer $\ell(v)$ is denoted as the ``type'' of the vertex $v$, and $r(v)$
is the ``mark'' of this vertex when it exists. 

For a vertex $u \in V(T) \setminus L(T)$, we denote as $u[r]$ the number of
children of $u$ with the mark $r$. The children of $u \in V(T)\setminus L(T)$
are ordered with respect to their mark. Specifically, the labels of the
children of $u$ are $(\ell^1,1), \ldots, (\ell^{u[1]}, 1), (\ell^{u[1]+1}, 2),
\ldots, (\ell^{u[1]+\cdots+u[q]}, q)$

Specific classes of these trees will be of importance. Following the
definitions and notations of \cite{bay-lel-mon-15}, we introduce the following
families of trees: 
\begin{itemize} 
\item $\overline\cT^t$ is the set of labelled trees as above, with depth 
 $t$ at most. 
\item $\cT^t \subset \overline\cT^t$ is the subset that additionally satisfies 
the following so-called non-backtracking condition: if $v_1 = \circ, v_2, 
 \ldots, v_k$ is a path starting from $\circ$, \emph{i.e.}, 
 $v_i = \pi(v_{i+1})$, then the corresponding sequence of types is 
non-backtracking. This means that for each $i \in [k-2]$, the three types
$\ell(v_i)$, $\ell(v_{i+1})$ and $\ell(v_{i+2})$ are distinct. 
\item $\cT^t_{i\to j}(r) \subset \cT^t$ is the subset of trees in $\cT^t$ for
which the type of the root is $i$, the type of the child $v$ of the root 
satisfies $\ell(v) \not\in \{ i, j \}$, and the mark of $v$ is $r(v) = r$.  
\item $\cT^t_{i}(r) \subset \cT^t$ is the subset of trees in $\cT^t$ for
which the type of the root is $i$, the type of the child $v$ of the root 
satisfies $\ell(v) \neq i$, and the mark of $v$ is $r(v) = r$.  
\end{itemize}

Given a tree $T$, we also set 
\begin{align*}
W(T) &= \prod_{(u\to v) \in E(T)} W_{\ell(u) \ell(v)}, \\
\Gamma(T, \balpha, t) &= \prod_{(u\to v) \in E(T)} 
  \alpha_{u[1],\ldots,u[q]}(r(u), \ell(u), t - |u|), \\
x(T) &= \prod_{v\in L(T)} \prod_{s\in[q]} 
  \left( x^0_{\ell(v)}(s)\right)^{v[s]} .
\end{align*}

\begin{lemma}[Lemma 1 of \cite{bay-lel-mon-15}]
\label{ztree} 
\begin{align}
z^t_{i\to j}(r) &= \sum_{T\in \cT^t_{i\to j}(r)} W(T) \Gamma(T,\balpha, t) 
 x(T), \nonumber  \\
z^t_{i}(r) &= \sum_{T\in \cT^t_{i}(r)} W(T) \Gamma(T,\balpha, t) x(T) . 
  \nonumber 
\end{align}
\end{lemma}
This lemma is a structural lemma which proof is not impacted by our specific 
variance profile. 

The following notions will be needed below. Define the set 
\[
\cK = \left\{ \{i,j\} \subset [n], \ s_{ij} > 0 \right\} . 
\]
For each $i\in [n]$, we define the section 
\[
\cK_i = \left\{ j \in [n], \ s_{ij} > 0 \right\} . 
\]
The next proposition shows that in the large dimensional regime, the joint
moments of the elements of a vector $\bz^t_i$ issued from the non-backtracking
iterations depend for large $n$ only on the first two moments of the elements
of $W$.  
\begin{proposition}
[adaptation of Proposition 1 of \cite{bay-lel-mon-15}]
\label{z-tz} 
Let $\{\tX_{ij}\}_{1 \leq i < j \leq n}$ be a family of independent random
variables satisfying Assumption~\ref{ass-X}, with distributions not 
necessarily identical to their analogues $X_{ij}$. Let $\tW$ be the symmetric 
matrix constructed similarly to $W$, but with the $X_{ij}$ replaced with the 
$\tX_{ij}$.
Starting with the set of $\RR^q$--valued vectors $\{ \tbz^0_{i\to j}, \ i,j \in
[n], \ i\neq j \}$ given as $\tbz^0_{i\to j} = \bcx^0_i$, define the vectors
$\tbz_i^t \in \RR^q$ by the recursion~\eqref{zij(r)} and the
equation~\eqref{zi(r)}, where $W$ is replaced with $\tW$.  Then, for each
$t\geq 1$ and each $\bm \in \NN^q$, there exists $C$ such that for each
$i\in[n]$, 
\[
\left| \EE (\bz^t_i)^\bm - \EE (\tbz^t_i)^\bm \right| \leq C K^{-1/2}  .
\]
\end{proposition} 
\begin{proof}
For simplicity, we restrict the proof to the case where the multi-index $\bm$
satisfies $m(r) = m$ for some integer $m > 0$ and $m(s) = 0$ for 
$s \in [q]\setminus \{ r \}$. By Lemma~\ref{ztree}, we have 
\[
\EE (z^t_i(r))^m = \sum_{T_1,\ldots, T_m \in \cT^t_{i}(r)} 
 \left( \prod_{k=1}^m \Gamma(T_k,\balpha, t) \right) 
 \EE \left[ \prod_{k=1}^m W(T_k) \right] 
  \prod_{k=1}^m x(T_k) . 
\]
For a tree $T$ and $j,l \in [n]$, define 
\[
\phi(T)_{j l} = \left| \left\{ (u\to v) \in E(T), \ \{ \ell(u),\ell(v) \} = 
  \{j, l\} \right\} \right| .  
\]
There is an integer constant $C_E = C_E(d,t,m)$ that bounds the total number
of edges in the trees $T_1,\ldots, T_m \in \cT^t_{i}(r)$, leading to  
\[
\sum_{k\in[m]} \sum_{j<l} \phi(T_k)_{jl} \leq C_E . 
\]
Given an integer $\frak m \in [C_E]$, let 
\begin{align*}
\cA_i(\frak m) = \Bigl\{ &  (T_1,\ldots, T_m) \, : \, T_k \in \cT^t_{i}(r) \  
  \text{for all} \ k \in [m] ,  \\  
 &\forall j < l, \ \sum_{k\in[m]} \phi(T_k)_{jl} \neq 1 , \\ 
 &\forall j < l, \ \sum_{k\in[m]} \phi(T_k)_{jl} > 0 \ \Rightarrow \{j,l\} \in 
 \cK,  \\ 
 &\sum_{k\in[m]} \sum_{j<l} \phi(T_k)_{jl} = \frak m \Bigr\}. 
\end{align*} 
Then, since the elements of $W$ beneath the diagonal are centered and 
independent, 
\begin{equation}
\label{ztm} 
\EE z^t_i(r)^m = \sum_{\frak m=1}^{C_E} \sum_{(T_1,\ldots, T_m) \in \cA_i(\frak m)}  
 \left( \prod_{k=1}^m \Gamma(T_k,\balpha, t) \right) 
 \left( \prod_{k=1}^m x(T_k)  \right) 
 \EE \left[ \prod_{k=1}^m W(T_k) \right] . 
\end{equation} 
Notice that the contributions of the $m$--uples of trees in 
$\cA_i(\frak m)$ for which 
\[
\forall j < l, \ 
 \sum_{k\in[m]} \phi(T_k)_{jl} \in \{ 0, 2 \} 
\]
are the same for $\EE z^t_i(r)^m$ and $\EE \tilde z^t_i(r)^m$ by the 
assumptions on the matrices $W$ and $\tW$. Thus, defining the set 
\[
\widecheck\cA_i(\frak m) = \Bigl\{ (T_1,\ldots, T_m) \in \cA_i(\frak m) \ : \ 
 \exists j < l, \ \sum_{k\in[m]} \phi(T_k)_{jl} \geq 3 \Bigr\}, 
\]
the proposition will be proven if we prove that for all $\frak m\in [C_E]$, the
real number 
\[
\xi_{\frak m} = 
 \sum_{(T_1,\ldots, T_m) \in \widecheck\cA_i({\frak m})}  
 \left( \prod_{k=1}^m \Gamma(T_k,\balpha, t) \right) 
 \left( \prod_{k=1}^m x(T_k)  \right) 
 \EE \left[ \prod_{k=1}^m W(T_k) \right] 
\]
satisfies $|\xi_{\frak m} | \leq C K^{-1/2}$. 

Using the bounds~\eqref{alpha-vec} and~\eqref{bnd-x0} provided in the 
statement of Proposition~\ref{amp-vec}, it is obvious that 
$\prod_{k=1}^m \Gamma(T_k,\balpha, t)$ and $\prod_{k=1}^m x(T_k)$ are bounded. 

Since $| \EE W_{jl}^s | \leq C K^{-s/2}$ for each integer $s > 0$ by
Assumptions~\ref{ass-X} and~\ref{ass-A}, for each $(T_1,\ldots, T_m) \in
\widecheck\cA_i({\frak m})$, we have 
\[
\left| \EE \prod_{k=1}^m W(T_k) \right| = 
 \prod_{j<l} \Bigl| \EE W_{jl}^{\sum_k \phi(T_k)_{jl}} \Bigr| 
 \leq C K^{-{\frak m} / 2}. 
\]
To complete the proof, we shall show that 
\begin{equation}
\label{x(T)}
 \left| \widecheck\cA_i({\frak m}) \right| \leq C K^{({\frak m} - 1) / 2} . 
\end{equation}

From a $m$--uple $(T_1,\ldots, T_m) \in \widecheck\cA_i({\frak m})$, let us
construct a graph $G = \bG(T_1,\ldots, T_m)$ as follows. The graph $G$ is the
rooted, undirected, labelled, and unmarked graph obtained by merging the $m$
trees $T_1,\ldots, T_m$ and by identifying the nodes that have the same type;
This common type will be the label of the obtained node in $G$.  The root node
of $G$ will be the node obtained by merging the roots of the trees $T_1,\ldots,
T_m$ (remember that they all have the same type $i$). The other nodes are
numbered, say, in the increasing order of their labels.  The edges of $G$ are
furthermore unweighted.  

The number of edges of $G$ is 
\[
 | E(G) | = \sum_{j<l} \1_{\sum_k \phi(T_k)_{jl} > 0}. 
\]
Remember that when $\sum_k \phi(T_k)_{jl} > 0$, this sum is $\geq 2$, and 
for some $j<l$, it is $\geq 3$. Consequently, 
$2 | E(G) | + 1 \leq \sum_{j<l} \sum_k \phi(T_k)_{jl}$, in other words, 
\[
 | E(G) | \leq \frac{{\frak m} - 1}{2}. 
\]
Note that since $G$ is connected as being obtained through the merger of
trees, $| V(G) | \leq | E(G) | +1$. Thus, $| \{ v \in G, v \neq \degree \} |
 \leq ({\frak m} - 1) / 2$. Also, by construction, $G$ satisfies the following 
property: 
\[
\{ u, v \} \in E(G) \ \Rightarrow \ \ell(u) \in \cK_{\ell(v)} .
\]
We shall denote as $\cG^{\frak m}_i$ the set of rooted, undirected, labelled and
connected graphs such that $\ell(\degree) = i$, $| E(G) | \leq \frac{{\frak m} -
1}{2}$, and the last property is satisfied. We denote as $\cR^{\frak m}$ the set of
all the elements of $\cG^{\frak m}_i$ but without the labels. Given a graph 
$G \in \cG^{\frak m}_i$, let us denote as $\bar G = \bU(G) \in \cR^{\frak m}$ the
unlabelled version of $G$.  With these notations, we have
\begin{equation}
\label{xcheck} 
 \left| \widecheck\cA_i({\frak m}) \right| = 
 \sum_{\bar G \in \cR^{\frak m}} \ 
 \sum_{\substack{G \in \cG^{\frak m}_i \, : \\\bU(G) = \bar G}} \ 
 \left| \left\{ 
 (T_1,\ldots, T_m) \in \widecheck\cA_i({\frak m}) \, : \, 
       \bG(T_1,\ldots, T_m)= G \right\} \right| . 
\end{equation} 
For each graph $G$, we have 
\begin{equation}
\label{TG} 
 \left| \left\{ (T_1,\ldots, T_m) \in \widecheck\cA_i({\frak m}) \ : \ 
       \bG(T_1,\ldots, T_m)= G \right\} \right| \leq C,
\end{equation} 
where $C = C(d,t,m)$ is independent of $G$. Our purpose is now to show that 
\begin{equation}
\label{GKmu} 
 \left| \left\{ G \in \cG^{\frak m}_i \, : \, \bU(G) = \bar G \right\} \right|  
  \leq C K^{({\frak m}-1) / 2}. 
\end{equation} 

Given $\bar G \in \cR^{\frak m}$, denote as $\degree$ the root node of 
$\bar G$, write $M = | V(\bar G) | - 1 \leq ({\frak m} - 1) / 2$, and write
$V(\bar G) \setminus \{\degree\} = [M]$.  Recalling that $\bar G$ is connected,
let us consider a spanning tree of this graph rooted in $\degree$. Denote as
$\bs\pi(v)$ the parent of the node $v$ in this tree. Writing $j_\degree = i$,
we obtain that 
\[ 
 \left| \left\{ G \in \cG^{\frak m}_i \, : \, \bU(G) = \bar G \right\} \right|  
 \leq \left|\left\{ (j_1,\ldots, j_M) \in [n]^M \, : \, 
 \forall k \in [M], \ j_k \in \cK_{j_{\bs\pi(k)}}  \right\} \right| . 
\] 
Denoting as $L \subset [M]$ the set of the leaves of the spanning tree, we can 
write 
\begin{align} 
\left|\left\{ (j_1,\ldots, j_M) \in [n]^M \, : \, 
 \forall k \in [M], \ j_k \in \cK_{j_{\bs\pi(k)}}  \right\} \right| 
 &= \sum_{\substack{j_1,\ldots, j_M \in [n] \, :  \\
 \forall k \in [M], \ j_k \in \cK_{j_{\bs\pi(k)}} }} 1 \nonumber \\ 
&= 
\sum_{k \in [M] \setminus L} \ 
  \sum_{j_k \in \cK_{j_{\bs\pi(k)}}}  
 \Bigl( \sum_{p \in L} \sum_{j_p \in \cK_{j_{\bs\pi(p)}}}  
   1 \Bigr)  \nonumber \\ 
 &\leq C K^{|L|} 
\sum_{k \in [M] \setminus L} \ \sum_{j_k \in \cK_{j_{\bs\pi(k)}}}  1 , 
\nonumber 
\end{align} 
recalling that $|\cK_j| \leq CK$ for all $j$ by Assumption~\ref{ass-A}, 
and using the inequality $|L| K \leq K^{|L|}$ for $K\geq 2$. 
If we prune the leaves of the original spanning tree, what remains is a tree
made of the nodes that constitute the first sum above plus the root node.  We
can apply the pruning operation to the new tree as above, and iterate until
exhausting all the set $[M] = V(\bar G) \setminus \{\degree\}$. This leads to 
\[
\left|\left\{ (j_1,\ldots, j_M) \in [n]^M \, : \, 
 \forall k \in [M], \ j_k \in \cK_{j_{\bs\pi(k)}}  \right\} \right| 
 \leq C K^{M} \leq C K^{({\frak m}-1) / 2}. 
\]
Inequality~\eqref{GKmu} follows. 
 
It is furthermore easy to check that 
\[
\left| \cR^{\frak m} \right| \leq C .
\]
Getting back to Inequality~\eqref{xcheck}, and using this last inequality along
with Inequalities~\eqref{GKmu} and~\eqref{TG}, we obtain
Inequality~\eqref{x(T)}, and the proposition is established. 
\end{proof}

Let us keep the notations of the former proof. Consider the subset 
$\widetilde\cA_i({\frak m})$ of $\cA_i({\frak m})$ defined as 
\[
\widetilde\cA_i({\frak m}) = \Bigl\{ (T_1,\ldots, T_m) \in \cA_i({\frak m}) \ : \ 
 \forall j < l, \ \sum_{k\in[m]} \phi(T_k)_{jl} \in \{0,2\} \Bigr\}. 
\]
For $(T_1,\ldots, T_m) \in \widetilde\cA_i({\frak m})$, let us denote as $G =
\bG(T_1,\ldots, T_m) \in \cG^{\frak m}_i$ the graph obtained by merging these trees,
and $\bar G = \bU(G) \in \cR^{\frak m}$ the unlabelled version of $G$, as we did for
$(T_1,\ldots, T_m) \in \widecheck\cA_i({\frak m})$. As we did in~\eqref{ztm}, 
$\EE z_i^t(r)^m$ can be written as 
\begin{equation}
\label{Ezm} 
\EE z^t_i(r)^m = 
 \sum_{{\frak m}=1}^{C_E} \chi_{\frak m} +  
  \sum_{{\frak m}=1}^{C_E} \xi_{\frak m} ,  
\end{equation} 
where 
\begin{align*} 
\chi_{\frak m} &= 
 \sum_{(T_1,\ldots, T_m) \in \widetilde\cA_i({\frak m})} 
    \left( \prod_{k=1}^m \Gamma(T_k,\balpha, t) \right) 
    \left( \prod_{k=1}^m x(T_k) \right)  
    \EE \left[ \prod_{k=1}^m W(T_k) \right] \\ 
&= \sum_{\bar G \in \cR^{\frak m}} \ 
 \sum_{\substack{G \in \cG^{\frak m}_i \, : \\ \bU(G) = \bar G}} \ 
 \sum_{\substack{(T_1,\ldots, T_m) \in \widetilde\cA_i({\frak m}) \, : \\
       \bG(T_1,\ldots, T_m)= G}} 
    \left( \prod_{k=1}^m \Gamma(T_k,\balpha, t) \right) 
    \left( \prod_{k=1}^m x(T_k) \right)  
    \EE \left[ \prod_{k=1}^m W(T_k) \right],  
\end{align*} 
and where we recall from the former proof that 
$\xi_{\frak m} = \sum_{(T_1,\ldots, T_m) \in \widecheck\cA_i({\frak m})} \cdots$
satisfies $|\xi_{\frak m} | \leq CK^{-1/2}$. 

Let us further decompose $\chi_{\frak m}$ as 
\begin{align*}
\chi_{\frak m} &= 
 \sum_{\substack{\bar G \in \cR^{\frak m} \, : \\
  \bar G \text{ is a tree}}} \ 
 \sum_{\substack{G \in \cG^{\frak m}_i \, : \\ \bU(G) = \bar G}} \ 
 \sum_{\substack{(T_1,\ldots, T_m) \in \widetilde\cA_i({\frak m}) \, : \\
       \bG(T_1,\ldots, T_m)= G}} \cdots 
 \quad + \quad 
 \sum_{\substack{\bar G \in \cR^{\frak m} \, : \\
  \bar G \text{ not a tree}}} \ 
 \sum_{\substack{G \in \cG^{\frak m}_i \, : \\ \bU(G) = \bar G}}  \ 
 \sum_{\substack{(T_1,\ldots, T_m) \in \widetilde\cA_i({\frak m}) \, : \\
       \bG(T_1,\ldots, T_m)= G}} \cdots \\ 
&= \chi^{\text{T}}_{\frak m} + \chi^{\text{NT}}_{\frak m} . 
\end{align*}
Then, the contribution of the term $\chi^{\text{NT}}_{\frak m}$ is negligible: 
\begin{lemma}[see Lemma~2 of \cite{bay-lel-mon-15}]
\label{Gtree}
$\left| \chi^{\text{T}}_{\frak m} \right| \leq C$, and 
$\left| \chi^{\text{NT}}_{\frak m} \right| \leq C / K$. 
\end{lemma}
\begin{proof}
By repeating the argument of the former proof, the terms 
$\prod_{k=1}^m \Gamma(T_k,\balpha, t)$ and $\prod_{k=1}^m x(T_k)$ are both 
bounded, and the term $\EE \prod_{k=1}^m W(T_k)$ accounts for a factor of 
order $K^{-{\frak m}/2}$ in both $\chi^{\text{T}}_{\frak m}$ and $\chi^{\text{NT}}_{\frak m}$. 

Furthermore, when $(T_1,\ldots, T_m) \in \widetilde\cA_i({\frak m})$, the graph $G =
\bG(T_1,\ldots, T_m)$ satisfies $| E(G) | = {\frak m} / 2$. We further know that
$\bar G = \bU(G)$ satisfies $|V(\bar G)| \leq |E(\bar G)| +1 = |E(G)| + 1$,
with equality if and only if $G$ is a tree. Therefore, when $\bar G \in
\cR^{\frak m}$ is a tree, $|V(\bar G)| = {\frak m} / 2 + 1$. Once we set $M = | V(G)| - 1 =
{\frak m} / 2$, the argument for establishing Inequality~\eqref{GKmu} in the former 
proof can be reproduced word for word here to show that 
$|\{ G \in \cG^{\frak m}_i \, : \, \bU(G) = \bar G \}| \leq C K^{{\frak m}/2}$. 
When $\bar G \in \cR^{\frak m}$ is not a tree, $|V(\bar G)| \leq {\frak m} / 2$, and we
obtain that 
$|\{ G \in \cG^{\frak m}_i \, : \, \bU(G) = \bar G \}| \leq C K^{{\frak m}/2-1}$. Thus, 
\begin{align*}
\left| \chi^{\text{T}}_{\frak m} \right| &\leq C K^{-{\frak m}/2} K^{{\frak m}/2} = C, \\
\left| \chi^{\text{NT}}_{\frak m} \right| &\leq C K^{-{\frak m}/2} K^{{\frak m}/2-1} = C
 K^{-1},  
\end{align*} 
and the lemma is proven. 
\end{proof} 

Getting back to the expression~\eqref{Ezm} of $\EE z^t_i(r)^m$ and using this
lemma along with the bound $|\xi_{\frak m}| \leq C K^{-1/2}$, we obtain that for 
each $t > 0$ and each multi-index $\bm\in\NN^q$, there exists a constant 
$C > 0$ such that 
\begin{equation}  
\label{Ezbnd} 
\max_{i\in[n]} \EE | (\bz^t_i)^\bm | \leq C .
\end{equation} 
This bound will be needed below. 

Recall that the samples $\by^t_i$ are obtained by drawing an independent matrix
$W^t$ at each iteration in the parameter $t$. We have: 
\begin{proposition}[Proposition 2 of \cite{bay-lel-mon-15}]
\label{z-y}
For each $t\geq 1$ and each $\bm \in \NN^q$, there exists $C$ such that
for each $i\in[n]$, 
\[
\left| \EE (\bz^t_i)^\bm - \EE (\by^t_i)^\bm \right| \leq C K^{-1/2}  .
\]
\end{proposition} 
\begin{proof}
Parallelling the quantities $W(T)$ introduced above for a tree $T \in 
 \overline\cT^t$, \cite{bay-lel-mon-15} introduced the quantities 
\[
\barW(T,t) = \prod_{(u\to v) \in E(T)} W^{t-|u|}_{\ell(u),\ell(v)} . 
\]
By an easy adaptation of the proof of \cite[Lm.~1]{bay-lel-mon-15} 
(Lemma~\ref{ztree} above), we can show that 
\begin{align}
y^t_{i\to j}(r) &= \sum_{T\in \cT^t_{i\to j}(r)} \barW(T) \Gamma(T,\balpha, t) 
 x(T), \nonumber  \\
y^t_{i}(r) &= \sum_{T\in \cT^t_{i}(r)} \barW(T) \Gamma(T,\balpha, t) x(T) . 
  \nonumber 
\end{align}
Similarly to the proof of Proposition~\ref{z-tz}, we assume that $m(r) = m$
and $m(s) = 0$ for $s\in[q]\setminus \{ r \}$. 
Similarly to Equation~\eqref{ztm}, and with the same notations, it holds that
\[
\EE y^t_i(r)^m = \sum_{{\frak m}=1}^{C_E} \sum_{(T_1,\ldots, T_m) \in \cA_i({\frak m})}  
 \left( \prod_{k=1}^m \Gamma(T_k,\balpha, t) \right) 
 \left( \prod_{k=1}^m x(T_k) \right)  
 \EE \left[ \prod_{k=1}^m \barW(T_k) \right] . 
\]
As we did for $\EE z^t_i(r)^m$ above, we partition $\cA_i({\frak m})$ as 
$\cA_i({\frak m}) = \widetilde\cA_i({\frak m}) \cup \widecheck\cA_i({\frak m})$. We prove with
the same arguments that the contribution of $\widecheck\cA_i({\frak m})$ is
of order $K^{-1/2}$. Furthermore, parallelling Lemma~\ref{Gtree}, we also 
obtain that within $\widetilde\cA_i({\frak m})$ we can limit ourselves to the terms
$\bar\chi^{\text{T}}_{\frak m}$ defined as 
\[
\bar\chi^{\text{T}}_{\frak m} = 
 \sum_{\substack{\bar G \in \cR^{\frak m} \, : \\
  \bar G \text{ is a tree}}} \ 
 \sum_{\substack{G \in \cG^{\frak m}_i \, : \\ \bU(G) = \bar G}} \ 
 \sum_{\substack{(T_1,\ldots, T_m) \in \widetilde\cA_i({\frak m}) \, : \\
       \bG(T_1,\ldots, T_m)= G}} 
    \left( \prod_{k=1}^m \Gamma(T_k,\balpha, t) \right) 
    \left( \prod_{k=1}^m x(T_k) \right)   
    \EE \left[ \prod_{k=1}^m \barW(T_k) \right],  
\]
the terms for which $\bar G$ is not a tree being of order $K^{-1}$. With this
at hand, the proposition will be established once we show that
$\bar\chi^{\text{T}}_{\frak m} = \chi^{\text{T}}_{\frak m}$, where we recall that
$\chi^{\text{T}}_{\frak m}$, introduced in the last proof, has the same expression as
$\bar\chi^{\text{T}}_{\frak m}$ except that the terms $\barW(T_k)$ in the latter are
replaced with $W(T_k)$. 

Consider an arbitrary $m$--uple $(T_1,\ldots, T_m)$ in the inner sum above. We
first notice that if $\EE \prod_{k=1}^m \barW(T_k) \neq 0$, then $\EE
\prod_{k=1}^m \barW(T_k) = \EE \prod_{k=1}^m W(T_k)$.  This is due to the fact
that if $j < l$ is active in $\EE \prod_{k=1}^m \barW(T_k)$ (in the sense that
$\sum_k \phi(T_k)_{jl} \neq 0$, and thus, is equal to $2$), then the
corresponding contribution of this $j < l$ to $\EE \prod_{k=1}^m W(T_k)$ will
be exactly the same. 

The proof will then be terminated if we show that if $\EE \prod_{k=1}^m W(T_k)
\neq 0$ and $\EE \prod_{k=1}^m \barW(T_k) = 0$, then necessarily, the graph $G
= \bU(T_1,\ldots, T_m)$ will not be a tree, \emph{i.e.}, it will contain a
cycle.  Assume that $\EE \prod_{k=1}^m W(T_k) \neq 0$ and $\EE \prod_{k=1}^m
\barW(T_k) = 0$. Then, there will be an edge $\{u , v\}$ in $G$, with $\{
\ell(u), \ell(v) \} = \{ j, l \}$, but that will appear at two different
distances to $\circ$ in the trees $T_k$. Let us consider the three possible
cases where this could happen: 
\begin{enumerate}
\item This happens in the same tree, say $T_1$, and on the same path to
$\degree$. Then, due to the non-backtracking nature of $T_1$, a cycle appears
in $G$. 

\item This happens in two different trees, say $T_1$ and $T_2$.  Namely, there
exists two edges $u\to v \in E(T_1)$ and $u'\to v' \in E(T_2)$ such that $\{
\ell(u), \ell(v) \} = \{ \ell(u'), \ell(v') \}$, and $|u| \neq |u'|$.  Then,
keeping in mind the backtracking property, it is easy to observe that a cycle
is created in $G$. 

\item A similar remark can be made when this happens in the same tree but on
two different paths to the root. 
\end{enumerate}  
Thus, we have a contradiction in the three cases, and we get that 
$\bar\chi^{\text{T}}_{\frak m} = \chi^{\text{T}}_{\frak m}$. The proposition is 
proven. 
\end{proof} 

The following proposition links the joint moments of the elements of the
vectors $\bz^t_i$ with those of the vectors $\bx^t_i$ provided by the AMP
algorithm~\eqref{ampol}. 
\begin{proposition}[proposition 3 of \cite{bay-lel-mon-15}] 
\label{z-x}
For each $t\geq 1$ and each $\bm \in \NN^q$, there exists $C$ such that
for each $i\in[n]$, 
\[
\left| \EE (\bz^t_i)^\bm - \EE (\bcx^t_i)^\bm \right| \leq C K^{-1/2}  .
\]
\end{proposition} 

Recalling the bound~\eqref{Ezbnd}, the bound~\eqref{Exbnd} can be deduced from
this proposition. 

To prove this proposition, new objects need to be introduced. In a directed 
and labelled graph, 
\begin{itemize}
\item 
A backtracking path of length $3$ is a path 
$a \to b \to c \to d$ such that $\ell(a) = \ell(c)$ and $\ell(b) = \ell(d)$. 
\item A backtracking star is a structure 
$a,b \to c \to d$ where $\ell(a) = \ell(b) = \ell(d)$. 
\end{itemize}
Let $\ucU^t$ be the set of equivalence classes of trees in $\overline\cT^t$
from which the marks have been removed.  Denote as $\cB^t$ the set of trees
$T$ in $\ucU^t$ that satisfy the following additional conditions: 
\begin{itemize}
\item If $u\to v \in E(T)$, then $\ell(u) \neq \ell(v)$. 
\item There exists in $T$ at least one backtracking path of length $3$ or one
backtracking star. 
\end{itemize}
Finally, $\cB_i^t$ is the subset of trees in $\cB^t$ with the root having 
the type $i$, and such that the type of the child $v$ of the root satisfies
$\ell(v) \neq i$. 

The proof of Proposition~\ref{z-x} relies on the following structural lemma: 
\begin{lemma}[Lemma~3 of \cite{bay-lel-mon-15}] 
\label{xtree} 
\[
x^t_i(r) = z^t_i(r) + \sum_{T\in \cB^t_i} W(T) \tGamma(T,t,r) x(T) , 
\]
where $| \tGamma(T,t,r) | \leq C(d,t)$. 
\end{lemma} 

\begin{proof}[Proof of Proposition~\ref{z-x}] 
Once again, we assume that $m(r) = m$ and $m(s) = 0$ for 
$s\in[q]\setminus \{ r \}$. We write 
\[
\EE (x^t_i(r))^m  = \EE 
 \Bigl( z^t_i(r) + \sum_{T\in \cB^t_i} W(T) \tGamma(T,t,r) x(T) \Bigr)^m
\]
where $z^t_i(r)$ is given by Lemma~\ref{ztree}. With this at hand, the 
proposition will be established once we bound the terms of the type
\[
\sum_{T_1\in \cB^t_i} 
  \sum_{T_2,\ldots, T_m \in \cB^t_i \cup \cT^t_i(r)} 
 \Bigl| \EE \prod_{k=1}^m W(T_k) \Bigr| .
\]
The argument is nearly the same as in the proof of Proposition~\ref{z-tz}.
Defining the set 
\begin{align*}
\cD_i({\frak m}) = \Bigl\{ & (T_1,\ldots, T_k) \ : \  T_1 \in \cB^t_i, 
   \ T_2, \ldots, T_m  \in  \cB^t_i \cup \cT^t_i(r), \\ 
 &\forall j < l, \ \sum_{k\in[m]} \phi(T_k)_{jl} \neq 1 , \\ 
 &\forall j < l, \ \sum_{k\in[m]} \phi(T_k)_{jl} > 0 \ \Rightarrow \{j,l\} \in 
 \cK,  \\ 
 &\sum_{k\in[m]} \sum_{j<l} \phi(T_k)_{jl} = {\frak m} \Bigr\}. 
\end{align*} 
Recalling the notations of the proof of Proposition~\ref{z-tz}, we need to
show that 
\[
\delta_{\frak m} = 
 \sum_{\bar G \in \cR^{\frak m}} \ 
 \sum_{\substack{G \in \cG^{\frak m}_i \, : \\ \bU(G) = \bar G}} \ 
 \sum_{\substack{(T_1,\ldots, T_m) \in \cD_i({\frak m}) \, : \\
       \bG(T_1,\ldots, T_m)= G}} 
    \Bigl| \EE \prod_{k=1}^m W(T_k) \Bigr| 
\]
satisfies $\delta_{\frak m} \leq C K^{-1/2}$. 

As usual, $\left| \EE \prod_{k=1}^m W(T_k) \right| \leq C K^{-{\frak m} / 2}$.  We
need to bound $| \cD_i({\frak m}) |$.  To this end, we observe that since $T_1 \in
\cB^t_i$, resulting in this tree having a backtracking path or a backtracking
star, it is easy to see that the graph $\bar G$ has an edge that results from
the fusion of three edges at least. This implies that $|E(\bar G)| \leq ({\frak m} -
1) / 2$. Reusing the argument of the proof of Proposition~\ref{z-tz}, we obtain
that $| \cD_i({\frak m}) | \leq C K^{({\frak m}-1)/2}$, which shows that 
$\delta_{\frak m} \leq C K^{-1/2}$ as required. Proposition~\ref{z-x} is proven.
\end{proof}

Making use of the independence of the matrices $W^t$, we now show that the
joint moments of the elements of a sample $\by^t_{i}$ are close to their
analogues for $U^t_i$, which distribution is provided before
Proposition~\ref{amp-vec}. It will be enough to consider that the matrices 
$W^t$ are Gaussian. 
\begin{proposition} 
\label{cvg-y} 
Assume that the matrix $W$ is Gaussian. Then for each multi-index
$\bm \in \NN^q$, each integer $t > 0$, 
\[
\max_{i\in[n]} 
\left| \EE (\by^t_{i})^\bm - \EE (U^{t}_{i})^\bm \right| \tolong 0 .
\]
\end{proposition} 
\begin{proof}
The uniform convergence we need to show can be equivalently stated as follows:
for each sequence $(i_n)$ valued in $[n]$, it holds that 
\begin{equation}
\label{cvg-in} 
\EE (\by^t_{i_n})^\bm - \EE (U^{t}_{i_n})^\bm  \tolong 0 .
\end{equation} 
Remember that $\by^0_{i\to j} = \bcx^0_i$ and 
$\by^{t+1}_{i\to j} = \sum_{l\in[n]\setminus\{j\}} 
  W_{il}^t f(\by^t_{l\to i}, l, t)$ for each $i\neq j$. First, using this
equation, it is easy to establish by recurrence on $t$ that 
\begin{equation}
\label{bmomy} 
\forall t \geq 0, \ \forall \bm \in \NN^q, \ 
\sup_n \max_{i\neq j} \EE | (\by^{t}_{i\to j})^\bm | < \infty. 
\end{equation} 
Fixing $t$, we shall show by recurrence on $u = 1, \ldots, t-1$ the following
assertion that we denote as $\cA(u)$: For each multi-index $\bm$, each
sequence $(j_n)$ valued in $[n]$, and each $(n-1)$--uple 
$(b_\ell)_{\ell\in[n]\setminus\{ j_n\}}$ with bounded elements, it holds that 
\[ 
\sum_{\ell\in[n]\setminus\{j\}} s_{i\ell} b_{\ell} 
(\by^u_{\ell\to j})^\bm - s_{i\ell} b_\ell \EE (U^u_\ell)^\bm  
   \toprobalong 0 ,  
\] 
where $j = j_n$. 
In all the proof, we shall need the covariance matrices $H_{ij}^u$ defined
for $u\geq 1$ and $i\neq j$ as 
\[
H_{ij}^u = \sum_{l\in[n]\setminus\{ j\}} s_{il} 
  f(\by^{u-1}_{l\to i}, l, u-1) f(\by^{u-1}_{l \to i}, l, u-1)^\T . 
\]
Starting with $\cA(1)$, let us assume for notational simplicity and 
without generality loss that $j = n$. We have 
\begin{align*}
\sum_{l\in[n-1]} s_{il} b_{l} 
 \left( (\by^{1}_{l\to n})^\bm - \EE (U^{1}_l)^\bm \right)  &= 
\sum_{l\in[n-1]} s_{il} b_{l} \left(  
 (\by^{1}_{l\to n})^\bm -  
   \EE (\by^{1}_{l\to n})^\bm \right) \\
&\phantom{=} 
 + \sum_{l\in[n-1]} s_{il} b_{l} 
   \left( \EE (\by^{1}_{l\to n})^\bm  - \EE (U^{1}_l)^\bm \right) \\
&= \chi_1 + \chi_2 .  
\end{align*} 
It is obvious that $\by^{1}_{l\to n} \sim \cN(0, H^1_{ln})$, 
where $H^1_{ln} = Q^1_l - E_{l,n}$, and the rank-one matrix  
\[
E_{l,n} = s_{l,n} f(\bcx^0_{n}, n, 0) f(\bcx^0_{n}, n, 0)^\T 
\]
has a spectral norm that converges to zero by Assumptions~\ref{ass-A} 
and~\ref{x0}. It is then easy to deduce that $\chi_2 \to_n 0$. 
 
To deal with $\chi_1$, we make use of Poincaré's inequality 
\cite[Ch.~2]{pas-livre}. For $u \geq 1$, let $\Sigma^{u}_{n}$ be the 
the $q(n-1) \times q(n-1)$ covariance matrix defined as 
$\Sigma^{u}_{n}
= \begin{bmatrix} \Sigma^{u}_{n}(k,l) \end{bmatrix}_{k,\ell = 1}^{n-1}$ where
the $q\times q$ block $\Sigma^{u}_{n}(k,l)$ is given as 
\[
\Sigma^{u}_{n}(k,l) = \left\{ 
\begin{array}{l} H_{k,n}^u  \ \text{if} \ k=l \\ 
s_{kl} f(\by^{u-1}_{l\to k}, l, u-1) 
  f(\by^{u-1}_{k\to l}, k, u-1)^\T 
 \ \text{if not.} 
\end{array} \right. 
\]
Note that $\Sigma^{1}_{n}$ is deterministic, and the $\RR^{q(n-1)}$--valued 
vector $\by^{1}_{\cdot\to n} = \begin{bmatrix} \by^1_{i\to n}
\end{bmatrix}_{i=1}^{n-1}$ has the distribution $\cN(0, \Sigma^{1}_{n})$. 
Define the function 
$\Gamma(\by^{1}_{\cdot\to n}) = \sum_{l\in[n-1]} s_{il} b_l
(\by^{1}_{l\to n})^\bm$, and write $\nabla_{\by} \by^\bm = p_\bm(\by)$, a
$\RR^q$--valued polynomial. Then, we obtain by Poincaré's inequality 
\[
\EE \chi_1^2 = 
\var\left( \Gamma(\by^{1}_{\cdot\to n}) \right) 
\leq  
 \sum_{k,l=1}^{n-1} s_{i k} s_{i l} b_{k} b_{l}
 \EE p_\bm(\by_{k\to n}^1)^\T \Sigma^{1}_{n}(k,l) 
  p_\bm(\by_{l\to n-1}^1) . 
\]
Considering the expression of $\Sigma^1_n(k,l)$, and using Assumption~\ref{x0}
and the bound~\eqref{bmomy}, the right hand side of the previous display 
satisfies the bounds 
\[
\Bigl| \sum_{k=l} \cdots \Bigr| \leq C \sum_{k=1}^{n-1} s_{ik}^2 \leq 
   \frac CK, \quad \text{and} \quad 
\Bigl| \sum_{k\neq l} \cdots \Bigr| \leq C \sum_{k,\ell=1}^{n-1} 
  s_{ik} s_{il} s_{kl} \leq \frac CK. 
\] 
It results that $\chi_{1} \toprobashortshort_n 0$, and $\cA(1)$ is 
established. 

Assuming that $\cA(u)$ is true, let us establish $\cA(u+1)$.  Define the
$\sigma$--field $\mcF^u = \sigma( W^0, \ldots, W^{u})$. Then, still setting
$j=n$, the conditional distribution 
$\mcL( \by^{u+1}_{\cdot\to n} \, | \, \mcF^{u-1})$ of the vector 
$\by^{u+1}_{\cdot\to n} = \begin{bmatrix} \by^{u+1}_{i\to n}
\end{bmatrix}_{i=1}^{n-1}$ given $\mcF^{u-1}$ is $\cN(0, \Sigma^{u+1}_n)$. 
We also have from $\cA(u)$ that 
\begin{equation}
\label{H-R} 
\forall (i_n) \text{ valued in } [n], \ H_{i_nn}^{u+1} - Q^{u+1}_{i_n} 
  \toprobalong 0 . 
\end{equation} 
With this at hand, we write 
\begin{align*}
\sum_{l\in[n-1]} s_{il} b_{l} 
 \left( (\by^{u+1}_{l\to n})^\bm - \EE (U^{u+1}_l)^\bm \right)  &= 
\sum_{l\in[n-1]} s_{il} b_{l} \left(  
 (\by^{u+1}_{l\to n})^\bm -  
   \EE \left[ (\by^{u+1}_{l\to n})^\bm \, | \, \mcF^{u-1} \right] \right) \\
&\phantom{=} 
 + \sum_{l\in[n-1]} s_{il} b_{l} 
   \left( \EE \left[ (\by^{u+1}_{l\to n})^\bm \, | \, \mcF^{u-1} \right] 
  - \EE (U^{u+1}_l)^\bm \right) \\
&= \chi_1 + \chi_2 .  
\end{align*} 
Given a small $\delta > 0$, we write 
\begin{align*} 
| \chi_2 | &\leq 
  \sum_{l\in[n-1]} s_{il} |b_{l}|  
   \sup_{H \geq 0 \, : \, \| H - Q^{u+1}_l \| \leq \delta} 
   \left| \varphi(H) - \varphi(Q^{u+1}_l) \right| \\
 &\phantom{=} + \sum_{l\in[n-1]} s_{il} |b_{l}| \,  | \varphi(H_{ln}^{u+1}) | 
   \1_{\|  H_{ln}^{u+1} - Q^{u+1}_l \| > \delta} 
 + \sum_{l\in[n-1]} s_{il} |b_{l}| \,  | \varphi( Q^{u+1}_l) | 
   \1_{\|  H_{ln}^{u+1} - Q^{u+1}_l \| > \delta} \\
&= \chi_{2,1} + \chi_{2,2} + \chi_{2,3}, 
\end{align*} 
where $\varphi(H) = \EE Y^\bm$ when $Y \sim \cN(0, H)$. Observe that $\sum_\ell
s_{il} |b_{l}|$ is bounded. Furthermore, assuming that $H$ and $Q$ belong to
a compact, it holds that $\varphi(H) - \varphi(Q) \to 0$ when 
$\| H - Q \| \to 0$ by the continuity of $\varphi$. Thus, using the 
bound~\eqref{Qfini}, we obtain that $\chi_{2,1} \to 0$ as $\delta\to 0$. 
We also have by the Jensen and the Cauchy-Schwarz inequalities that 
\[
\EE \chi_{2,2} \leq 
  \sum_{l\in[n-1]} s_{il} |b_{l}| \,  
  (\EE | (\by^{u+1}_{l\to n})^\bm |^2)^{1/2}  
   \PP[\|  H_{ln}^{u+1} - Q^{u+1}_l \| > \delta]^{1/2} . 
\]
The convergence~\eqref{H-R} can be rewritten as 
\[
\forall\delta > 0, 
   \max_{l\in[n]} \PP[\|  H_{ln}^{u+1} - Q^{u+1}_l \| > \delta] 
 \xrightarrow[n\to\infty]{} 0 .
\]
Using the bound~\eqref{bmomy}, we obtain that $\EE \chi_{2,2} \to_n 0$, thus,
$\chi_{2,2} \toprobashortshort_n 0$. 
It is easy to show that $\chi_{2,3} \toprobashortshort_n 0$. In conclusion, 
$\chi_{2} \toprobashortshort_n 0$. 

To deal with $\chi_1$, we use Poincaré's inequality involving this time the 
conditional distribution 
$\mcL( \by^{u+1}_{\cdot\to n} \, | \, \mcF^{u-1})$. By an argument similar to
above, this leads to 
\[
\EE \chi_1^2 = 
\EE \var\left( \Gamma(\by^{u+1}_{\cdot\to n}) \, | \, \mcF^{u-1} \right) 
\leq  
 \sum_{k,l=1}^{n-1} s_{i k} s_{i l} b_{k} b_{l}
 \EE\left[ p_\bm(\by_{k\to n}^{u+1})^\T \Sigma^{u+1}_{n}(k,l) 
  p_\bm(\by_{l\to n}^{u+1}) \right] \leq \frac CK  
\]
with the help of Inequality~\eqref{bmomy}. 
It results that $\chi_{1} \toprobashortshort_n 0$, and $\cA(u+1)$ is 
established.  

We now use $\cA(t-1)$ to prove the convergence stated by our proposition. 
Recall that 
$\by^{t}_{i} = \sum_{l\in[n]} W_{il}^{t-1} 
 f(\by^{t-1}_{l\to i}, \ell, t-1)$. Set $i = i_n$ as in~\eqref{cvg-in}.  
Given $A > 0$, define the real function $\eta_A : \RR \to \RR$ as
the function that coincides with the identity on $[-A,A]$ and is equal to $A$
on $(A,\infty)$ and to $-A$ on $(-\infty, -A)$. To study 
$\EE (\by^{t}_i)^\bm - \EE (U^{t}_i)^\bm$, we can assume that 
$\by^{t}_{i} = (H_{i}^t)^{1/2} \xi$ and $U^{{t}}_i = (Q^{t}_i)^{1/2} \xi$, 
where 
\[
H_{i}^t = \sum_{l\in[n]} s_{il} 
  f(\by^{t-1}_{l\to i}, l, t-1) f(\by^{t-1}_{l \to i}, l, t-1)^\T , 
\]
$\xi \sim \cN(0, I_q)$ is independent of $H_{i}^t$, and $(\cdot)^{1/2}$ is the 
semidefinite positive square root. Write $Y = (\by^{t}_{i})^\bm$ and 
$U = (U^{t}_i)^\bm$.  With this, we have 
\begin{align*}
\EE Y - \EE U &= 
\left( \EE \eta_A(Y) - \EE\eta_A(U) \right) 
 + \left( \EE Y - \EE \eta_A(Y) \right) + \left( \EE U - \EE \eta_A(U) \right) 
    \\
&= \chi_1 + \chi_2 + \chi_3 .
\end{align*}
For $\delta > 0$, we have 
\begin{align*} 
|\chi_1 | &\leq \sup_{H \geq 0 \, : \, \| H - Q^{t}_i \|
 \leq \delta } \left| \EE_\xi \eta_A( H^{1/2} \xi)^\bm ) 
 - \EE_\xi \eta_A( ((Q^{t}_i)^{1/2} \xi)^\bm ) \right| 
  + 2A \PP\left[ \| H_{i}^t - Q^t_i \| > \delta \right] \\ 
 &\leq C_1(A,\delta) + C_2(A,\delta,n), 
\end{align*} 
where $C_1(A,\delta) \to 0$ as $\delta \to 0$, and $C_2(A,\delta,n) \to 0$ 
as $n\to\infty$ as in~\eqref{H-R}. Regarding $\chi_2$, we have 
\[
\chi_2 = \EE ( Y - \eta_A(Y) ) \1_{|Y| > A} \leq 
 \EE ( |Y| + A ) \1_{|Y| > A} \leq \sup_n \frac{\EE Y^2 + \EE | Y|^3}{A^2} 
 \leq \frac{C}{A^2} 
\]
for some $C > 0$. We have a similar bound for $\chi_3$. By taking $A$ large
enough then $\delta$ small enough, we easily obtain that 
$\EE (\by^t_{i})^\bm - \EE (U^{t}_i)^\bm  \to_n 0$. 
\end{proof}

The results of Propositions~\ref{z-tz}, \ref{z-y}, \ref{z-x}, and~\ref{cvg-y} 
will lead to the convergences~\eqref{cvg-vec}, which will be the consequences 
of the two following propositions.  
\begin{proposition}
\label{cvg-mean} 
Let $\psi^{(n)} : \RR^q \times [n] \to \RR$ be such that 
$\psi^{(n)}(\cdot, l)$ is a multivariate polynomial with a bounded 
degree and bounded coefficients as functions of $(l,n)$. Then, for each 
set $\cS^{(n)} \subset [n]$ with $|\cS^{(n)}| \to_n \infty$, it holds that 
\[
\frac{1}{|\cS^{(n)}|} \sum_{i\in\cS^{(n)}} 
  \EE \psi(\bcx^{(n),t}_i,i) - \EE\psi(U^{(n),t}_i,i) 
 \xrightarrow[n\to\infty]{} 0 . 
\]
\end{proposition} 
\begin{proof} 
By Propositions \ref{z-tz}, \ref{z-y} and \ref{z-x}, we obtain that for each $t
\geq 1$ and each $\bm \in \NN^q$, there exists $C > 0$ such that $\left| \EE
(\bcx^t_i)^\bm - \EE (\by^t_i)^\bm \right| \leq C K^{-1/2}$ for each $i\in
[n]$, and furthermore, $W$ can be assumed Gaussian in the construction of the
$\by^t_i$.  Using Proposition~\ref{cvg-y}, we obtain that 
$\max_{i\in[n]} | \EE (\bcx^t_{i})^\bm - \EE (U^t_{i})^\bm | \to_n 0$. 

Furthermore, using the moment bound~\eqref{Exbnd} and observing that the
mixed moments of the $U^t_i$ are bounded by~\eqref{Qfini}, we obtain the
result. 
\end{proof}

\begin{proposition}[adaptation of Proposition~5 of \cite{bay-lel-mon-15}] 
\label{varphi} 
Let $\psi^{(n)} : \RR^q \times [n] \to \RR$ be as in the previous proposition.
Let $\cS^{(n)} \subset [n]$ be a non empty set such that 
$| \cS^{(n)} | \leq C K^{(n)}$. Then, for each $t > 0$, 
\begin{align} 
& \var\Bigl(\frac 1K \sum_{i \in \cS} \psi(\bcx^t_i,i) \Bigr) 
 \tolong  0, \quad \text{and} \label{var-S} \\ 
&\var\Bigl(\frac 1n \sum_{i\in[n]} \psi(\bcx^t_i,i) \Bigr) 
 \tolong  0. \nonumber 
\end{align} 
\end{proposition} 
\begin{proof}
We adapt the proof of \cite[Prop.~5]{bay-lel-mon-15} to our situation. 
Let $\{ \chX_{ij} \}_{1\leq i < j \leq n}$ be a set of real independent random
variables that satisfy the same assumptions as  $\{ X_{ij} \}_{1\leq i < j \leq
n}$. Assume furthermore that these two sets are independent.  Write $\chX_{ji}
= \chX_{ij}$ for $1\leq i < j\leq n$, and let $\chX_{ii} = 0$ for $i\in [n]$. 
Define the $n\times n$ matrix 
$\chX = \begin{bmatrix} \chX_{ij} \end{bmatrix}_{i,j=1}^n$. 

Let $B = [b_{ij}]$ be the $n\times n$ symmetric matrix which first row is
defined as $b_{1j} = \1_{j\in \cS}$ for $j\in[n]$, and which have zeros outside
its first row and first column. 
To establish the first convergence, we build the matrix 
$\chW = \begin{bmatrix} \chW_{ij} \end{bmatrix}_{i,j=1}^n$ defined as 
\[
\chW_{ij} = \frac{1}{\sqrt{K}}  B \odot \chX 
\]
(here we assume without affecting the conclusion of the proposition
that $1 \not\in \cS$). 
We construct a new AMP sequence around the $2n\times 2n$ matrix
$\begin{bmatrix} W & 0 \\ 0 & \chW \end{bmatrix}$, which obviously satisfies
Assumptions~\ref{ass-X} and \ref{ass-A}. The new algorithm, which delivers the 
$2n$-uple $(\bcy^{k}_1,\ldots, \bcy^{k}_{2n})$ at Iteration $k$, is written as 
follows:
$\bcy^0_i = \bcx^0_i$ for $i \in [n]$ and $\bcy^0_i = 0$ otherwise, 
\begin{align*} 
\begin{bmatrix} \ \ (\bcy^{k+1}_1)^\T \ \ \\ \vdots \\ (\bcy^{k+1}_n)^\T \\ 
 0 \\ \vdots \\ 0 \end{bmatrix} = 
\begin{bmatrix} W & 0 \\ 0 & \chW \end{bmatrix} 
\begin{bmatrix} \ \ f(\bcy^{k}_1, 1,k)^\T \ \ \\ \vdots \\ 
  f(\bcy^k_n, n, k)^\T \\ 
 0 \\ \vdots \\ 0 \end{bmatrix} 
- \begin{bmatrix} \ \ f(\bcy^{k-1}_1, 1,k-1)^\T \ \ \\ \vdots \\ 
 f(\bcy^{k-1}_n, n, k-1)^\T \\ 0 \\ \vdots \\ 0 
\end{bmatrix} 
\begin{bmatrix} \\ 
  \sum_{l} W_{i,l}^2 
  \frac{\partial f_r}{\partial y(s)} 
   (\bcy^{k}_l,l,k) \\ \\   
 \end{bmatrix}_{s,r=1}^q, 
\end{align*} 
for $k = 0, \ldots, t-1$, and 
\begin{align*} 
\begin{bmatrix}  0  \\ \vdots \\ 0 \\ 
 (\bcy^{t+1}_{n+1})^\T \\ \vdots \\ \ \ (\bcy^{t+1}_{2n})^\T  \ \ 
  \end{bmatrix} = 
\begin{bmatrix} W & 0 \\ 0 & \chW \end{bmatrix} 
\begin{bmatrix} \ \ 0 \ \ \\ \vdots \\ 0 \\ 
 [ \psi(\bcy^t_1,1), 0,\ldots, 0 ]  \\ \vdots \\ 
 [ \psi(\bcy^t_n,n), 0,\ldots, 0 ] \end{bmatrix} 
\end{align*} 
(next iterations are irrelevant). This sequence enters the framework of 
Proposition~\ref{amp-vec}. 

It is clear that $\bcy^k_i = \bcx^k_i$ for $k \in [t]$ and $i \in [n]$, and 
therefore, 
\[
\bcy^{t+1}_{n+i}(1) = \sum_{l\in[n]} \chW_{il} \psi(\bcx^t_l,l) 
 \quad \text{for } i \in [n]. 
\]
Set $i = 1$. On the one hand, we have 
\begin{align} 
\EE \bcy^{t+1}_{n+1}(1)^4 &= 
  \sum_{l_1l_2l_3l_4 \in \cS} 
  \EE \chW_{1 l_1} \chW_{1 l_2} \chW_{1 l_3} \chW_{1l_4} 
 \EE \psi(\bcx^t_{l_1},l_1) \psi(\bcx^t_{l_2},l_2) 
  \psi(\bcx^t_{l_3},l_3) \psi(\bcx^t_{l_4},l_4) \nonumber \\
 &= \frac{3}{K^2} \sum_{l_1,l_2 \in \cS}  
 \EE \psi(\bcx^t_{l_1},l_1)^2 \psi(\bcx^t_{l_2},l_2)^2 + \varepsilon 
\label{4momy} 
\end{align} 
where $|\varepsilon| \leq C / K$. On the other hand, 
Propositions~\ref{z-tz}, \ref{z-y}, \ref{z-x} and~\ref{cvg-y} applied to our
new AMP sequence show that 
\[
\EE \bcy^{t+1}_{n+1}(1)^4 - \EE (\widetilde U^{t+1})^4 \tolong 0, 
\]
where $\widetilde U^{t+1} \sim \cN(0, (\sigma^{t+1})^2)$, with 
\[
(\sigma^{t+1})^2 = \frac 1K \sum_{l \in \cS} \EE\psi(U^t_l,l)^2 , 
\]
By the Gaussianity and centeredness of $\widetilde U^{t+1}$, we thus have 
\[
\EE \bcy^{t+1}_{n+1}(1)^4 - 
 \frac{3}{K^2} \Bigl( \EE \sum_{l \in \cS} \psi(U^t_l,l)^2 \Bigr)^2
 \tolong 0, 
\]
and since $\max_{i\in[n]} | \EE (\bcx^t_i)^\bm - \EE (U^t_i)^\bm | \to_n 0$, 
we get that 
\[
\EE \bcy^{t+1}_{n+1}(1)^4 - 
 \frac{3}{K^2} \Bigl( \EE \sum_{l\in\cS} \psi(\bcx^t_l,l)^2 \Bigr)^2
 \tolong 0, 
\]
Combining this convergence with~\eqref{4momy}, we obtain the first convergence
stated by our proposition for polynomials of the type $\psi(\cdot,l)^2$. 
To obtain this convergence for arbitrary polynomials, write 
$\Psi(\bx, l) = ( 1 + \varepsilon \psi(\bx, l))$ for $\varepsilon > 0$. 
Since 
\[
\frac{1}{\varepsilon^2} 
\var\Bigl(\frac 1K \sum_{i \in \cS} \Psi(\bcx^t_i,i)^2 \Bigr) = 
\var\Bigl(\frac{\varepsilon}{K} \sum_{i \in \cS} \psi(\bcx^t_i,i)^2 
 + \frac 2K \sum_{i \in \cS} \psi(\bcx^t_i,i) 
\Bigr) 
\]
must vanish for all $\varepsilon > 0$, we get the convergence~\eqref{var-S}. 

To establish the other convergence in the statement, let 
$\cS_i^{(n)} = \{ i +1, ((i+1)\!\!\mod n)+1, \ldots, 
((i+K_n-1)\!\!\mod n)+1 \}$ for $i\in \{0,\ldots, n-1 \}$. 
Then, we have that 
\[
\forall l \in [n], \quad 
  \psi(\bcx^t_l,l) = \frac{1}{K} \sum_{i=0}^{n-1} \psi(\bcx^t_l,l) 
 \1_{l \in S_i} .
\]
Therefore, writing 
\[
\frac 1n \sum_{l\in[n]} \psi(\bcx^t_l,l) = 
 \frac 1n \sum_{i=0}^{n-1} \frac{1}{K} \sum_{l\in S_i} \psi(\bcx^t_l,l), 
\]
we can use Minkowski's inequality along with the convergence~\eqref{var-S} 
to show that the variance of the left hand side converges to zero. 
\end{proof} 

The convergences~\eqref{cvg-vec} follow at once from 
Propositions~\ref{cvg-mean} and~\ref{varphi}. Proposition~\ref{amp-vec} is 
proven.

\subsubsection{Proof of Proposition~\ref{amp-pscal}}  

To establish Proposition~\ref{amp-pscal} for a given fixed $t > 0$, we
apply Proposition~\ref{amp-vec} with $q = t$ and a properly designed sequence 
of functions $f(\cdot,\cdot,0), \ldots, f(\cdot,\cdot,t-1)$, along the
idea used in \cite[Proof of Th.~5]{bay-lel-mon-15}. 

For $k = 0,\ldots, t-1$ define the function $f(\bx, i, k)$ of the statement of
Proposition~\ref{amp-vec} as follows.  Consider the initial vector 
$x^0$ in Algorithm~\eqref{amp-posc} as a constant parameter vector.  For 
$\bx = \left[ x(1),\ldots, x(t) \right]$, set  
\[
\begin{matrix} 
f(\bx, i, 0)^\T &=& [\! & p(x^0_i, i, 0) & 0              & \ldots & 0 & \!] \\
f(\bx, i, 1)^\T &=& [\! & p(x^0_i, i, 0) & p( x(1), i, 1) & \ldots & 0 & \!] \\
\vdots \\
f(\bx, i,t-1)^\T &=& [\! & p( x^0_i, i, 0) & p( x(1), i, 1) &\ldots & 
  p(x(t-1), i, t-1) & \!]  . 
\end{matrix} 
\]
(note that $f(\cdot, i, 0)$ is a polynomial with degree zero). 
With this construction, if we start Algorithm~\eqref{amp-posc} with the initial
value $x^0$ and Algorithm~\eqref{ampol} with an arbitrary initial value, 
then we can easily show by recurrence on the first $t$ iterations running in 
parallel for both algorithms that 
\[
\begin{matrix} 
(\bcx^1_i)^\T &=& [\! & \cx^1_i & 0       & \ldots & 0 & \!]  \\
(\bcx^2_i)^\T &=& [\! & \cx^1_i & \cx^2_i & \ldots & 0 & \!]  \\
\vdots \\
(\bcx^t_i)^\T &=& [\! & \cx^1_i & \cx^2_i & \ldots &\cx^t_i & \!]  . 
\end{matrix} 
\]
We also notice that 
\[
\forall i \in [n], \quad Q^t_i = \chR^t_i .
\]
With this at hand, Proposition~\ref{amp-pscal} follows from 
Proposition~\ref{amp-vec}. 

In order to deduce Theorems~\ref{th-amp-lip} and~\ref{th-amp-gal} from
Proposition~\ref{amp-pscal}, we now need to approximate the activation
functions present in the statements of these theorems with polynomials. The
next subsection is devoted to this purpose.

\subsection{From polynomial to general activation functions} 
 
In all this subsection, $h : \RR\times \cQ_\eta\times \NN \to \RR$ is a
function that complies with either Assumption~\ref{ass:h-lip} or
Assumption~\ref{ass:h-gal}. 

The proof of the following lemma makes use of the density of the polynomials
in the Hilbert space $L^2(\RR,\nu)$ when $\nu$ is a Gaussian measure. 
Density arguments of this kind have been used in the AMP literature in, 
\emph{e.g.}, \cite{dud-lu-sen-(arxiv)22,wan-zho-fan-(arxiv)22}. 
\begin{lemma}
\label{approx-hg}
Fix $t$, and fix two positive numbers $0 < \sigma_{\min}^2 \leq
\sigma_{\max}^2$. Let $\varepsilon > 0$ be an arbitrarily small number. Then,
there exists a function $g_\varepsilon(\cdot,\cdot,t) : \RR \times \cQ_\eta 
 \to \RR$ that satisfies the following properties: 
For each $\ueta\in\cQ_\eta$, the function $g_\varepsilon(\cdot,\ueta,t)$ is a
polynomial. Denoting as $\partial g_\varepsilon$ the derivative of 
$g_\varepsilon$ with respect to the first parameter, the inequalities
\[
\EE (h(\uxi,\ueta,t) - g_\varepsilon(\uxi,\ueta,t))^2 \leq
\varepsilon, \quad \text{and} \quad 
| \EE [\partial h(\uxi,\ueta,t) - \partial g_\varepsilon(\uxi,\ueta,t)] | 
  \leq \varepsilon
\]
hold true for each random variable $\uxi \sim \cN(0, \sigma^2)$ with 
$\sigma^2 \in [ \sigma_{\min}^2, \sigma_{\max}^2 ]$ and each 
$\ueta\in\cQ_\eta$. 
Finally, the function $p^{(n)}(\ux,i,t) = g_\varepsilon(\ux, \eta^{(n)}_i, t)$ 
satisfies the assumptions of Proposition~\ref{amp-pscal}. 
\end{lemma} 
\begin{proof}
Let $\varepsilon' > 0$ be a small number to be set at the end of the
proof. Let $\kappa$ be the function from Assumption~\ref{ass:h-lip}
or~\ref{ass:h-gal}. If Assumption~\ref{ass:h-lip} is chosen, set $m=1$, 
otherwise, let $m$ be the integer specified in Assumption~\ref{ass:h-gal}.
Define $\delta > 0$ as 
\[
\delta = \max \Bigl\{ e \in ( 0, D_{\cQ_\eta}], \, : \, 
 \kappa(e)^2 \leq \frac{\varepsilon'}{\EE (1 + |\xi|^m)^2} \Bigr\}, 
\]
where $\xi\sim\cN(0,\sigma_{\max}^2)$, and where $D_{\cQ_\eta}$ is the diameter
of $\cQ_\eta$. Since $\cQ_\eta$ is compact, it contains a $\delta$--net with
finite cardinality.  Let $\cS \subset \cQ_\eta$ be a $\delta$--net of
$\cQ_\eta$ with the smallest cardinality, and write $M = |\cS|$.  Let $\phi :
\cQ_\eta \to \cS$ be such that $\phi(\ueta)$ is the closest element in $\cS$ to
$\ueta$ if this closest element is unique, and $\phi(\ueta)$ is the closest
element smaller than $\ueta$ if not.  Denoting as $\bs\eta_1 < \bs\eta_2 < 
\cdots < \bs\eta_M$ the elements of $\cS$, define the function 
$\psi : \cQ_\eta \to [M]$ as $\phi(\ueta) = \bs\eta_{\psi(\ueta)}$.  With these 
definitions, we have 
$\EE (h(\xi,\ueta,t) - h(\xi,\phi(\ueta),t))^2 \leq \varepsilon'$. 

It is now well-known that the polynomials are dense in the space 
$L^2(\RR, \nu)$ when $\nu$ is a Gaussian distribution. Therefore, there
are $M$ polynomials $P(\cdot, l,t)$ such that 
$\EE (h(\xi, \bs\eta_l,t) - P(\xi,l,t))^2 \leq \varepsilon'$ for each 
$l \in [M]$. Fixing $l$, let 
$\varphi(\ux) = (h(\ux, \bs\eta_l,t) - P(x,l,t))^2$. For 
$\uxi \sim \cN(0, \sigma^2)$ with 
$\sigma^2 \in [ \sigma_{\min}^2, \sigma_{\max}^2 ]$, we have 
\[
\EE (h(\uxi, \bs\eta_l,t) - P(\uxi,l,t))^2 = 
 \frac{1}{\sqrt{2\pi \sigma^2}} \int \varphi(\ux) e^{-\frac{\ux^2}{2\sigma^2}} 
  d\ux 
 \leq \frac{\sigma_{\max}}{\sigma_{\min}} 
 \frac{1}{\sqrt{2\pi \sigma_{\max}^2}} 
  \int \varphi(\ux) e^{-\frac{\ux^2}{2\sigma_{\max}^2}} d\ux = 
 \frac{\sigma_{\max}}{\sigma_{\min}} \varepsilon' .
\]
Putting things together, we obtain that 
\begin{align*} 
\EE(h(\uxi,\ueta,t) - P(\uxi,\psi(\ueta),t))^2 &\leq 
2 \EE(h(\uxi,\ueta,t) - h(\uxi,\phi(\ueta),t))^2 + 
2 \EE(h(\uxi,\phi(\ueta),t) - P(\uxi,\psi(\ueta),t))^2 \\
 &\leq  2 ( 1 + \sigma_{\max} / \sigma_{\min} ) \varepsilon' . 
\end{align*} 
We also have by Stein's lemma that 
\[
\EE [\partial h(\uxi,\ueta,t) - \partial_{\uxi} P(\uxi,\psi(\ueta),t)] = 
 \frac{1}{\sigma^2} \EE [\uxi ( h(\uxi,\ueta,t) - P(\uxi,\psi(\ueta),t))] . 
\]
Thus, we obtain by Cauchy-Schwarz and the previous result that 
\[
| \EE [\partial h(\uxi,\ueta,t) - \partial_{\uxi} P(\uxi,\psi(\ueta),t)] | \leq 
 \frac{\sigma_{\max}}{\sigma_{\min}^2} 
 \sqrt{2 ( 1 + \sigma_{\max} / \sigma_{\min} )} \sqrt{\varepsilon'} , 
\]
By adjusting $\varepsilon'$, we thus obtain that 
\[
\EE (h(\uxi,\ueta,t) - P(\uxi,\psi(\ueta),t))^2 \leq
\varepsilon, \quad \text{and} \quad 
| \EE [\partial h(\uxi,\ueta,t) - \partial_{\uxi} P(\uxi,\psi(\ueta),t)] | 
  \leq \varepsilon, 
\]
and it remains to set $g_\varepsilon(\ux,\ueta,t) = P(\ux,\psi(\ueta),t)$. 
We also notice that the polynomials $p^{(n)}(\ux, i, t) = 
 g_\varepsilon(\ux,\ueta_i^{(n)},t)$ have bounded degrees and bounded 
coefficients, and thus comply with the statement of 
Proposition~\ref{amp-pscal}. 
\end{proof} 
\begin{lemma}
\label{ge} 
There exists a constant $c > 0$ such that 
$R_i(t,t) \geq c$ for all $i\in[n]$ and $t\in [\tmax]$. Let $e > 0$ be a
small number. Then, there exists a set of $\RR\times\cQ_\eta \to \RR$ functions 
$\{g_e(\cdot,\cdot,t)\}_{t = 0}^{\tmax-1}$ such that 
for each $\eta\in\cQ_\eta$, the function $g_e(\cdot,\eta,t)$ is a
polynomial, and furthermore,
\[
 \forall i \in [n], \quad 
| h(x^0_i,\eta_i,0) - g_e(x^0_i,\eta_i,0)| \leq e, 
\]
and $\forall t \in[\tmax-1]$, $\forall i \in [n]$, 
\[
\EE (h(Z^t_i,\eta_i,t) - g_e(Z^t_i,\eta_i,t))^2 \leq e, 
  \quad 
| \EE [\partial h(Z^t_i,\eta_i,t) - \partial g_e(Z^t_i,\eta_i,t)] | 
  \leq e, 
\]
and the functions $p(x,i,t) = g_e(x,\eta_i,t)$ defined for 
$0\leq t\leq \tmax - 1$ satisfy the assumptions of Proposition~\ref{amp-pscal}. 

Furthermore, there exists a non-negative function $\delta(e)$ on a small
interval $(0,\varepsilon)$ that converges to zero as $e \to 0$, and that
satisfies the following property. Given $e \in (0,\varepsilon)$, construct the function $g_e$
as above, and let $(\cbZ^t = [\cZ^t_i]_{i=1}^n )_{t\in[\tmax]}$ be the finite
sequence of centered Gaussian vectors which distribution is determined by the
$(S,g_e,\eta,x^0)$--state evolution equations stopped at $\tmax$, and leading
to the covariance matrices $\{\chR^{\tmax}_i\}_{i\in[n]}$.  Then, 
$\| R^{\tmax}_i - \chR^{\tmax}_i \| \leq \delta(e)$ for all $i\in[n]$. In
particular, 
\begin{equation}
\label{bnd-cR} 
\sup_n \max_{i\in[n]} \| \chR^{\tmax}_i \| < \infty. 
\end{equation}  
\end{lemma} 

\begin{proof} 
Recall the notations introduced in Assumptions~\ref{ass-A} and~\ref{nondeg},
and recall that $\cK_i = \{ j \in [n], \ s_{ij} > 0 \}$.  We first show that
for each $i\in [n]$, the set $\cI_i \subset \cK_i$ defined as 
$\cI_i = \{ j\in[n], \ s_{ij} > c_S / (2 C_{\text{card}} K) \}$ satisfies
$|\cI_i| > \alpha_S K$.  Using Assumptions~\ref{ass-A} and~\ref{nondeg}, we 
indeed have 
\[
c_S \leq \sum_{j\in\cK_i} s_{ij} = \sum_{j\in\cI_i} s_{ij} 
 + \sum_{j\in\cK_i\setminus\cI_i} s_{ij} < \frac{C_S}{K} |\cI_i| 
 + \frac{c_S}{2 C_{\text{card}} K}( C_{\text{card}} K - | \cI_i | ), 
\]
and the result is obtained by rearranging. 

Using Assumption~\ref{nondeg} again, we then have 
$R_i(1,1) = \sum_l s_{il} h(x^0_l, \eta_l, 0)^2 \geq c > 0$ for each $i\in[n]$,
by lower bounding the sum to $l \in \cI_i$. 
Assuming that $R_i(t,t) \geq c > 0$ for each $i\in[n]$, we also have that
$R_i(t+1,t+1) = \sum_l s_{il} \EE h(Z^t_l, \eta_l, t)^2 \geq c > 0$: This can
be checked by noticing from Assumption~\ref{nondeg} and the properties of $h$
that there exists $c > 0$ such that 
$\sum_{l} s_{il} \EE h(\uxi,\eta_l,t)^2 \geq c_{\text{h}}(t)$ when
$\uxi\sim\cN(0,1)$, and by making standard Gaussian calculations to modify the
variance of $\uxi$ and the bound $c_{\text{h}}(t)$.  This establishes the first
result of the lemma. 

With this result, the existence of the functions 
$\{g_e(\ux,\ueta,t) \}_{t = 1}^{\tmax - 1}$ follows at once from
Lemma~\ref{approx-hg}, with the numbers $\sigma^2_{\min}$ and $\sigma^2_{\max}$
in the statement of that lemma chosen as 
$\sigma^2_{\min} = \min_{i,t} R_i(t,t) / 2$ and $\sigma^2_{\max} = 2 \max_{i,t}
R_i(t,t)$.  As regards $g(\ux, \ueta, 0)$, we can use Assumption~\ref{x0}
and either Assumption~\ref{ass:h-lip} or Assumption~\ref{ass:h-gal},  
and invoke the Stone-Weierstrass theorem to choose
$g(\cdot,\cdot, 0)$ as a multivariate polynomial that satisfies
$\max_{(\ux,\ueta)\in\cQ_x \times\cQ_\eta} | g(\ux,\ueta,0)
- h(\ux,\ueta,0) | \leq e$. It is readily checked that $p(\ux,l,0) =
  g_e(\ux,\ueta_l, 0)$ complies with the assumptions of
Proposition~\ref{amp-pscal}. 

We show the last result by recurrence on $t$.  
We first have $|R^1_i - \chR^1_i| \leq \sum_i s_{il} (h(x^0_l,\eta_l,0)^2 -
g_e(x^0_l,\eta_l,0)^2) \leq C e^2$. Assume that there exists a non-negative
function $\delta_t(e)$ that converges to zero as $e \to 0$, and such that 
$\| R^t_i - \chR^t_i \| \leq \delta_t(e)$. We have 
\begin{align*} 
| R^{t+1}_i(t+1,1) - \chR^{t+1}_i(t+1,1) | 
    &\leq \sum_l s_{il} | \EE h(Z^t_l,\eta_l, t) h(x^0_l,\eta_l, 0) - 
     \EE g_e(\cZ^t_l,\eta_l, t) g_e(x^0_l,\eta_l, 0) | \\
&\leq \sum_l s_{il} | \EE [h(Z^t_l,\eta_l, t) - h(\cZ^t_l,\eta_l, t) ] | \, 
   | h(x^0_l,\eta_l, 0) | \\
&\phantom{=} + \sum_l s_{il} | \EE h(\cZ^t_l,\eta_l, t) | \,  
   |  h(x^0_l,\eta_l, 0) - g_e(x^0_l,\eta_l, 0) | \\ 
&\phantom{=} + \sum_l s_{il} 
   | \EE [ h(\cZ^t_l,\eta_l, t) - g_e(\cZ^t_l,\eta_l, t) ] | \,  
   | g_e(x^0_l,\eta_l, 0) | . 
\end{align*} 
Using Assumptions~\ref{x0}, \ref{eta}, any of the two 
Assumptions~\ref{ass:h-lip} or~\ref{ass:h-gal}, the bound~\eqref{bnd-R}, and 
the recurrence assumption, we
obtain that the term at the second line can be bounded by some non-negative
function $\delta_{t+1}(e)$ converging to zero with $e$.  Indeed, when 
the triple $(\ueta,\sigma^2,\check\sigma^2)$ belongs to a compact set, 
the function $\varphi(\ueta, \sigma^2, \check\sigma^2) = \EE h(\uxi,\ueta,t) 
  - \EE h(\check\uxi,\ueta,t)$ with $\uxi \sim \cN(0,\sigma^2)$ and 
$\check\uxi \sim \cN(0, \check\sigma^2)$ is (uniformly) continuous on this
compact, and vanishes on the set $\{ (\ueta, \sigma^2, \check\sigma^2) 
\, : \, \sigma^2 = \check\sigma^2 \}$. 
One can readily check that the terms at the third and fourth line of the 
previous display can be bounded by $C e$ and $C \sqrt{e}$ respectively. We 
thus obtain that 
$| R^{t+1}_i(t+1,1) - \chR^{t+1}_i(t+1,1) | \leq \delta_{t+1}(e)$, by
possibly modifying the function $\delta_{t+1}$. A similar treatment can
be applied to $R^{t+1}_i(t+1,k) - \chR^{t+1}_i(t+1,k)$ for $k=2,\ldots,t+1$,
leading to $\| R_i^{t+1} - \chR_i^{t+1}\| \leq \delta_{t+1}(e)$ by 
possibly modifying $\delta_{t+1}(e)$ once again. This leads to the required
result. 
\end{proof} 

We now use the previous lemma for $e >0$ small to construct a polynomial AMP
sequence.  Starting with $\cx^0 = x^0$, this sequence denoted as
$(\cx^{t})_{t\in[\tmax]}$ is given as 
\begin{equation} 
\label{amp-ge} 
\cx^{t+1} = 
 W g_e(\cx^{t},\eta,t) - 
 \diag \Bigl(W^{\odot 2} \partial g_e(\cx^{t},\eta,t) \Bigr)  
   g_e(\cx^{t-1},\eta,t-1) . 
\end{equation} 
Proposition~\ref{amp-pscal} can be applied to this sequence, which implies in
particular that the convergences~\eqref{cvg-pscal} hold true for polynomial
test functions.  We now show that the polynomial test function can be replaced
with a continuous function that increases at most polynomially near the
infinity. 
\begin{lemma}
\label{poly->qque} 
Consider the polynomial AMP sequence $(\cx^{t})_{t\in[\tmax]}$ provided by
Algorithm~\eqref{amp-ge}.  Let the $n$--uple $(\beta_1,\ldots,\beta_n)$ be
as in the statements of Theorems~\ref{th-amp-lip} or~\ref{th-amp-gal}. Let 
$\varphi : \cQ_\eta \times \RR^{\tmax} \to \RR$ be a continuous function such
that $|\varphi(\alpha, u_1, \ldots, u_\tmax) | 
 \leq C(1 + |u_1|^m + \cdots + |u_\tmax|^m)$ for a given arbitrarily integer 
$m > 0$. Then, 
\[
\frac 1n \sum_{i\in[n]} 
  \beta_i \varphi( \eta_i, \cx^{1}_i, \ldots, \cx^{\tmax}_i) - 
  \beta_i \EE\varphi(\eta_i, \cZ^{1}_i, \ldots, \cZ^{\tmax}_i) 
 \toprobalong 0 . 
\]
\end{lemma} 
\begin{proof} 
Define the $\cP(\RR^{\tmax+2})$--valued random probability measure 
$\hat{\varrho}^{(n)}$ as 
\[
\hat{\varrho}^{(n)} = \frac 1n \sum_{i\in[n]} 
  \delta_{(\beta_i,\eta_i, \cx^1_i,\ldots, \cx^{\tmax}_i)} , 
\]
and the deterministic measure $\bs\varrho^{(n)}$ valued on the same space as
\[
\bs\varrho^{(n)} = \mcL((\beta_\theta,\eta_\theta,\cZ^{1}_\theta, \ldots, 
 \cZ^{\tmax}_\theta)), 
\]
where the random variable $\theta$ is uniformly distributed on $[n]$ and is
independent of $\{ \cZ^t_i \}_{t\in[\tmax],i\in[n]}$. Let
$P(\ueta,\ux^1,\ldots,\ux^\tmax)$ be a multivariate polynomial. Since the
elements $\eta_i$ of the vector $\eta$ are in $\cQ_\eta$, and since the
$\beta_i$ are bounded by assumption, the function 
$\psi : \RR^\tmax \times [n] \to \RR, (\ux^1,\ldots, \ux^\tmax,l) \mapsto
\beta_l P(\eta_l,\ux^1,\ldots,\ux^\tmax)$ complies with the assumptions of the
statement of Proposition~\ref{amp-pscal}.  Observing that 
$\int \underline\beta P(\ueta,\ux^1,\ldots,\ux^\tmax) \,
\bs\varrho^{(n)}(d\underline\beta, d\ueta,d\ux^1,\ldots,d\ux^\tmax) 
= n^{-1} \sum_i \beta_i \EE P(\eta_i, \cZ^1_i, \ldots, \cZ^\tmax_i)$, we 
obtain from Proposition~\ref{amp-pscal} that 
\begin{equation}
\label{cvg-mom} 
\int \underline\beta P(\ueta,\ux^1,\ldots,\ux^\tmax) \, d\hat{\varrho}^{(n)} 
  - \int \underline\beta P(\ueta,\ux^1,\ldots,\ux^\tmax) \, d{\bs\varrho}^{(n)} 
 \toprobalong 0 . 
\end{equation} 
We now show that this convergence remains true when $P$ is replaced with the
function $\varphi$ of the statement. 
Recalling the bound~\eqref{bnd-cR} and the bounds on the $\beta_i$ and the 
$\eta_i$, we know that for each sequence $(n)$ of integers, there is a 
subsequence (that will still denote as $(n)$) and a deterministic measure 
${\bs\varrho}^{\infty} \in \cP(\RR^{\tmax+2})$ such 
that ${\bs\varrho}^{(n)}$ converges narrowly to ${\bs\varrho}^{\infty}$. 
We shall show that $\hat{\varrho}^{(n)}$ converges narrowly in probability 
towards ${\bs\varrho}^{\infty}$, or, equivalently, that 
\begin{equation}
\label{charac}
\forall \bs\omega \in \RR^{\tmax +2}, \quad 
\Phi_{\hat{\varrho}^{(n)}}(\bs\omega) \toprobalong 
\Phi_{{\bs\varrho}^{\infty}}(\bs\omega), 
\end{equation} 
where $\Phi_{\hat{\varrho}^{(n)}}$ 
(respectively $\Phi_{{\bs\varrho}^{\infty}}$) is the characteristic function 
of $\hat{\varrho}^{(n)}$ (respectively ${\bs\varrho}^{\infty}$). This narrow
convergence, coupled with the moment convergence implied
by~\eqref{cvg-mom} along our subsequence, leads to the result of the lemma.  

Define respectively as 
$\hat{\varrho}^{(n)}_{\bs\omega}, {\bs\varrho}^{(n)}_{\bs\omega},
{\bs\varrho}^{\infty}_{\bs\omega} \in \cP(\RR)$ the push-forward of
$\hat{\varrho}^{(n)}, {\bs\varrho}^{(n)}$, and ${\bs\varrho}^{\infty}$ by the 
linear function $\ps{\bs\omega,\cdot}$.  Writing 
$\bs\omega = [ \omega_\beta, \omega_\eta, \bs\omega_x^\T ]^\T$ where 
$\omega_\beta$ and $\omega_\eta$ are scalars, the moment generating function
$\psi_{\bs\omega}^{(n)}(z)$ of ${\bs\varrho}^{(n)}_{\bs\omega}$ is given as 
\[
\psi_{\bs\omega}^{(n)}(z) 
= \frac 1n \sum_{i=1}^n \exp(
 (\omega_\beta \beta_i + \omega_\eta \eta_i) z 
    - {\bs\omega}_x^\T \chR^{\tmax}_i {\bs\omega}_x \, z^2 / 2) .
\]
Using the normal family theorem in conjunction with the 
boundedness of the $\beta_i$, the $\eta_i$ and the norms 
$\| \chR^{\tmax}_i \|$, we obtain that the moment generating function
$\psi_{\bs\omega}^{\infty}(z)$ of ${\bs\varrho}^{\infty}_{\bs\omega}$ is
holomorphic in a neighborhood of zero (even entire), thus,
${\bs\varrho}^{\infty}_{\bs\omega}$ is determined by its moments. 
Moreover, we obtain from the convergence~\eqref{cvg-mom} that each moment of 
$\hat{\varrho}^{(n)}_{\bs\omega}$ converges in probability to its analogue
for ${\bs\varrho}^{\infty}_{\bs\omega}$. This implies that 
$\hat{\varrho}^{(n)}_{\bs\omega}$ converges narrowly in probability towards
${\bs\varrho}^{\infty}_{\bs\omega}$ for each $\bs\omega$, which is equivalent 
to the convergence~\eqref{charac}. 
\end{proof} 

In Algorithm~\eqref{amp-ge}, the term $\diag (W^{\odot 2} \partial
g_e(\cx^{t},\eta,t) )$ can be approximated with the more manipulable term
$\diag (S \partial g_e(\cx^{t},\eta,t) )$. This will be a consequence of the 
next lemma. 
\begin{lemma} 
\label{W2->S} 
For each $t\in[\tmax - 1]$, there is a constant $C \geq 0$ such that for each 
$i\in[n]$, 
\[
\EE \Bigl[ \Bigl( \sum_{l\in[n]} (W_{il}^2 - s_{il}) 
  \partial g_e(\cx_l^{t},\eta_l,t) \Bigr)^4 \Bigr] \leq C / K^2 .
\]
\end{lemma}
The proof of this lemma follows closely the development of
\cite[Sec.~A.2]{bay-lel-mon-15}, with an adaptation to our variance profile
model very similar to what we did in the proof of Proposition~\ref{z-tz} 
above. We omit the details. \\

We are now in position to prove Theorems~\ref{th-amp-lip} and~\ref{th-amp-gal}. 
In the next two subsections, the following notational conventions will be 
adopted.  We shall most often write $g_e(\ux) = g_e(\ux,\ueta,t)$ and 
$h(\ux) = h(\ux,\ueta,t)$ for conciseness; The values of $\ueta$ and $t$ will 
be clear from the context. Denote as $\cE_W$ the event 
$\cE_W = [\| W \| \leq C_W ]$. Given a sequence $(\xi^{(n)})$ of non-negative 
random variables and a number $e > 0$, the notations $\xi^{(n)} \lessprob e$ 
and $\xi^{(n)} \lessprobEw e$ will stand respectively for 
$\PP[ \xi^{(n)} \geq e ] \to_n 0$ and 
$\PP[ [ \xi^{(n)} \geq e ] \cap \cE_W ] \to_n 0$. 
The relations $\lessprob$ and $\lessprobEw$ satisfy some obvious calculation 
rules, such as $\xi^{(n)}_1 + \xi^{(n)}_2 \lessprobEw e_1 + e_2$ when
$\xi^{(n)}_1 \lessprobEw e_1$ and $\xi^{(n)}_2 \lessprobEw e_2$. 
In both proofs, the function $\delta : (0, \epsilon) \to \RR_+$ defined for 
some $\epsilon > 0$ is a generic function, independent of $n$, such that 
$\delta(e) \to 0$ as $e\to 0$. This function can change from a display
to another. 

\subsection{Proof of Theorem~\ref{th-amp-lip}} 
\label{prf-th-lip} 

Given an arbitrarily small $e > 0$, let us construct the functions
$g_e(\cdot,\cdot,t)$ for $t = 0,\ldots, \tmax - 1$ as specified by
Lemma~\ref{ge}.  With these functions at hand, let $(\cx^1,\ldots, \cx^\tmax)$
be the iterates obtained by Algorithm~\eqref{amp-ge}. We shall compare the
iterates of Algorithm~\eqref{alg-amp} with those of the former, and show by
recurrence on $t = 1, \ldots, \tmax$ that 
\[
\| \cx^{t} - x^{t} \|_n \lessprobEw \delta(e), \quad \text{and} \quad 
\| h(x^t) - g_e(\cx^t) \|_n \lessprobEw \delta(e). 
\]
Starting with $t = 1$, since $x^1 - \cx^1 = W (h(x^0) - g_e(x^0))$, it holds 
that 
\[
\| x^1 - \cx^1 \|_n \leq 
 C_W \| h(x^0) - g_e(x^0) \|_n \leq C_W e 
\]
on the event $\cE_W$, by the construction of $g_e$ specified by Lemma~\ref{ge}. 
Using this bound along with the Lipschitz property of $h(\cdot,\eta_i,1)$ 
provided by Assumption~\ref{ass:h-lip}, we obtain 
\[
\| h(x^1) - g_e(\cx^1) \|_n \leq 
\| h(x^1) - h(\cx^1) \|_n + \| h(\cx^1) - g_e(\cx^1) \|_n 
\leq 
C \| x^1 - \cx^1 \|_n + \| h(\cx^1) - g_e(\cx^1) \|_n . 
\]
Writing 
\[
\frac 1n \sum_i (h(\cx^1_i) - g_e(\cx^1_i))^2 = 
\frac 1n \sum_i \left( (h(\cx^1_i) - g_e(\cx^1_i))^2 - 
 \EE (h(\cZ^1_i) - g_e(\cZ^1_i))^2  \right) 
+ \frac 1n \sum_i \EE (h(\cZ^1_i) - g_e(\cZ^1_i))^2 , 
\]
the first term at the right hand side converges to zero in probability by
Lemma~\ref{poly->qque}, and the second term is bounded by $e$ by construction
of the function $g_e(\cdot,\cdot,1)$. This establishes the recurrence property
for $t=1$. 

Let $t \in [\tmax - 1]$, and assume that 
$\| x^s - \cx^s \|_n \lessprobEw \delta(e)$ and 
$\| h(x^s) - g_e(\cx^s) \|_n \lessprobEw \delta(e)$ for $s = 1, \ldots, t$. 
To establish these bounds for $s = t+1$, we first show that 
\begin{equation}
\label{Ons} 
 \left\| \diag (W^{\odot 2} \partial g_e(\cx^{t}) ) g_e(\cx^{t-1}) - 
 \diag (S \EE \partial h(\bZ^t) ) h(x^{t-1}) \right\|_n 
 \lessprobEw \delta(e) .
\end{equation}
To this end, we write 
\begin{align*} 
& \diag (W^{\odot 2} \partial g_e(\cx^{t}) ) g_e(\cx^{t-1}) - 
 \diag (S \EE \partial h(\bZ^t) ) h(x^{t-1}) \\
 &= 
\diag \bigl( (W^{\odot 2} - S) \partial g_e(\cx^{t}) \bigr) g_e(\cx^{t-1})  
 + \diag \bigl(S \bigl(\partial g_e(\cx^{t}) - \EE\partial g_e(\cbZ^t)
  \bigr) \bigr) g_e(\cx^{t-1})  \\ 
&\phantom{=} 
 + \diag \bigl(S \EE \partial g_e(\cbZ^t) \bigr) 
  \bigl( g_e(\cx^{t-1}) - h(x^{t-1}) \bigr) 
 + \diag \bigl(S \bigl(\EE \partial g_e(\cbZ^t) - 
 \EE \partial h(\bZ^t) \bigr) \bigr) h(x^{t-1}) . 
\end{align*} 
We limit ourselves to $t \geq 2$, and omit the easy adaptations of the proof
to manage the terms $g_e(x^0)$ and $h(x^0)$ above when $t = 1$.  
Using Lemma~\eqref{W2->S} and the bound~\eqref{mom-cx}, the term $\chi_{1}$ 
defined as  
\[
\chi_{1} = 
\| \diag ((W^{\odot 2} - S) \partial g_e(\cx^{t}) ) g_e(\cx^{t-1}) \|_n^2 
 = \frac 1n \sum_{i\in[n]} 
  [ (W^{\odot 2} - S ) \partial g_e(\cx^{t}) ]_i^2 
    [g_e(\cx^{t-1})]_i^2  
\]
satisfies $\EE \chi_1 \to 0$ by the Cauchy-Schwarz inequality, thus, 
$\chi_1 \toprobashort 0$. 

We now deal with the next term $\chi_{2} = \|\diag (S (\partial g_e(\cx^{t}) 
 - \EE \partial g_e(\cbZ^t) ) ) g_e(\cx^{t-1}) \|_n^2$. 
For $i = i_n \in [n]$, write 
$\xi_i = [S (\partial g_e(\cx^{t}) - \EE \partial g_e(\cbZ^t) ]_i$. 
Applying the convergence~\eqref{cvg-small} in the statement of 
Proposition~\ref{amp-pscal} with $\cS = \cK_i = \{ j \in[n], \ s_{ij} > 0 \}$
and $\psi(\cx^1_j,\ldots,\cx^t_j,j) = K s_{ij} \partial g_e(\cx^{t}_j)$, 
we obtain that $\xi_{i_n} \toprobashortshort_n 0$. Furthermore, for each 
integer $m > 0$, we obtain from the bound~\eqref{mom-cx} and Minkowski's 
inequality that $\max_{j\in[n]} \EE | \xi_j |^m$ is bounded. 
Let $\varepsilon > 0$. Writing $\xi_i = \xi_i \1_{|\xi_i| \geq \epsilon} 
 + \xi_i \1_{|\xi_i| < \epsilon}$ and using Cauchy-Schwarz inequality
along with the bound~\eqref{mom-cx}, we obtain that 
\[
\EE \chi_2 \leq \frac Cn \sum_{i\in[n]} \PP[\xi_i \geq \varepsilon]^{1/2} 
 + C \varepsilon^2 , 
\]
thus, $\limsup_n \EE\chi_2 \leq C \varepsilon^2$, and since $\varepsilon$ is
arbitrary, $\EE\chi_2 \to 0$.  

Considering the term  
 $\chi_4 = \| \diag \bigl(S \bigl(\EE \partial g_e(\cbZ^t) - 
 \EE \partial h(\bZ^t) \bigr) \bigr) h(x^{t-1}) \|_n^2$, we have 
\begin{align} 
| [ S \EE \partial g_e(\cbZ^t) - S \EE \partial h(\bZ^t) ]_i | &\leq 
| [ S \EE \partial g_e(\cbZ^t) - S \EE \partial h(\cbZ^t) ]_i | + 
| [ S \EE \partial h(\cbZ^t) - S \EE \partial h(\bZ^t) ]_i | \nonumber \\ 
 &\leq Ce + \delta(e),  \label{ESgh} 
\end{align} 
where the bound $Ce$ on the first term comes from 
the construction of $g_e$ in Lemma~\ref{ge}, and the bound $\delta(e)$ on 
the second term is due to the inequality 
$\max_i \| \chR^{\tmax}_i - R^{\tmax}_i \| \leq \delta(e)$ stated by the same 
lemma. Writing 
$\| h(x^{t-1}) \|_n \leq \| h(x^{t-1}) - h(\cx^{t-1}) \|_n + 
 \| h(\cx^{t-1}) \|_n \leq C \| x^{t-1} - \cx^{t-1} \| + \| h(\cx^{t-1}) \|_n$
and using Proposition~\ref{poly->qque} to control $\| h(\cx^{t-1}) \|_n$, 
we obtain that there exists a constant $C > 0$ such that 
$\PP[ [ \| h(x^{t-1}) \|_n \geq C ] \cap \cE_W ] \to 0$. By consequence, 
$\chi_4 \lessprobEw \delta(e)$. 

We finally consider the term 
$\chi_3 = \| \diag \bigl(S \EE \partial g_e(\cbZ^t) \bigr) 
\bigl( g_e(\cx^{t-1}) - h(x^{t-1}) \bigr) \|_n^2$. By Inequality~\eqref{ESgh} 
and the bound~\eqref{bnd-R}, the real numbers 
$b_i = [ S \EE \partial g_e(\cbZ^t) ]_i^2$ are bounded. Using the recurrence 
assumption $\| g_e(\cx^{t-1}) - h(x^{t-1}) \|_n \lessprobEw \delta(e)$, we 
thus obtain that $\chi_3 \lessprobEw \delta(e)$. 

Gathering these results on $\chi_1$ to $\chi_4$, we obtain the convergence
expressed by~\eqref{Ons}. 

With this at hand, we obtain from Equations~\eqref{amp-ge} and~\eqref{alg-amp}
that 
\[
\| x^{t+1} - \cx^{t+1} \|_n \leq \| W \| \| h(x^t) - g_e(\cx^t) \|_n 
 + \left\| \diag (W^{\odot 2} \partial g_e(\cx^{t}) ) g_e(\cx^{t-1}) - 
 \diag (S \EE \partial h(\bZ^t) ) h(x^{t-1}) \right\|_n 
\]
which shows with the help of the recurrence assumption again that 
$\| x^{t+1} - \cx^{t+1} \|_n \lessprobEw \delta(e)$. Similarly to what we 
did for $\| h(x^{1}) - g_e(\cx^{1}) \|_n$, we also obtain that 
$\| h(x^{t+1}) - g_e(\cx^{t+1}) \|_n \lessprobEw \delta(e)$. 

We now have all the elements to show the convergence~\eqref{cvg-W2} provided
in the statement of the theorem. 

Let $\varphi\in\PL_2(\RR^{\tmax +1})$. 
Write $\bx_i = \begin{bmatrix} \eta_i, x^1_i, \ldots, x^\tmax_i
  \end{bmatrix}^\T$ and 
$\check{\bx}_i = \begin{bmatrix} \eta_i, \cx^1_i, \ldots, \cx^\tmax_i
  \end{bmatrix}^\T$, and let  
$\bu = \begin{bmatrix} \bx_1^\T, \ldots, \bx_n^\T \end{bmatrix}^\T$, and
$\check\bu = \begin{bmatrix} \check\bx_1^\T, \ldots, \check\bx_n^\T 
\end{bmatrix}^\T$. Then, by the Cauchy-Schwarz inequality, the inequality 
$(a+b+c)^2 \leq 3 (a^2 + b^2 + c^2)$, and Assumption~\ref{eta}, we have 
\begin{align*} 
\frac 1n \Bigl| \sum_{i\in[n]} 
 \beta_i \varphi(\eta_i, x^1_i, \ldots, x^\tmax_i) -
 \beta_i \varphi(\eta_i, \cx^1_i, \ldots, \cx^\tmax_i) \Bigr|  
&\leq \frac Cn \sum_{i\in[n]} \| \bx_i - \check{\bx}_i \|
   (1 + \| \bx_i \| + \| \check{\bx}_i \| )  \\
&\leq \frac Cn \| \bu - \check{\bu} \|
 \Bigl(\sum_{i\in[n]} (1 + \| \bx_i \| + \| \check{\bx}_i \| )^2 \Bigr)^{1/2}\\
&\leq C \| \bu - \check{\bu} \|_n (1 + \| \bu \|_n + \| \check{\bu} \|_n )  \\
&\leq C \Bigl(\sum_{t=1}^{\tmax} \| x^t - \cx^t \|_n \Bigr) 
 \Bigl( 1 + \sum_{t=1}^{\tmax} \| x^t \|_n + \| \cx^t \|_n \Bigr) . 
\end{align*}
We just showed that 
$\sum_{t=1}^{\tmax} \| x^t - \cx^t \|_n \lessprobEw \delta(e)$. We furthermore
have that 
\[
\forall t \in [\tmax], \quad
 \frac 1n \sum_{i\in[n]} (\cx^t_i)^2 - \EE (\cZ^t_i)^2 \toprobalong 0 
\]
by Proposition~\ref{amp-pscal}, thus, $\| \cx^t \|_n \lessprob C$ for 
each $t\in[\tmax]$. Also, $\| x^t \|_n \leq \| x^t - \cx^t \|_n + \| \cx^t \|_n
 \lessprobEw C$. It follows that 
\[
\frac 1n \Bigl| \sum_{i\in[n]} 
 \beta_i \varphi(\eta_i, x^1_i, \ldots, x^\tmax_i) -
 \beta_i \varphi(\eta_i, \cx^1_i, \ldots, \cx^\tmax_i) \Bigr|  
 \lessprobEw \delta(e) .
\]
Using Lemma~\ref{poly->qque} in conjunction with the bound 
\[
\Bigl|  
 \frac 1n \sum_{i\in[n]} 
  \beta_i \EE\varphi(\eta_i, \cZ^{1}_i, \ldots, \cZ^{t}_i) 
  - \beta_i \EE\varphi(\eta_i, Z^{1}_i, \ldots, Z^{t}_i) 
  \Bigr| \leq \delta(e) ,  
\]
we obtain that 
\[
\frac 1n \Bigl| \sum_{i\in[n]} 
 \beta_i \varphi(\eta_i, x^1_i, \ldots, x^\tmax_i) -
 \beta_i \EE\varphi(\eta_i, Z^{1}_i, \ldots, Z^{t}_i) \Bigr| 
 \lessprobEw \delta(e) , 
\]
and since $e$ is arbitrarily small, we obtain the convergence~\eqref{cvg-W2}. 
Theorem~\ref{th-amp-lip} is proven. 

\subsection{Proof of Theorem~\ref{th-amp-gal}}

Here also, our starting point will be a polynomial AMP sequence. 
Given a small $e > 0$, construct the functions $g_e(\cdot,\cdot,t)$ for $t =
0,\ldots, \tmax - 1$ as specified by Lemma~\ref{ge}.  With these functions at
hand, let $(\cx^1,\ldots, \cx^\tmax)$ be the iterates obtained by
Algorithm~\eqref{amp-ge}, starting with $\cx^0 = x^0$. 

We shall make use of this sequence in a different way than in the previous
subsection. As a matter of fact, we shall not be able to construct a ``true''
AMP sequence $(x^t)$ as in~\eqref{alg-amp} and compare it with the sequence
$(\cx^t)$ recursively on $t$ as we did in the previous section, because we have
lost the Lipschitz character of $h(\cdot, \eta_i, t)$. However, by writing for
each $t = 0, 1, \ldots, \tmax-1$, 
\[
\bar x^{(n),t+1} = 
  W^{(n)} h(\cx^{(n),t},\eta^{(n)},t)  - 
 \diag \left( S^{(n)} \EE \partial h(\bZ^{(n),t},\eta^{(n)},t)\right) 
  h(\cx^{(n),t-1},\eta^{(n)},t-1), 
\]
with the term $\diag(\cdots) h(\cdots)$ being zero for $t=0$, we shall be able 
to show that there exists a function $\delta_x(e) \geq 0$, defined for $e$ 
small enough, such that $\delta_x(e) \to 0$ as $e \to 0$, and such that 
\begin{equation}
\label{x-bx-delta} 
\forall t \in [\tmax], \quad 
  \| \cx^{t} - \bar x^{t} \|_n \lessprobEw \delta_x(e) . 
\end{equation} 
This will be the first step to prove Theorem~\ref{th-amp-gal}. 

We first establish
\begin{equation}
\label{POns} 
\left\| \diag (W^{\odot 2} \partial g_e(\cx^{t}) ) g_e(\cx^{t-1}) - 
 \diag (S \EE \partial h(\bZ^t) ) h(\cx^{t-1}) \right\|_n \lessprob \delta(e) 
\end{equation} 
for $t\in[\tmax - 1]$. To this end, we write 
\begin{align*} 
& \diag (W^{\odot 2} \partial g_e(\cx^{t}) ) g_e(\cx^{t-1}) - 
 \diag (S \EE \partial h(\bZ^t) ) h(\cx^{t-1}) \\
 &= 
\diag \bigl( (W^{\odot 2} - S) \partial g_e(\cx^{t}) \bigr) g_e(\cx^{t-1})  
 + \diag \bigl(S \bigl(\partial g_e(\cx^{t}) - \EE\partial g_e(\cbZ^t)
  \bigr) \bigr) g_e(\cx^{t-1})  \\ 
&\phantom{=} 
 + \diag \bigl(S \EE \partial g_e(\cbZ^t) \bigr) 
  \bigl( g_e(\cx^{t-1}) - h(\cx^{t-1}) \bigr) 
 + \diag \bigl(S \bigl(\EE \partial g_e(\cbZ^t) - 
 \EE \partial h(\bZ^t) \bigr) \bigr) h(\cx^{t-1}) , 
\end{align*} 
Assume for simplicity that $t \geq 2$ as in the previous proof. The terms 
$\chi_{1} = 
\| \diag ((W^{\odot 2} - S) \partial g_e(\cx^{t}) ) g_e(\cx^{t-1}) \|_n^2$ 
and $\chi_{2} = \|\diag (S (\partial g_e(\cx^{t}) 
 - \EE \partial g_e(\cbZ^t) ) ) g_e(\cx^{t-1}) \|_n^2$ are identical to their
analogues in the previous proof. To manage the term 
$\chi_3 = \| \diag \bigl(S \EE \partial g_e(\cbZ^t) \bigr) 
\bigl( g_e(\cx^{t-1}) - h(\cx^{t-1}) \bigr) \|_n^2$, we use 
Lemma~\ref{poly->qque} to obtain that 
\[
\left\| g_e(\cx^{t-1}) - h(\cx^{t-1}) \right\|_n^2 - 
 \frac 1n \sum_{i\in[n]} \EE (g_e(\cZ^{t-1}_i) - h(\cZ^{t-1}_i) )^2 
 \toprobalong 0 .
\]
Since $\EE (g_e(\cZ^{t-1}_i) - h(\cZ^{t-1}_i) )^2 \leq e$ by the construction 
of $g_e$ and the $[ S \EE \partial g_e(\cbZ^t) ]_i$ are bounded, we obtain
that $\chi_3 \lessprob \delta(e)$. The term 
 $\chi_4 = \| \diag \bigl(S \bigl(\EE \partial g_e(\cbZ^t) - 
 \EE \partial h(\bZ^t) \bigr) \bigr) h(\cx^{t-1}) \|_n^2$ can be managed by
a similar method. Details are omitted. This leads to~\eqref{POns}. 

On the event $\cE_W$, we furthermore have 
$\left\| W ( h(\cx^t) - g_e(\cx^t)) \right\| \leq 
C_W \left\| h(\cx^t) - g_e(\cx^t) \right\|$ for $t = 0,\ldots, \tmax - 1$. 
By working similarly as for $\chi_3$ above, we obtain that 
\[
\forall t \in [\tmax], \quad \left\| W ( h(\cx^t) - g_e(\cx^t)) \right\|_n 
  \lessprobEw C \sqrt{e} 
\]
(specificities for $t=0$ obvious). Combining this with the 
convergence~\eqref{POns}, we obtain the convergences~\eqref{x-bx-delta}. 

To pursue, let $\varphi : \cQ_\eta \times \RR^{\tmax} \to \RR$ be as in the
statement of Theorem~\ref{th-amp-gal}. From Lemma~\ref{poly->qque} and from 
the bound 
\[
\Bigl|  
 \frac 1n \sum_{i\in[n]} 
  \beta_i \EE\varphi(\eta_i, \cZ^{1}_i, \ldots, \cZ^{t}_i) 
  - \beta_i \EE\varphi(\eta_i, Z^{1}_i, \ldots, Z^{t}_i) 
  \Bigr| \leq \delta(e) ,  
\]
we obtain that there exists a function $\delta_\varphi(e) \geq 0$ converging
to zero as $e \to 0$, and such that 
\begin{equation}
\label{phi-cvg} 
\Bigl| \frac 1n \sum_{i\in[n]} 
 \beta_i \varphi( \eta_i, \cx^{1}_i, \ldots, \cx^{\tmax}_i) - 
 \beta_i \EE\varphi(\eta_i, Z^{1}_i, \ldots, Z^{\tmax}_i) 
  \Bigr| \lessprob \delta_\varphi(e) . 
\end{equation} 

With the help of these results, we can now construct the sequence of matrices
$(\tbX^{(n)})$ provided in the statement of the theorem.  Given a large enough
integer $k_0$, let $k\geq k_0$, and choose $e > 0$ small enough so that
$\delta_x(e) \vee \delta_\varphi(e) = 1/k$.  By~\eqref{x-bx-delta}, there 
exists a random matrix $\tensor[^k]{\cbX}{^{(n)}} = \Bigl[
\tensor[^k]{\cx}{^{(n),1}} \ \cdots \ \tensor[^k]{\cx}{^{(n),\tmax}} \Bigr] \in
\RR^{n\times\tmax}$ such that 
\[
 \PP\left[\left[ \left\| \tensor[^k]{\cx}{^{t+1}} -  
 \left( W h(\tensor[^k]{\cx}{^{t}})  - 
 \diag \left( S \EE \partial h(\bZ^t)\right) 
  h(\tensor[^k]{\cx}{^{t-1}}) \right) \right\|_n 
 \geq 1/k \right] \cap \cE_W \right] \tolong 0  
\]
for each $t =0,\ldots, \tmax - 1$, and by \eqref{phi-cvg}, 
\[
\PP\left[ \left| \frac 1n \sum_{i\in[n]} 
  \beta_i \varphi( \eta_i, [\tensor[^k]{\cx}{^1}]_i, \ldots, 
 [\tensor[^k]{\cx}{^{\tmax}}]_i ) - 
  \beta_i\EE\varphi(\eta_i, Z^{1}_i, \ldots, Z^{\tmax}_i) \right| 
 \geq 1/k \right] 
 \tolong 0 , 
\]
where the $[\tensor[^k]{\cx}{^{t}}]_i$ are the elements of the vector
$\tensor[^k]{\cx}{^{t}}$. 
If $k=k_0$, let $N_k \in \NN_*$ be the smallest integer such that for all 
$n \geq N_k$, the two probabilities above are upper bounded by $1/k$. If 
$k > k_0$, do the same thing for $N_k$ with the supplementary condition that 
$N_k > N_{k-1}$. Construct the sequence of matrices 
$( \tbX^{(n)} )_{n \geq 2}$, with 
$\tbX^{(n)} = \Bigl[ {\tx}{^{(n),1}} \ \cdots \ {\tx}{^{(n),\tmax}} 
 \Bigr] \in \RR^{n\times \tmax}$ in such a way that 
$\tbX^{(n)} = \tensor[^k]{\cbX}{^{(n)}}$ for $n = N_k, N_k +1, \ldots, 
N_{k+1} - 1$. Then, this sequence satisfies the properties provided in the 
statement of Theorem~\ref{th-amp-gal}. 

\section{Proof of Theorem~\ref{th-lvsym}} 
\label{prf-lvsym} 

As said in the introduction, $u_\star^{(n)}$ can be identified as the solution
of a LCP problem \cite{tak-livre96}. We recall herein the elements of the LCP
theory needed in this paper.  As above, we often drop the superscript $^{(n)}$. 
We shall also use the well-known notations $\succcurlyeq$, $\preccurlyeq$,
and $\prec$ to refer to element-wise inequalities between vectors. 
Recall that the notation $xy$ refers to the elementwise product of the
$\RR^n$--valued vectors $x$ and $y$. 

\subsection{The equilibrium $u_\star$ as the solution of a LCP problem} 

Given a matrix $B \in \RR^{n\times n}$ and a vector $b \in \RR^n$, the LCP
problem, denoted as $\LCP(B,b)$, consists in finding a couple of vectors
$(z,y) \in \RR^n \times \RR^n$ such that
\begin{align*}
& y = B z + b \succcurlyeq 0, \\
&  z \succcurlyeq 0, \\
& \ps{z,y} = 0. 
\end{align*}
When a solution $(z,y)$ exists and is unique, we write $z = \LCP(B,b)$.

Obviously, an equilibrium $\bar u$ of our LV dynamical system~\eqref{lv} is
defined by the system
\begin{align*}
& \bar u  \succcurlyeq 0,  \\
& \bar u \left( r + \left(\Sigma - I \right) \bar u \right)
  = 0 .
\end{align*}
Furthermore, the supplementary condition
\[
 r + \left(\Sigma - I \right) \bar u  \preccurlyeq 0
\]
turns out to be a necessary condition for the equilibrium $\bar u$ to be stable
in the classical sense of Lyapounov theory (see \cite[Chapter 3]{tak-livre96}
to recall the different notions of stability, and
\cite[Theorem~3.2.5]{tak-livre96} for this result). These three conditions can
be rewritten as
\begin{align*}
& \bar w = \left( I - \Sigma \right) \bar u - r \succcurlyeq 0,  \\
& \bar u \succcurlyeq 0,  \\
& \ps{\bar w, \bar u} = 0,  
\end{align*} 
in other words, the couple $(\bar u, \bar w)$ solves the problem
$\LCP(I - \Sigma, -r)$.

Theorem~\ref{th-lvsym} is more specific, since it asserts that the
ODE~\eqref{lv} has a (unique) \emph{globally stable} equilibrium when $\|
\Sigma \| < 1$.  Recalling that $\Sigma$ is a symmetric matrix, this can be
obtained from the following result:
\begin{proposition}[Lemma 3.2.2 and Theorem 3.2.1 of \cite{tak-livre96}]
Given a symmetric matrix $B\in\RR^{n\times n}$ and a vector $b\in\RR^n$,
consider the LV ODE
\[
\dot u(t) = u(t) \left( b + B u(t) \right) , \quad t \geq 0, \
 u(0) \in \RR_{*+}^n. 
\]
Then, the LCP problem $\LCP(-B,-b)$ has an unique solution for each $b \in
\RR^N$ if and only if $B$ is negative definite (notation $B < 0$). On the
domain where $B < 0, b \in \RR^n$, the function $x = \LCP(-B,-b)$ is
measurable.
Moreover, if $B < 0$, then for each $b \in \RR^N$, the ODE above has a
globally stable equilibrium $u$ given as $u = \LCP(-B,-b)$.
\end{proposition}
Considering our LV ODE~\eqref{lv}, the first part of Theorem~\ref{th-lvsym}
is true by setting 
\begin{equation} 
\label{ulcp} 
u_\star^{(n)} = \left\{\begin{array}{ll}
\LCP(I_n - \Sigma^{(n)}, -r^{(n)}) & \text{if } \| \Sigma^{(n)} \| < 1 , \\
 0 & \text{otherwise.} 
 \end{array}\right.   
\end{equation}

\subsection{Behavior of $\mu^{u_\star}$ by an AMP approach: proof outline}

In the remainder, given $x,y,z,v \in \RR^n$ and $A \in \RR^{n\times n}$, 
expressions such as $xAy$ or $x A yzv$ always denote $\RR^n$--valued vectors
with the bracketing always starting from the right. Thus, $xAy = x(Ay)$ and 
$xAyzv = x(A(yzv))$. 
 
Let us outline our approach for studying the asymptotic behavior of 
the measure $\mu^{u_\star^{(n)}}$. 

Let $S^{(n)} \in \RR_+^{n\times n}$ and $\eta^{(n)} \in \RR_+^{n}$ be
respectively a deterministic symmetric matrix and a deterministic vector that
comply with Assumptions~\ref{ass-A} and~\ref{eta} respectively. These objects
will be specified below. For the moment, we just assume that 
\begin{equation}
\label{bnd-S} 
\limsup_n \rownorm S^{(n)} \rownorm < 1 . 
\end{equation} 
Given the matrix $S$, let $W$ be random symmetric $n\times n$ matrix given by
Equation~\eqref{def-W}. In this equation, the matrix $X$ is precisely the one
used to construct the interaction matrix $\Sigma$ of our LV model.  Regarding
the function $h: \RR \times \cQ_\eta \times \NN \to \RR$ of
Section~\ref{def-amp}, set $h(\ux,\ueta,t) = (\ux + \ueta)_+$. Observe that 
this function is uniformly Lipschitz in $\ux$ and satisfies 
Assumption~\ref{ass:h-lip}. 
With a small notational modification, we shall consider that $h$ is a 
$\RR \times \cQ_\eta \to \RR$ function, and write 
\begin{equation} 
\label{def-h}
h(\ux,\ueta) = (\ux + \ueta)_+ . 
\end{equation} 
Let us consider a Gaussian sequence $(\bZ^{t} = [ Z^{t}_i ]_{i=1}^n)_{t}$ 
which distribution is determined by the 
$(S,h,\eta,1_n)$--state evolution equations. In particular, writing 
$a^{t}_i = \EE  (Z^{t}_i)^2$, the vector of variances 
$a^{t} = [a^{t}_i]_{i=1}^n$ satisfies the recursion 
\begin{equation}
\label{p-recurs} 
a^{1} = S \left(1 + \eta\right)^2, \quad \text{and} \quad  
a^{t+1} = S \EE \left(\sqrt{a^{t}} \xi + \eta\right)_+^2,  
\end{equation} 
with $\xi \sim \cN(0, I_n)$. 

Notice that $\EE \partial h(Z^{t}, \eta) = \PP[ Z^t + \eta \geq 0 ]
 = \PP[  \sqrt{a^{t}} \xi + \eta \geq 0 ]$ with $\partial h(x,\eta) = 
\partial h(\cdot,\eta)|_x$. 
Writing $\zeta^{t} = S \PP[  \sqrt{a^{t}} \xi + \eta \geq 0 ]$, we can apply
Theorem~\ref{th-amp-lip} to the AMP sequence in $t$ 
\begin{equation} 
\label{xt+1xt}
x^0 = 1, \quad x^{t+1} = W (x^t+\eta)_+ - \zeta^t (x^{t-1}+\eta)_+ .  
\end{equation} 
Thus, given a random variable $\theta$ uniformly distributed on $[n]$
and independent of $\bZ^t$, the distribution $\mu^{x^{t}}$ can be 
approximated by $\mcL(Z^t_\theta)$ in the sense that 
$\bs d_2(\mu^{x^{t}}, \mcL(Z^t_\theta)) \toprobashort 0$ as $n\to\infty$. 

In all what follows, when we say that a sequence $(y^t)_{t=1,2,\ldots}$ of
$\RR^n$--valued vectors converge to the vector $y$ as $t\to\infty$, we 
always consider that this convergence occurs in the max-norm 
$\| \cdot \|_\infty$ and is uniform in $n$. 

Using the condition $\limsup_n \rownorm S \rownorm < 1$, it is not difficult 
to show that the iterates $a^{t}$ converge as $t\to\infty$ to the vector $a$ 
defined as the unique solution of the identity 
\begin{equation} 
\label{pinf} 
a = S \EE \left(\sqrt{a} \xi + \eta \right)_+^2. 
\end{equation} 
On the one hand, this implies that for large values of $t$, we can replace the 
vector $\bZ^t$ with a vector $\bZ = [Z_i]_{i=1}^n \eqlaw \sqrt{a} \xi$, which 
implies that $\mu^{x^t}$ can be approximated by $\mcL( Z_\theta )$, where 
the uniformly distributed random variable $\theta$ on $[n]$ is independent of 
$\bZ$. On the other hand, $\zeta^t$ converges in $t$ to the vector $\zeta$ 
given as 
\begin{equation}
\label{zeta} 
\zeta = S \PP [ \sqrt{a} \xi + \eta \geq 0 ]. 
\end{equation} 
We can thus write 
\[
x^{t+1} = W (x^t+\eta)_+ - \zeta (x^{t-1}+\eta)_+ + \chi_1, 
\]
where $\chi_1 = (\zeta - \zeta^t) (x^{t-1}+\eta)_+$ is such that 
$\| \chi_1 \|_n$ is small for large $t$. 

Next, we need to show that the vectors $x^t$ and $x^{t+1}$ tend to become 
aligned as $t$ grows. To that end, building on a result of 
\cite{mon-ric-16,don-mon-16}, we show that the correlation vector $q^t$, 
defined as 
\[
q^{t} = \frac{\EE \bZ^{t} \bZ^{t-1}}{\sqrt{a^{t} a^{t-1}} },  
\] 
converges to the vector $1_n$ as $t\to\infty$. With these elements, the AMP 
iteration can be written as 
\[
x^t = W (x^t + \eta)_+ - \zeta (x^t + \eta)_+ + \chi_1 + \chi_2, 
\]
where 
\[
\chi_2 = \zeta \left[ (x^t + \eta)_+ - (x^{t-1} + \eta)_+ \right] 
   + x^t - x^{t+1} 
\]
has the property that $\| \chi_2 \|_n$ is small for large values of $t$. The 
next to last equation can be rewritten
\[
- (x^t + \eta)_- = W (x^t + \eta)_+ - (1+\zeta) (x^t + \eta)_+ 
    + \eta + \chi_1 + \chi_2, 
\]
which leads to the equivalent equation 
\begin{align} 
- (1+\zeta)^{-1/2} (x^t+\eta)_- &= (1+\zeta)^{-1/2} W (1+\zeta)^{-1/2} 
     (1+\zeta)^{1/2} (x^t + \eta)_+  
   - (1+\zeta)^{1/2} (x^t + \eta)_+ \nonumber \\
 &\phantom{=} + (1+\zeta)^{-1/2} \eta + \bs\varepsilon^t, 
\label{lcp-pert} 
\end{align}
with $\bs\varepsilon^t = (1+\zeta)^{-1/2} (\chi_1 + \chi_2)$. 

We know that when $\|\Sigma \| < 1$, the equilibrium vector of the
ODE~\eqref{lv} is $u_\star = \LCP(I-\Sigma, -r)$. On the other hand, the
previous equation shows that the vector $u^t = (1+\zeta)^{1/2} (x^t +
\eta)_+$ is a solution to the LCP problem $\LCP(I-\diag(1+\zeta)^{-1/2} W
\diag(1+\zeta)^{-1/2}, - (1+\zeta)^{-1/2} \eta - \bs\varepsilon^t)$. 
Thus, if we choose the matrix $W$ as 
$W = \diag(1+\zeta)^{1/2} \Sigma \diag(1+\zeta)^{1/2}$, or, equivalently, if
we put 
\begin{equation}
\label{eq-S}
S = \diag(1+\zeta) V \diag(1+\zeta),
\end{equation}
and furthermore, if we set 
\begin{equation}
\label{eq-eta}
\eta = (1 + \zeta)^{1/2} r,
\end{equation}
then we obtain that 
\[
u^t = \LCP(I - \Sigma, - r - \bs\varepsilon^t). 
\]
Thus, we constructed with the help of an AMP approach a ``perturbed'' version 
of $u_\star$ which empirical distribution can be evaluated with arbitrary 
precision for large $n$, and the perturbation $\bs\varepsilon^t$ can be
made ``arbitrarily small'' in the norm $\|\cdot\|_n$ for large $t$. 

Recalling Equations~\eqref{pinf} and~\eqref{zeta}, the choices made 
in~\eqref{eq-S} and~\eqref{eq-eta} result in the following system of 
$2n$ equations in the unknown vectors $a$ and $\zeta$:  
\begin{align*}
a &= (1+\zeta) V (1+\zeta) 
  \EE \left(\sqrt{a} \xi + (1+\zeta)^{1/2} r\right)_+^2 \\
\zeta &= (1+\zeta) V (1+\zeta) 
  \PP\left[\sqrt{a} \xi + (1+\zeta)^{1/2} r \geq 0\right]. 
\end{align*}
It is a bit more convenient to write 
\begin{equation}
\label{p-a}
p = \frac{a}{1 + \zeta},
\end{equation}
resulting in the equivalent system~\eqref{sys-lv}.  
Remembering that $\mu^{x^t} \simeq \mcL(Z_\theta)$ and observing that 
$u^t = (1+ \zeta)^{1/2}(x^t + (1+\zeta)^{1/2} r)_+$, we obtain from what 
precedes that $\mu^{u_\star} \simeq \mcL( (Y_\theta)_+ )$ where $Y_\theta$
is the random variable specified in the statement of Theorem~\ref{th-lvsym}.

To prove this theorem rigorously, we thus need to perform the 
following steps:
\begin{enumerate} 

\item Prove that under Hypothesis~\ref{hyp-V}, the system~\eqref{sys-lv} admits
an unique solution $(p, \zeta) \in \RR_+^n \times [0,1]^n$. This is the content
of Lemma~\ref{sys-uniq} below. With this solution, construct the matrix $S$ and
the vector $\eta$ as in~\eqref{eq-S} and~\eqref{eq-eta} respectively.  Note
that the bound~\eqref{bnd-S} is satisfied thanks to Hypothesis~\ref{hyp-V}. 

\item With the help of Theorem~\ref{th-amp-lip}, establish 
Equation~\eqref{lcp-pert} with a control on the error term $\bs\varepsilon^t$. 

\item We just saw that $u^t$ is a perturbed version of $u_\star$. 
To be able to approximate $\mu^{u_\star}$ with $\mu^{u^t}$, we need a  
LCP perturbation result. This will be provided by the work of 
Chen and Xiang in \cite{che-xia-07}, which will let us control the norm 
$\| u^t - u_\star \|$. 
\end{enumerate}

\subsection{Behavior of $\mu^{u_\star}$: Detailed proof} 

\begin{lemma}
\label{sys-uniq} 
The system~\eqref{sys-lv} admits an unique solution 
$(p^{(n)},\zeta^{(n)}) \in \RR_+^n \times [0,1]$, and this solution satisfies
$\sup_n \| p^{(n)} \|_\infty < \infty$. 
\end{lemma} 
This lemma will be proven in Section~\ref{prf-sys-uniq} below. 

We now consider the second step of the proof. 
Given the solution $(p,\zeta)$ of System~\eqref{sys-lv}, let $S \in
\RR_+^{n\times n}$ and $\eta\in\RR_{*+}^n$ be given by equations~\eqref{eq-S}
and~\eqref{eq-eta} respectively. Note that the bound~\eqref{bnd-S} is satisfied
thanks to Hypothesis~\ref{hyp-V}.  Define the random matrix $W =
S^{\odot {1/2}} \odot X$.  Define the function $h$ as in~\eqref{def-h}, and let
$x^0 = 1_n$. 

We these definitions, let us check that the assumptions leading to
Theorem~\ref{th-amp-lip} are satisfied.  Using Hypothesis~\ref{lv-sigma}, it is
clear that our matrix $W$ satisfies Assumptions~\ref{ass-X} and~\ref{ass-A}.
Trivially, $x^0 = 1_n$ satisfies Assumption~\ref{x0}.  Letting $r_{\min},
r_{\max} > 0$ be respectively the minimum and maximum values of the elements of
the compact $\cQ_r$ of Hypothesis~\ref{lv-r}, we obtain from
Equation~\eqref{eq-eta} that the elements of $\eta$ belong to the compact
interval $\cQ_\eta = [ \eta_{\min}, \eta_{\max} ] \subset \RR_{*+}$, where
$\eta_{\min} = r_{\min}$ and $\eta_{\max} = \sqrt{2} r_{\max}$, and
Assumption~\ref{eta} is satisfied.  Finally, one can readily check that the
function $h$ and the matrix $S$ satisfy Assumptions~\ref{ass:h-lip}
and~\ref{nondeg}. 

\begin{lemma}
\label{cvg-azq} 
As $t\to\infty$, the sequence $(a^t)$ converges to the vector 
$a\in \RR_{*+}^n$ given as the unique solution of Equation~\eqref{pinf}, 
the sequence $(\zeta^t)$ converges to $\zeta$, and 
the sequence $(q^t)$ of correlation vectors converges to $1_n$. 
\end{lemma} 

\begin{proof}
Let $\uxi \sim \cN(0,1)$, and define the function $f : \RR_+ \times \cQ_\eta
\to \RR_+, (\ua,\ueta) \mapsto \EE (\sqrt{\ua} \uxi + \ueta )_+^2$. A small
calculation that we omit reveals that the derivative $\partial f(\ua,\ueta) =
\partial f(\cdot, \ueta)|_{\ua}$ on $\RR_{*+}$ is given as $\partial f(\ua,
\ueta) = \PP [\sqrt{\ua} \uxi + \ueta \geq 0 ]$. In particular,
$f(\cdot,\ueta)$ is increasing. Therefore, $f(\ua, \ueta) \geq f(0, \ueta) =
\ueta^2 \geq \eta_{\min}$. We thus have from Assumption~\ref{nondeg} and 
Equations~\eqref{p-recurs} that
\[
a^1 \succcurlyeq c_S 1 \quad \text{and} \quad 
\forall t \geq 1, \ a^{t+1}  \succcurlyeq S \eta^2 \succcurlyeq 
   c_S \eta_{\min}^2 1 . 
\]
Moreover, since $\rownorm S \rownorm < 1$, and since
$f(\ua, \cdot)$ is also increasing, we have 
\[
a^{t+1} = S f(a^t, \eta) \preccurlyeq S f(a^t, \eta_{\max} 1) 
 \prec f(\| a^t \|_\infty, \eta_{\max}) 1 , 
\] 
thus, $\| a^{t+1} \|_\infty <  f(\| a^t \|_\infty, \eta_{\max})$.  For $\ua$
large, we have that $f(\ua, \eta_{\max}) = \ua \EE(\uxi + \eta_{\max} /
\sqrt{\ua})_+^2 \sim \ua / 2$. Noting from~\eqref{p-recurs} that $\| a^{1}
\|_\infty < ( 1 + \eta_{\max} )^2$, we readily obtain that for each $t$, the
elements of the vector $a^t$ stay in a compact interval $\cQ_a \subset
\RR_{*+}$ which is independent of $n$. 

Let us study the iterations $a^{t+1} = S f(a^t, \eta)$ given 
by~\eqref{p-recurs}. The vector-valued matrix function $F(a,\eta) = 
S f(a,\eta)$ satisfies 
$| [F(a,\eta) - F(a',\eta)]_i | \leq \rownorm S \rownorm \max_i 
| f(a_i,\eta_i) - f(a'_i,\eta_i) | \leq \rownorm S \rownorm \, 
\| a - a' \|_\infty$, where $a = [a_i]$, $a' = [a'_i]$, and $\eta = [\eta_i]$.
Thus $\| F(a,\eta) - F(a',\eta) \|_\infty \leq \rownorm S \rownorm 
\| a - a' \|_\infty$, and the convergence $a^t \to_t a$, where $a$ is
defined as the unique solution of Equation~\eqref{pinf}, is obtained by
Banach's fixed point theorem. That this convergence is uniform in $n$ 
in the norm $\|\cdot\|_\infty$ results from the inequality 
$\| a^{t} - a \|_\infty \leq \rownorm S \rownorm^{t-1} 
 \| a^{1} - a \|_\infty \leq \rownorm S \rownorm^{t-1} D_{\cQ_a}$, where 
$D_{\cQ_a}$ is the diameter of the compact $\cQ_a$.

Recall that $\zeta^t = S \PP [ \sqrt{a^t} \xi + \eta \geq 0 ]$.  The
convergence $\zeta^t \to_t \zeta$ uniformly in $n$ in the norm
$\|\cdot\|_\infty$ is due to the boundedness of $(\rownorm S^{(n)} \rownorm)$,
the continuity of the function $(\ua,\ueta) \mapsto \PP [ \sqrt{\ua} \uxi +
\ueta \geq 0 ]$ on the compact $\cQ_a \times \cQ_\eta$, and to the convergence 
of $a^t$ to $a$ as above. 

We now establish the convergence $q^t \to_t 1$. 
Let $[G_1,G_2]$ be a centered Gaussian vector such that $\EE G_1^2 = \EE
G_2^2 = 1$, and such that $\EE G_1G_2 = \uq \in [0,1]$. 
Writing $\cQ_a = [ a_{\min}, a_{\max} ]$, define the compact interval 
$\cQ_\sH = [ \eta_{\min} / \sqrt{a_{\max}}, \eta_{\max} / \sqrt{a_{\min}} ]$,
and define the continuous function $\sH : [0,1] \times \cQ_\sH \to [0,1]$ as 
\[
(\uq,\ub) \mapsto \sH(\uq,\ub) = \frac{\EE (G_1 + \ub)_+ (G_2 + \ub)_+}
   {\EE (G_1 + \ub)_+^2} . 
\]
A consequence of \cite[Lemma~38 and proof of Lemma~37]{mon-ric-16} is that for
each $\ub \in \cQ_\sH$, the function $\sH(\cdot,\ub)$ 
satisfies $\sH(q,\ub) > q$ for all $q \in [0,1)$, and $\sH(1,\ub) = 1$.

We also define the continuous function 
$H : [0,1] \times \cQ_\sH \times \cQ_\sH \to [0,1]$ as 
\[
(\uq,\ub,\ud) \mapsto H(\uq,\ub,\ud) = \frac{\EE (G_1 + \ub)_+ (G_2 + \ud)_+}
   {\left(\EE (G_1 + \ub)_+^2  \EE (G_2 + \ud)_+^2\right)^{1/2}} . 
\]
Given $t > 0$ and $k = t-1, t$, write the Gaussian vectors $\bZ^k$ from 
the $(S,h,\eta,1)$--state evolution equations as $\bZ^k = \sqrt{a^k} \xi^k$,
with $\xi^k \sim \cN(0, I_n)$, and with $q^t = \EE \xi^t \xi^{t-1}$. 
By the $(S,h,\eta,1)$--state evolution equations, we have 
\begin{align*} 
\EE \bZ^{t+1} \bZ^{t} &= S \EE (\bZ^{t} + \eta)_+ (\bZ^{t-1} + \eta)_+ \\
 &= S \left(\EE (\bZ^{t} + \eta)_+^2 \EE (\bZ^{t-1} + \eta)_+^2 \right)^{1/2} 
 H(q^t, \eta /\sqrt{a^t}, \eta / \sqrt{a^{t-1}} ) \\
&= S \left( f(a^t, \eta) f(a^{t-1},\eta) \right)^{1/2}  
 H(q^t, \eta /\sqrt{a^t}, \eta / \sqrt{a^{t-1}} ), 
\end{align*}  
and using the state evolution equations again, we obtain that 
\[
q^{t+1} = \left(a^{t+1} a^{t}\right)^{-1/2} S \ 
 (f(a^t,\eta) f(a^{t-1},\eta))^{1/2} 
  H(q^t, \eta /\sqrt{a^t}, \eta / \sqrt{a^{t-1}}) . 
\]
Notice that $H(\uq, \ub, \ub) = \sH(\uq, \ub)$.  By the continuity of $H$ on
$[0,1] \times \cQ_\sH \times \cQ_\sH$ and the uniform convergence $\| a^t - a
\|_\infty \to_t 0$, we obtain that $\| H(q^t, \eta/\sqrt{a^t},
\eta/\sqrt{a^{t-1}} ) - \sH(q^t, \eta/\sqrt{a}) \|_\infty \to_t 0$ uniformly in
$n$.  We also have that $\| (a^{t+1} a^t)^{-1/2} - a^{-1} \|_\infty \to_t 0$
uniformly, and by using the continuity of $f$ on $\cQ_a \times \cQ_\eta$, that
$\| (f(a^t,\eta) f(a^{t-1},\eta))^{1/2} - f(a,\eta) \|_\infty \to_t 0$
uniformly.  Using the boundedness of $\rownorm S \rownorm$ and observing that
$a^{-1} S f(a,\eta) =  1$, we obtain that 
$\| q^{t+1} - \sH(q^t, \eta/\sqrt{a}) \|_\infty \to_t 0$ uniformly in $n$. 


Since the set of zeros of the continuous function 
$(\uq,\ub) \mapsto \uq - \sH(\uq,\ub)$ on the compact $[0,1]\times \cQ_\sH$ is 
reduced to $\{ 1 \} \times \cQ_\sH$, we deduce from the last convergence that
$\| q^{(n),t} -  1_n \|_\infty \to_t 0$ uniformly in $n$. 
\end{proof} 

In the remainder of the proof, we reuse the notation $\lessprob$ that was
introduced before the subsection~\ref{prf-th-lip}. 
Similarly, $\delta : (0, \epsilon) \to \RR_+$ defined for 
some $\epsilon > 0$ is a generic function, independent of $n$, such that 
$\delta(e) \to 0$ as $e\to 0$. This function can change from a display
to another. 

Recalling that $W = \diag(1+\zeta)^{1/2} \Sigma \diag(1+\zeta)^{1/2}$, we have 
by Hypothesis~\ref{|Sig|} that $\limsup_n \| W \| < 2$ with probability one. 
Let us set $C_W = 2$ in the statement of Theorem~\ref{th-amp-lip}. 

Consider the AMP sequence~\eqref{xt+1xt}. Fix a small number $e > 0$, and 
choose the index $t$ in this sequence to be large enough (independently of 
$n$) so that 
\begin{align}
& \EE \| \bZ^{t+1} - \bZ^t \|_n^2 + \EE \| \bZ^{t} - \bZ^{t-1} \|_n^2 \leq e,
  \nonumber \\ 
& \| \zeta^t - \zeta \|_\infty \leq e , \quad \text{and} \nonumber \\
& \| a^t - a \|_\infty \leq e, \label{at-a} 
\end{align}
which is possible by Lemma~\ref{cvg-azq}, after noting that 
$\EE \| \bZ^{t} - \bZ^{t-1} \|_n^2 = n^{-1} \sum_i a^{t}_i + a^{t-1}_i 
 - 2 (a^{t}_i a^{t-1}_i)^{1/2} q^{t}_i$. 

Repeating the derivations that follow Equation~\eqref{xt+1xt}, we reach the 
equation~\eqref{lcp-pert}, and we obtain the bounds 
\begin{align*} 
\| (1+\zeta)^{-1/2} \chi_1 \|_n^2 &\leq 
  \| (\zeta - \zeta^t) (x^{t-1} + \eta)_+ \|_n^2 \\
 &\leq \| (\zeta - \zeta^t) \|_\infty^2 
\bigl( \| (x^{t-1} + \eta)_+ \|_n^2 - \EE \| (\bZ^{t-1} + \eta)_+ \|_n^2\bigr) 
  + \| (\zeta - \zeta^t) \|_\infty^2 \EE\| (\bZ^{t-1} + \eta)_+ \|_n^2 \\
 &\lessprob C e^2 
\end{align*} 
by Theorem~\ref{th-amp-lip} and Lemma~\ref{cvg-azq}, and 
\begin{align*} 
\| (1+\zeta)^{-1/2} \chi_2 \|_n^2 &\leq 
  2 \| (x^{t} + \eta)_+ - (x^{t-1} + \eta)_+ \|_n^2 + 
  2 \| x^{t+1}- x^{t} \|_n^2  \\ 
&\leq 
  2 \| (x^{t} + \eta)_+ - (x^{t-1} + \eta)_+ \|_n^2 - 
  2 \EE \| (\bZ^{t} + \eta)_+ - (\bZ^{t-1} + \eta)_+ \|_n^2 \\
 &\phantom{=} + 2 \EE \| (\bZ^{t} + \eta)_+ - (\bZ^{t-1} + \eta)_+ \|_n^2 \\
 &\phantom{=} 
       + 2 \| x^{t+1}- x^{t} \|_n^2  - 2 \EE \| \bZ^{t+1}- \bZ^{t} \|_n^2 \\ 
 &\phantom{=} + 2 \EE \| \bZ^{t+1}- \bZ^{t} \|_n^2 \\
 &\lessprob \delta(e), 
\end{align*} 
by Theorem~\ref{th-amp-lip} and Lemma~\ref{cvg-azq} again, and by noticing that 
$\EE \| (\bZ^{t} + \eta)_+ - (\bZ^{t-1} + \eta)_+ \|_n^2 + 
 \EE \| \bZ^{t+1}- \bZ^{t} \|_n^2 \leq e$ thanks to the Lipshitz property
of the function $\ux \mapsto (\ux+\ueta)_+$. 

All in all, we have that the vector $u^t = (1+\zeta)^{1/2}(x^t + \eta)_+$
satisfies $u^t = \LCP(I-\Sigma, -r-\bs\varepsilon^t)$ on the event 
$[\| \Sigma \| < 1]$, with 
\begin{equation}
\| \bs\varepsilon^t \|_n \lessprob \delta(e). 
\label{eps-amp-petit}
\end{equation} 
We now tackle the third step of the proof, applying the LCP perturbation result
of Chen and Xiang in \cite{che-xia-07}. By 
\cite[Th.~2.7 and Th.~2.8]{che-xia-07}, we have 
\[
\| u^t - u_\star \| \leq \left\| ( I - \Sigma )^{-1} \right\| 
  \| \bs\varepsilon^t \| , 
\]
thus, $\| u^{t} - u_\star \| \leq C \| \bs\varepsilon^t \|$ with probability
one for all $n$ large by Hypothesis~\ref{|Sig|}. 

Recall the definition of the Gaussian vector $Y = (1+\zeta) (\sqrt{p} \xi + r)$
in the statement of Theorem~\ref{th-lvsym}. To finish the proof of this
theorem, it remains to prove that 
\begin{equation} 
\label{cvg-pl} 
\forall \varphi\in\PL_2(\RR), \quad
\frac{1}{n} \sum_{i\in[n]} \varphi(u_{\star,i}) - \EE\varphi((Y_i)_+) 
\toprobalong 0 . 
\end{equation} 
Notice that $Y = (1+\zeta)^{1/2}(\sqrt{a} \xi + \eta) \eqlaw
(1+\zeta)^{1/2}(\bZ + \eta)$ by Equations~\eqref{eq-eta} and~\eqref{p-a}. 
For $\varphi\in\PL_2(\RR)$, we write 
\begin{align*} 
& \frac 1n \sum_{i\in[n]} \varphi(u_{\star,i}) 
  - \EE \varphi((1+\zeta_i)^{1/2}(Z_i + \eta_i)_+) \\
&= \frac 1n \sum_{i\in[n]} 
    \left( \varphi(u_{\star,i}) - \varphi(u^t_i) \right) 
  + \frac 1n \sum_{i\in[n]} \left( \varphi(u^t_i)  
   - \EE \varphi((1+\zeta_i)^{1/2}(Z^t_i + \eta_i)_+) \right) \\
 &\phantom{=} + \frac 1n \sum_{i\in[n]}  
   \EE \varphi((1+\zeta_i)^{1/2}(Z^t_i + \eta_i)_+)  
   - \EE \varphi((1+\zeta_i)^{1/2}(Z_i + \eta_i)_+)   \\
&= \epsilon_1 + \epsilon_2 + \epsilon_3. 
\end{align*} 
We have  
\[
|\epsilon_1| \leq \frac Cn \sum_{i\in[n]} | u_{\star,i} - u^t_i |
 ( 1 + 2 |u^t_i| + | u_{\star,i} - u^t_i | ) 
 \leq C \| u_{\star} - u^t \|_n ( 1 + 2 \| u^t \|_n + 
  \| u_{\star} - u^t \|_n ) 
\]
by Cauchy-Schwarz. Thus, with probability one, 
$|\epsilon_1| \leq C \|\bs\varepsilon^t\|_n 
 ( 1 + \| u^t \|_n + \|\bs\varepsilon^t\|_n )$ for all $n$ large. 
Applying Theorem~\ref{th-amp-lip} to $\| u^t \|_n$ and using the bound 
$\|\bs\varepsilon^t\|_n \lessprob \delta(e)$, we obtain that 
$|\epsilon_1| \lessprob \delta(e)$. 

We also have that 
\begin{equation}
\label{e2}
\epsilon_2 \toprobalong 0 
\end{equation}
by Theorem~\ref{th-amp-lip}.

To deal with $\epsilon_3$, we can write $\bZ^t = \sqrt{a^t} \xi$ and 
$\bZ = \sqrt{a} \xi$, with $\xi \sim \cN(0,I_n)$, and use the pseudo-Lipschitz
property of $\varphi$ along with the bound~\eqref{at-a} to show after
a small derivation that $|\epsilon_3| \leq Ce$. 

Putting these bounds together, we obtain that 
\[
\Bigl| \frac{1}{n} \sum_{i\in[n]} \varphi(u_{\star,i}) - \EE\varphi((Y_i)_+) 
 \Bigr| \lessprob \delta(e) . 
\]
Since $e$ is arbitrary, the convergence~\eqref{cvg-pl} follows, and 
Theorem~\ref{th-lvsym} is proven.

\subsection{Proof of Lemma~\ref{sys-uniq}} 
\label{prf-sys-uniq} 
Given $\zeta \in [0,1]^n$, Equation~\eqref{var-lcp} is 
$p = V(1+\zeta)^2 f(p, r)$, where $f$ is the function introduced in the proof 
of Lemma~\ref{cvg-azq}. We have that $2 f (\underline p, \underline r) / 
 \underline p \to 1$ as $\underline p \to \infty$, uniformly in 
$\underline r \in \cQ_r$ from Hypothesis~\ref{lv-r}. Thus, there exists 
$p_{\max} > 0$ such that 
$f(\underline p, \underline r) \leq 3 \underline p / 4$ for each 
$\underline p \geq p_{\max}$ and each $\underline r \in \cQ_r$.  By 
consequence, if $\| p \|_\infty > p_{\max}$, then 
$\| V(1+\zeta)^2 f(p, r) \|_\infty \leq (3 \| p \|_\infty / 4) 
\| V(1+\zeta)^2 1_n \|_\infty < (3 \| p \|_\infty / 4)$ by 
Hypothesis~\ref{hyp-V}, thus, $V(1+\zeta)^2 f(p, r)$ cannot be equal to $p$ 
when $p \not\in [0, p_{\max}]^n$.  On the other hand, if 
$p \in [0, p_{\max}]^n$, it holds that $\| V(1+\zeta)^2 f(p, r) \|_\infty \leq 
f(p_{\max}, r_{\max}) \| V(1+\zeta)^2 1_n \|_\infty \leq 3 p_{\max} / 4$, 
thus, $V(1+\zeta)^2 f(p, r) \in [ 0, p_{\max}]^n$. 

Turning to Equation~\eqref{ons-lcp}, for each $\zeta \in [0,1]^n$, we see that
$(1+\zeta) V (1+\zeta) \PP\left[\sqrt{p} \xi + r \geq 0\right] \subset [0,1]^n$
for each $p$ by Hypothesis~\ref{hyp-V} again.  

Thus, writing the system \eqref{sys-lv} as $(p, \zeta) = G(p, \zeta)$, we
obtain that $G([0, p_{\max}]^n \times [0,1]^n) \subset [0, p_{\max}]^n \times
[0,1]^n$, and furthermore, $G$ does not have a fixed point outside 
$(\RR_+^n\setminus [0, p_{\max}]^n) \times [0,1]^n$. Since 
$[0, p_{\max}]^n \times [0,1]^n$ is a compact convex set of $\RR^{2n}$, we 
obtain by Brouwer's fixed point theorem that $G$ has a fixed point in this set.
In particular, when $(p,\zeta)$ is a fixed point, $\| p \|_\infty < p_{\max}$.  

To complete the proof, we need to show that this fixed point is unique. To
that end, we rely on the construction of the previous paragraph, where
we note that this uniqueness is never used. 

In all the remainder of this proof, the integer $n$ is fixed.  Choose a
solution $(p,\zeta)$ of the system~\eqref{sys-lv}. From this solution,
construct the matrix $S^{(n)}$ and the vector $\eta^{(n)}$ according
to~\eqref{eq-S} and~\eqref{eq-eta} respectively.  Let $(M)$ be a sequence of
integers converging to infinity. For each $M$, construct the symmetric matrix
$\bS^{(M)} \in \RR_+^{nM \times nM}$ and the vector 
$\bs\eta^{(M)} \in \RR_{*+}^{nM}$ as 
\[
\bS^{(nM)} = S^{(n)} \otimes \left(M^{-1}  1_M  1_M^\T\right), 
 \quad \text{and} \quad 
\bs\eta^{(nM)} = \eta^{(n)} \otimes  1_M ,  
\]
where $\otimes$ is the Kronecker product. 
Let $\bX^{(nM)} = [ \bX^{(nM)}_{ij} ]_{1\leq i,j \leq nM}$ be a real random 
symmetric $nM \times nM$ matrix such that the random variables 
$\{ \bX^{(nM)}_{ij} \}_{1\leq i < j\leq nM}$ are independent $\cN(0,1)$ random 
variables, and such that $\bX^{(nM)}_{ii} = 0$ for $i\in[nM]$. Define the 
random matrix $\bW^{(nM)}$ as 
\[
\bW^{(nM)} = (\bS^{(nM)})^{\odot {1/2}} \odot \bX^{(nM)} . 
\]
We shall consider herein the $\RR^{nM}$--valued AMP iterates based on the
matrix $\bW^{(nM)}$ and on the $(\bS^{(nM)},h,\bs\eta^{(nM)},1_{nM})$--state 
evolution equations, which take the form 
\[
\bu^{(nM),t+1} = \bW^{(nM)} \left( \bu^{(nM),t} + \bs\eta^{(nM)} \right)_+ 
 - \bs\zeta^{(nM),t} \left( \bu^{(nM),t-1} + \bs\eta^{(nM)} \right)_+ 
\]
(expression of $\bs\zeta^{(nM),t}$ omitted). 
In this context, one can check that Assumptions~\ref{ass-X}--\ref{ass:h-lip}
and~\ref{nondeg} applied to this model are satisfied with $n$ and $K_n$ there 
replaced with $nM$ and $K_n M$ respectively. Write 
\[
\bV^{(nM)} = V^{(n)} \otimes (M^{-1} 1_M 1_M^\T) 
 \quad \text{and} \quad 
\br^{(nM)} = r^{(n)} \otimes 1_M 
\]
(so that $\bS = \diag(1_{nM} + \bs\zeta) \bV \diag(1_{nM} + \bs\zeta)$ with 
$\bs\zeta = \zeta \otimes 1_M$ and $\bs\eta=(1_{nM} + \bs\zeta)^{1/2} \br$), 
and let 
\[
\bSigma^{(nM)} = (\bV^{(nM)})^{\odot {1/2}} \odot \bX^{(nM)} . 
\]
Observe that since $n$ is now fixed, our matrices $\bV^{(nM)}$ are no more 
sparse, and $\| \bV^{(nM)} \|_\infty \lesssim 1/M$. Thus, recalling 
the spectral norm controls made above in the Gaussian case, the positive
number 
\[
\bT_{\text{Gauss}}^{(nM)} = (1+\varepsilon) 
  \left( 2 \rownorm \bV^{(nM)} \rownorm^{1/2} 
  + \frac{6}{\sqrt{\log(1+\varepsilon)}} 
  (\| \bV^{(nM)} \|_\infty \log (nM))^{1/2}
 \right) 
\]
satisfies $\limsup_M \bT_{\text{Gauss}}^{(nM)} < 1$ by choosing $\varepsilon$
small enough. Thus, by the Gaussian concentration such as in~\eqref{gauss-con},
we obtain that $\limsup_M \| \bSigma^{(nM)} \| < 1$ with the probability one. 
By consequence, on this event, the equation 
$\bu_\star^{(nM)} = \LCP(I - \bSigma^{(nM)}, - \br^{(nM)})$ is well-defined
for all $M$ large. It is important to note that this vector does not depend
on the chosen solution $(p,\zeta)$ of the system~\eqref{sys-lv}. 

Write $\bu_\star^{(nM)} = \bigl[ \bu_{\star}^{(nM)}(1)^\T, \ldots, 
 \bu_{\star}^{(nM)}(n)^\T \bigr]^\T$ where 
$\bu_{\star}^{(nM)}(i) = \bigl[ \bu_{\star,1}^{(nM)}(i), \ldots, 
\bu_{\star,M}^{(nM)}(i) \bigr]^\T$. 
Recall the definition of the $\RR^n$--valued Gaussian vector $Y^{(n)}$ as
provided in the statement of Theorem~\ref{th-lvsym}. 
By repeating the argument of the previous paragraph, by relying this time on 
the AMP sequence $(\bs u^{(nM),t})_t$, we are able to show that 
\[
\forall \varphi\in\PL_2(\RR), \quad
\frac{1}{nM} \sum_{i\in[n]} \sum_{j\in[M]}  \varphi(\bu_{\star,j}^{(nM)}(i)) 
\xrightarrow[M\to\infty]{\mathcal P} 
\frac 1n \sum_{i\in[n]} \EE\varphi((Y_i)_+) .
\]
However, we need here a bit more than this convergence, which requires a slight 
modification of the approach of the previous paragraph. By relying on our new
AMP construction, we have that for each $e > 0$, there exists an integer
$t > 0$ and a random $\RR^{nM}$--valued error vector $\bs\varepsilon^{(nM),t}$ 
such that 
\[
\bigl\| \bu_\star^{(nM)} - \bu^{(nM),t} \bigr\| \leq 
 \bigl\| \bigl(I - \bSigma^{(nM)} \bigr)^{-1} \bigr\| \  
\bigl\| \bs\varepsilon^{(nM),t} \bigr\| 
\]
with 
\[
\PP\bigl[ \bigl\| \bs\varepsilon^{(nM),t} \bigr\|_{nM}^2 \geq e \bigr]  
\xrightarrow[M\to\infty]{} 0 
\]
Writing $\bu^{(nM),t} = \bigl[ \bu^{(nM),t}(1)^\T,\ldots, 
 \bu^{(nM),t}(n)^\T \bigr]^\T$ with 
$\bu^{(nM),t}(i) = \bigl[ \bu_1^{(nM),t}(i), \ldots, \bu_{M}^{(nM),t}(i) 
  \bigr]^\T$ and remembering that $n$ is fixed, this implies that there is 
$C > 0$ such that
\[
\forall i \in [n], \ 
\PP\bigl[ \bigl\| \bu_{\star}^{(nM)}(i) - \bu^{(nM),t}(i)  \bigr\|_M^2 
 \geq C n e \bigr]  \xrightarrow[M\to\infty]{} 0 . 
\]
This is the analogue of the convergence \eqref{eps-amp-petit}. 

Furthermore, let $\varphi\in\PL_2(\RR)$, let $i \in [n]$, and define the 
$nM$--uple 
\[
(\beta^{(nM)}_1, \ldots \beta^{(nM)}_{nM}) = 
( 0, \ldots, 0, \! \begin{array}[t]{c} \underbrace{1,1,\ldots 1,} \\ 
 \text{length } M \end{array}\! 0, \ldots, 0),
\]  
where the first element of the $M$--uple $(1,\ldots,1)$ is at the 
$((i-1)M+1)^{\text{th}}$ place. By applying Theorem~\ref{th-amp-lip} with these 
weights, we obtain 
\[
\forall i \in [n], \quad 
\frac 1M \sum_{l=1}^M \varphi(\bu^{(nM),t}_l(i)) 
 - \EE \varphi( (1+\zeta_i)^{1/2} (\sqrt{a_i^t} \uxi + \eta_i)_+) 
 \xrightarrow[M\to\infty]{{\mathcal P}} 0 ,  
\]
where $\uxi \sim \cN(0,1)$, and where the vector $a^t = [ a^t_i ]_{i=1}^n$ is 
precisely the one given by the recursion~\eqref{p-recurs}. This is the 
analogue of the convergence~\eqref{e2}. 

Completing the argument of the previous paragraph with these new convergences,
we obtain that 
\[
\forall i \in [n], \ 
\mu^{\bu_{\star}^{(nM)}(i)} 
\xrightarrow[M\to\infty]{\mathcal P} 
  \mcL\left((1+\zeta_i) (\sqrt{p_i} \uxi + r_i)_+\right) \quad 
 \text{in } \cP_2(\RR) . 
\]
From the uniqueness of these limits, we deduce that the solution 
$(p = [p_i]_{i=1}^n,\zeta= [\zeta_i]_{i=1}^n)$ of the system~\eqref{sys-lv} is 
unique. 

\subsection*{Acknowledgements}
I would like to thank C\'edric Gerbelot, Marc Lelarge, and the 
authors of \cite{akj-etal-(arxiv)22} who were all members of the French 
CNRS 80 prime project KARATE, for the inspiring and stimulating discussions. 

\appendix

\section{Proof of Proposition~\ref{Wfini}} 
\label{anx-Wfini} 
Assume that $\rho > 0$. Following the notations of~\cite{ban-vha-16}, write 
$\sigma = \max_{i\in[n]} (\sum_{j\in[n]} s_{ij})^{1/2}$ and 
$\sigma_* = \max_{ij} \sqrt{s_{ij}}$. 
Then, the proof of \cite[Cor.~3.5]{ban-vha-16} shows that 
\[
\Bigl( \EE \| W \|^{2\lceil \log n\rceil} \Bigr)^{1/(2\lceil\log n\rceil)} 
\leq C \left( \sigma + \sigma_* (\log n)^{(\rho\vee 1)/2} \right) 
\leq C \left( 1 + \sqrt{\frac{(\log n)^{\rho\vee 1}}{K_n}} \right), 
\]
where the second inequality is due to our Assumption~\ref{ass-A}. Using 
Markov's inequality and the hypothesis $K_n \gtrsim (\log n)^{\rho \vee 1}$, 
we obtain that
\[
\PP[ \| W \| \geq \eta ]^{1/(2\lceil\log n\rceil)} \leq C / \eta 
\]
for any $\eta > 0$. Choosing $\eta$ large enough, the result follows from
the Borel-Cantelli lemma. 

If $\rho = 0$, we can just apply the concentration results provided by
\cite[Cor~3.12 and Rem.~3.13]{ban-vha-16}.


\begin{thebibliography}{AHMN23}

\bibitem[ABC{\etalchar{+}}22]{akj-etal-(arxiv)22}
I.~Akjouj, M.~Barbier, M.~Clenet, W.~Hachem, M.~Ma{\"\i}da, F.~Massol,
  J.~Najim, and V.~C. Tran.
\newblock Complex systems in ecology: A guided tour with large
  {L}otka-{V}olterra models and random matrices.
\newblock {\em arXiv preprint arXiv:2212.06136}, 2022.

\bibitem[AHMN23]{akj-hac-mai-naj-(arxiv)23}
I.~Akjouj, W.~Hachem, M.~Ma{\"\i}da, and J.~Najim.
\newblock Equilibria of large random {L}otka-{V}olterra systems with vanishing
  species: a mathematical approach.
\newblock {\em arXiv preprint arXiv:2302.07820}, 2023.

\bibitem[AT12]{all-tan-12}
S.~Allesina and S.~Tang.
\newblock Stability criteria for complex ecosystems.
\newblock {\em Nature}, 483(7388):205--208, 2012.

\bibitem[AT15]{all-tan-15}
S.~Allesina and S.~Tang.
\newblock The stability–complexity relationship at age 40: a random matrix
  perspective.
\newblock {\em Population Ecology}, 57(1):63--75, 2015.

\bibitem[BBC18]{bir-bun-cam-18}
G.~Biroli, G.~Bunin, and C.~Cammarota.
\newblock Marginally stable equilibria in critical ecosystems.
\newblock {\em New Journal of Physics}, 20(8):083051, 2018.

\bibitem[BGP14]{ben-pec-14}
F.~Benaych-Georges and S.~P\'{e}ch\'{e}.
\newblock Largest eigenvalues and eigenvectors of band or sparse random
  matrices.
\newblock {\em Electron. Commun. Probab.}, 19:no. 4, 9, 2014.

\bibitem[BLM15]{bay-lel-mon-15}
M.~Bayati, M.~Lelarge, and A.~Montanari.
\newblock Universality in polytope phase transitions and message passing
  algorithms.
\newblock {\em Ann. Appl. Probab.}, 25(2):753--822, 2015.

\bibitem[BM11]{bay-mon-11}
M.~Bayati and A.~Montanari.
\newblock The dynamics of message passing on dense graphs, with applications to
  compressed sensing.
\newblock {\em IEEE Transactions on Information Theory}, 57(2):764--785, 2011.

\bibitem[BMN19]{ber-mon-ngu-19}
R.~Berthier, A.~Montanari, and P.-M. Nguyen.
\newblock {State evolution for approximate message passing with non-separable
  functions}.
\newblock {\em Information and Inference: A Journal of the IMA}, 9(1):33--79,
  01 2019.

\bibitem[Bol14]{bol-14}
E.~Bolthausen.
\newblock An iterative construction of solutions of the {TAP} equations for the
  {S}herrington--{K}irkpatrick model.
\newblock {\em Communications in Mathematical Physics}, 325(1):333--366, 2014.

\bibitem[BR22]{beh-ree-pmlr22}
J.~K. Behne and G.~Reeves.
\newblock Fundamental limits for rank-one matrix estimation with groupwise
  heteroskedasticity.
\newblock In {\em Proceedings of The 25th International Conference on
  Artificial Intelligence and Statistics}, volume 151 of {\em Proceedings of
  Machine Learning Research}, pages 8650--8672. PMLR, 28--30 Mar 2022.

\bibitem[BSHM17]{bus-etal-17}
D.~M. Busiello, S.~Suweis, J.~Hidalgo, and A.~Maritan.
\newblock Explorability and the origin of network sparsity in living systems.
\newblock {\em Scientific reports}, 7(1):12323, 2017.

\bibitem[Bun17]{bun-17}
G.~Bunin.
\newblock Ecological communities with {L}otka-{V}olterra dynamics.
\newblock {\em Phys. Rev. E}, 95:042414, Apr 2017.

\bibitem[BvH16]{ban-vha-16}
A.~S. Bandeira and R.~van Handel.
\newblock {Sharp nonasymptotic bounds on the norm of random matrices with
  independent entries}.
\newblock {\em The Annals of Probability}, 44(4):2479 -- 2506, 2016.

\bibitem[CEN22]{cle-fer-naj-22}
M.~Clenet, H.~{El Ferchichi}, and J.~Najim.
\newblock Equilibrium in a large {L}otka–{V}olterra system with pairwise
  correlated interactions.
\newblock {\em Stochastic Processes and their Applications}, 153:423--444,
  2022.

\bibitem[CL21]{che-lam-21}
W.-K. Chen and W.-K. Lam.
\newblock {Universality of approximate message passing algorithms}.
\newblock {\em Electronic Journal of Probability}, 26(none):1 -- 44, 2021.

\bibitem[CNM22]{cle-naj-mas-(arxiv)22}
M.~Clenet, J.~Najim, and F.~Massol.
\newblock Equilibrium and surviving species in a large {L}otka-{V}olterra
  system of differential equations.
\newblock {\em ArXiv preprint}, arXiv:2205.00735, 2022.

\bibitem[CPS09]{cot-pan-sto-livre09}
R.~W. Cottle, J.-S. Pang, and R.~E. Stone.
\newblock {\em The linear complementarity problem}, volume~60 of {\em Classics
  in Applied Mathematics}.
\newblock Society for Industrial and Applied Mathematics (SIAM), Philadelphia,
  PA, 2009.
\newblock Corrected reprint of the 1992 original [ MR1150683].

\bibitem[CX07]{che-xia-07}
X.~Chen and S.~Xiang.
\newblock Perturbation bounds of {$P$}-matrix linear complementarity problems.
\newblock {\em SIAM J. Optim.}, 18(4):1250--1265, 2007.

\bibitem[DLS22]{dud-lu-sen-(arxiv)22}
R.~Dudeja, Y.~M. Lu, and S.~Sen.
\newblock Universality of approximate message passing with semi-random
  matrices.
\newblock {\em arXiv preprint arXiv:2204.04281}, 2022.

\bibitem[DM16]{don-mon-16}
D.~Donoho and A.~Montanari.
\newblock High dimensional robust {M}-estimation: asymptotic variance via
  approximate message passing.
\newblock {\em Probab. Theory Related Fields}, 166(3-4):935--969, 2016.

\bibitem[Fan22]{fan-22}
Z.~Fan.
\newblock Approximate message passing algorithms for rotationally invariant
  matrices.
\newblock {\em Ann. Statist.}, 50(1):197--224, 2022.

\bibitem[FSW18]{fou-str-wel-(livre)18}
D.~Fourdrinier, W.~E. Strawderman, and M.~T. Wells.
\newblock {\em Shrinkage estimation}.
\newblock Springer Series in Statistics. Springer, Cham, 2018.

\bibitem[FVRS22]{fen-etal-(now)22}
O.~Y. Feng, R.~Venkataramanan, C.~Rush, and R.~J. Samworth.
\newblock A unifying tutorial on approximate message passing.
\newblock {\em Foundations and Trends® in Machine Learning}, 15(4):335--536,
  2022.

\bibitem[Gal18]{gal-18}
T.~Galla.
\newblock Dynamically evolved community size and stability of random
  {L}otka-{V}olterra ecosystems.
\newblock {\em Europhysics Letters}, 123(4):48004, sep 2018.

\bibitem[GB21]{ger-ber-arxiv22}
C.~Gerbelot and R.~Berthier.
\newblock Graph-based {A}pproximate {M}essage {P}assing iterations.
\newblock {\em ArXiv preprint}, arXiv:2109.11905, 2021.

\bibitem[GKKZ22]{gui-ko-krz-zde-(arxiv)22}
A.~Guionnet, J.~Ko, F.~Krzakala, and L.~Zdeborov{\'a}.
\newblock Low-rank matrix estimation with inhomogeneous noise.
\newblock {\em arXiv preprint arXiv:2208.05918}, 2022.

\bibitem[JM13]{jav-mon-13}
A.~Javanmard and A.~Montanari.
\newblock State evolution for general approximate message passing algorithms,
  with applications to spatial coupling.
\newblock {\em Information and Inference: A Journal of the IMA}, 2(2):115--144,
  2013.

\bibitem[Kho08]{kho-08}
O.~Khorunzhiy.
\newblock Estimates for moments of random matrices with {G}aussian elements.
\newblock In {\em S\'{e}minaire de probabilit\'{e}s {XLI}}, volume 1934 of {\em
  Lecture Notes in Math.}, pages 51--92. Springer, Berlin, 2008.

\bibitem[LJM09]{li-etal-09}
X.~Li, D.~Jiang, and X.~Mao.
\newblock Population dynamical behavior of {L}otka-{V}olterra system under
  regime switching.
\newblock {\em J. Comput. Appl. Math.}, 232(2):427--448, 2009.

\bibitem[May72]{may-72}
R.~M. May.
\newblock Will a large complex system be stable?
\newblock {\em Nature}, 238(5364):413--414, 1972.

\bibitem[MR16]{mon-ric-16}
A.~Montanari and E.~Richard.
\newblock Non-negative principal component analysis: message passing algorithms
  and sharp asymptotics.
\newblock {\em IEEE Trans. Inform. Theory}, 62(3):1458--1484, 2016.

\bibitem[OD92]{opp-die-92}
M.~Opper and S.~Diederich.
\newblock Phase transition and $1/f$ noise in a game dynamical model.
\newblock {\em Phys. Rev. Lett.}, 69:1616--1619, Sep 1992.

\bibitem[PKK23]{pak-ko-krz-(arxiv)23}
A.~Pak, J.~Ko, and F.~Krzakala.
\newblock Optimal algorithms for the inhomogeneous spiked {W}igner model.
\newblock {\em arXiv preprint arXiv:2302.06665}, 2023.

\bibitem[PS11]{pas-livre}
L.~A. Pastur and M.~Shcherbina.
\newblock {\em Eigenvalue Distribution of Large Random Matrices}, volume 171 of
  {\em Mathematical Surveys and Monographs}.
\newblock American Mathematical Society, Providence, RI, 2011.

\bibitem[Ran11]{ran-11}
S.~Rangan.
\newblock Generalized approximate message passing for estimation with random
  linear mixing.
\newblock In {\em 2011 IEEE International Symposium on Information Theory
  Proceedings}, pages 2168--2172. IEEE, 2011.

\bibitem[Sod10]{sod-10}
S.~Sodin.
\newblock The spectral edge of some random band matrices.
\newblock {\em Ann. of Math. (2)}, 172(3):2223--2251, 2010.

\bibitem[Tak96]{tak-livre96}
Y.~Takeuchi.
\newblock {\em Global dynamical properties of {L}otka-{V}olterra systems}.
\newblock World Scientific Publishing Co., Inc., River Edge, NJ, 1996.

\bibitem[Tok04]{tok-04}
K.~Tokita.
\newblock Species abundance patterns in complex evolutionary dynamics.
\newblock {\em Phys. Rev. Lett.}, 93:178102, Oct 2004.

\bibitem[Vil09]{vil-livre09}
C.~Villani.
\newblock {\em Optimal transport}, volume 338 of {\em Grundlehren der
  mathematischen Wissenschaften [Fundamental Principles of Mathematical
  Sciences]}.
\newblock Springer-Verlag, Berlin, 2009.
\newblock Old and new.

\bibitem[WZF22]{wan-zho-fan-(arxiv)22}
T.~Wang, X.~Zhong, and Z.~Fan.
\newblock Universality of approximate message passing algorithms and tensor
  networks.
\newblock {\em arXiv preprint arXiv:2206.13037}, 2022.

\end{thebibliography}

\newcommand{\etalchar}[1]{$^{#1}$}
\def\cprime{$'$} \def\cdprime{$''$} \def\cprime{$'$} \def\cprime{$'$}
  \def\cprime{$'$} \def\cprime{$'$}

\end{document}